\documentclass[a4paper]{scrartcl}
\usepackage[margin=1in]{geometry}

\usepackage[utf8]{inputenc}
\usepackage[T1]{fontenc}
\usepackage{lmodern}
\usepackage{microtype}

\usepackage[dvipsnames]{xcolor}
\usepackage{subcaption}
\usepackage{tikz}

\usepackage{amsmath}
\usepackage{amssymb}
\usepackage{amsthm}
\newtheorem{theorem}{Theorem}
\newtheorem{lemma}[theorem]{Lemma}
\newtheorem{corollary}[theorem]{Corollary}
\theoremstyle{definition}
\newtheorem{remark}[theorem]{Remark}
\newtheorem{example}[theorem]{Example}

\usepackage[ruled,vlined,linesnumbered]{algorithm2e}
\SetKwInOut{Input}{Input}
\SetKwInOut{Output}{Output}
\DontPrintSemicolon

\usepackage{enumitem}
\setlist{noitemsep,topsep=0pt}

\usepackage[backend=biber,style=alphabetic]{biblatex}
\usepackage{breakurl}
\usepackage[colorlinks]{hyperref}
\hypersetup{
  linkcolor=BrickRed,
  citecolor=Green,
  urlcolor=NavyBlue,
  menucolor=BrickRed
}

\newcommand{\N}{\mathbb{N}}
\newcommand{\Nn}{\N_0}
\newcommand{\Q}{\mathbb{Q}}
\newcommand{\Qge}{\Q_{\ge 0}}
\newcommand{\Qgt}{\Q_{> 0}}
\newcommand{\R}{\mathbb{R}}
\newcommand{\Sigp}[1]{$\Sigma^{\text{p}}_{#1}$}

\newcommand{\VCP}{\textsc{Vertex Cover Problem}}
\newcommand{\SP}{\textsc{Selection Problem}}
\newcommand{\CSP}{\textsc{Continuous Selection Problem}}
\newcommand{\BSP}{\textsc{Bilevel Selection Problem}}
\newcommand{\CBSP}{\textsc{Continuous Bilevel Selection Problem}}
\newcommand{\RBSP}{\textsc{Robust Bilevel Selection Problem}}
\newcommand{\RCBSP}{\textsc{Robust Continuous Bilevel Selection Problem}}
\newcommand{\KP}{\textsc{Knapsack Problem}}
\newcommand{\CKP}{\textsc{Continuous Knapsack Problem}}
\newcommand{\BKP}{\textsc{Bilevel Knapsack Problem}}
\newcommand{\BCKP}{\textsc{Bilevel Continuous Knapsack Problem}}
\newcommand{\RBCKP}{\textsc{Robust Bilevel Continuous Knapsack Problem}}

\newcommand{\U}{\mathcal{U}}
\newcommand{\E}{\mathcal{E}}
\newcommand{\Ell}{{\mathcal{E}_l}}
\newcommand{\Eff}{{\mathcal{E}_f}}
\newcommand{\nll}{{n_l}}
\newcommand{\nllb}{{n_l^b}}
\newcommand{\nff}{{n_f}}
\newcommand{\Effn}{\mathcal{E}_f^0}
\newcommand{\bll}{{b_l}}
\newcommand{\bllm}{{b_l^-}}
\newcommand{\bllp}{{b_l^+}}
\newcommand{\blls}{{b_l^*}}
\newcommand{\bff}{{b_f}}
\newcommand{\bffm}{{b_f^-}}
\newcommand{\bffp}{{b_f^+}}
\newcommand{\bffs}{{b_f^*}}
\newcommand{\pll}{{p_l}}
\newcommand{\pff}{{p_f}}
\newcommand{\pffd}{{p_f^d}}
\newcommand{\bm}{{b^-}}
\newcommand{\bp}{{b^+}}
\newcommand{\PLFl}{\textsc{PLF}}
\newcommand{\PLFf}{\overline{\textsc{PLF}}}
\newcommand{\gll}{{g_l}}
\newcommand{\gff}{{g_f}}
\newcommand{\gffole}{{g_f^{\overline{e}}}}
\newcommand{\gffbffs}{g_f^{\bffs}}
\newcommand{\gffbffsesdeltas}{g_f^{\bffs, e^*, \delta^*}}

\DeclareMathOperator{\st}{\,s.t.}
\DeclareMathOperator*{\argmin}{argmin}
\DeclareMathOperator*{\argmax}{argmax}
\DeclareMathOperator{\conv}{conv}

\begin{document}

\title{The Robust Bilevel Selection Problem\thanks{The author thanks Kristina Hartmann for her work on several results in this paper in the context of her Master's thesis~\cite{hartmann2021}. In particular, the results in Section~\ref{sec_leader_general_approx} include original ideas of hers and were jointly developed by her and the author of this article, who supervised the thesis. Thanks also to Christoph Buchheim for helpful discussions and for providing feedback on preliminary versions of this article. The author has partially been supported by Deutsche Forschungsgemeinschaft (DFG) under grant no. BU 2313/6.}}

\author{Dorothee Henke\thanks{University of Passau, School of Business, Economics and Information Systems, Chair of Business Decisions and Data Science, 94030 Passau, Germany. \texttt{dorothee.henke@uni-passau.de}}}

\date{}

\maketitle

\begin{abstract} 
  In bilevel optimization problems,
  two players, the leader and the follower,
  make their decisions in a hierarchy,
  and both decisions may influence each other.
  Usually one assumes that
  both players have full knowledge
  also of the other player's data.
  In a more realistic model,
  uncertainty can be quantified,
  e.g., using the robust optimization approach:
  We assume that the leader does not know the follower's objective function precisely,
  but only knows an uncertainty set of potential follower's objectives,
  and her aim is to optimize the worst case of the corresponding scenarios.
  Now the question arises
  how the computational complexity
  of bilevel optimization problems
  changes under the additional complications
  of this type of uncertainty.

  We make a further step towards answering this question
  by examining an easy bilevel problem.
  In the \BSP{},
  the leader and the follower each select some items
  from their own item set,
  while a common number of items to select in total is given,
  and each of the two players minimizes
  the total costs of the selected items,
  according to different sets of item costs.
  We show that this problem can be solved in polynomial time
  without uncertainty
  and then investigate the complexity of its robust version.
  If the leader's item set and the follower's item set are disjoint,
  it can still be solved in polynomial time
  in case of discrete uncertainty, interval uncertainty, and discrete uncorrelated uncertainty,
  using ideas from~\cite{buchheimhenke2022}.
  Otherwise,
  we show that the \RBSP{} becomes NP-hard,
  even for discrete uncertainty.
  We present a 2-approximation algorithm
  and exact exponential-time approaches for this setting,
  including an algorithm that runs in polynomial time
  if the number of scenarios is a constant.

  Furthermore,
  we investigate variants of the \BSP{}
  where one or both of the two decision makers
  take a continuous decision.
  One variant leads to an example of a bilevel optimization problem
  whose optimal value may not be attained.
  For the \RCBSP{},
  where all variables are continuous,
  we generalize results from~\cite{buchheimhenke2022}
  and also develop a new approach
  for the setting of discrete uncorrelated uncertainty,
  which gives a polynomial-time algorithm for the \RCBSP{}
  and a pseudopolynomial-time algorithm for the \RBCKP{},
  answering an open question from~\cite{buchheimhenke2022}.

\end{abstract}

\section{Introduction} \label{sec_introduction}

In classical optimization problems,
there is one decision maker
whose aim it is to select a feasible solution
that is optimal with respect to some objective function.
\emph{Bilevel optimization} considers an important extension of this concept
in which there are two decision makers,
called the \emph{leader} and the \emph{follower}.%
\footnote{We will always refer to the leader using ``she/her'' and to the follower using ``he/him/his''.}
Each of them has their own optimization problem
(i.e., decision variables, constraints, and objective function),
and they select their solutions one after the other,
in a hierarchical order:
first the leader, then the follower.
The key feature is that
both decisions may also have an impact on the
other decision maker's optimization problem.
This interdependence
makes bilevel optimization problems very difficult in general.

Bilevel optimization has received a lot of interest lately.
A recent overview of more than 1000 references
on the topic
can be found in~\cite{dempe2020}.
For general introductions to bilevel optimization,
we refer to, e.g.,~\cite{colsonmarcottesavard2007,
  dempeetal2015,
  beckschmidt2021}.

Our focus in this article
lies on the perspective of combinatorial optimization
and computational complexity.
The first complexity results for bilevel optimization problems
were obtained by Jeroslow in~\cite{jeroslow1985}.
His results imply that bilevel linear programs with continuous variables are NP-hard
and that bilevel integer linear programs are even \Sigp{2}-hard.
See also \cite{marcottesavard2005}
for an overview of hardness results
for bilevel linear programs,
and \cite{woeginger2021}
for illustrations and intuitions regarding the class~\Sigp{2}.

In this article,
we will deal with the computational complexity
of a specific bilevel optimization problem:
The \BSP{} is solvable in polynomial time,
which, in light of the known complexity results,
is unusual for bilevel optimization problems.
Since the single-level \SP{} is solvable in polynomial time,
it is not surprising that
the \BSP{} is not on the second level of the polynomial hierarchy.
However,
it is not obvious whether it is polynomial-time solvable or NP-hard
in the general setting.
We show in Section~\ref{sec_BSP_alg} that
it is indeed polynomial-time solvable.
  
Moreoever,
precisely because the \BSP{} is relatively easy,
we consider it to be a useful starting point
for studying the question that motivates this work:
We aim for a better understanding of the additional complexity
induced by uncertainty in bilevel optimization problems.

An important concept to model uncertainties
concerning some parameters of an optimization problem
is \emph{robust optimization}.
It works with \emph{uncertainty sets}
containing all scenarios
that are considered to be possible or reasonably likely,
and aims at optimizing the worst-case outcome.
For an introduction to the field of robust optimization
and an overview of important literature,
we refer to~\cite{buchheimkurtz2018} and the references therein,
as well as to the books~\cite{kouvelisyu1997,bentalelghaouinemirovski2009}.

Similarly to bilevel optimization,
also a robust optimization problem can be understood as
the interplay of two decision makers:
First,
the main decision maker selects a solution,
and second,
the \emph{adversary} selects a scenario
that results in the worst-case objective value.
In fact,
bilevel optimization and robust optimization are closely related;
see~\cite{goerigketal2023}.

The properties and the complexity
of a robust optimization problem
highly depend on the structure of the uncertainty set.
The following three common types of uncertainty sets
will be studied in this article.

\begin{itemize}
\item A \emph{discrete uncertainty set}
consists of a finite number of scenarios
that are explicitly given as part of the input
of the robust optimization problem.
Typically,
the robust counterpart subject to discrete uncertainty
is significantly harder than
the underlying certain optimization problem.
For many combinatorial problems,
the robust counterpart is strongly NP-hard in general
and weakly NP-hard if the number of scenarios is fixed; see~\cite{buchheimkurtz2018}.
This is in contrast to the \RBSP{}
and also other robust bilevel optimization problems,
where discrete uncertainty seems to introduce less additional complexity
than other types of uncertainty,
compared to the underlying bilevel problem;
see~\cite{buchheimhenkehommelsheim2021,buchheimhenke2022}.

\item In an \emph{interval uncertainty set},
every parameter may independently vary within an interval.
Usually,
interval uncertainty is not very interesting to study
because it can easily be eliminated
and therefore does not introduce any additional complexity
compared to the underlying certain problem.
However,
in the robust bilevel setting,
it turns out to be more involved;
see also~\cite{buchheimhenkehommelsheim2021,buchheimhenke2022}.

\item Another discrete type of uncertainty sets,
which we call \emph{discrete uncorrelated uncertainty},
shares the idea of interval uncertainty that
each parameter can vary independently in some set,
which now is a finite set instead of an interval.
Due to the form in which the set is given,
the scenarios cannot be enumerated efficiently here,
in contrast to discrete uncertainty.
However,
like interval uncertainty,
discrete uncorrelated uncertainty can often be eliminated
in a trivial way.
Moreover,
an uncertainty set can usually be replaced by its convex hull,
which, in particular, makes interval uncertainty and discrete uncorrelated uncertainty equivalent.
In the robust bilevel setting however,
this is not the case in general;
see~\cite{buchheimhenkehommelsheim2021,buchheimhenke2022}.
But, in fact,
for the \RBSP{},
discrete uncorrelated uncertainty turns out to be equivalent to interval uncertainty,
for problem-specific reasons
(Theorem~\ref{thm_RBSP_discrete_uncorr_equiv}).
\end{itemize}

The combination of bilevel optimization and uncertainty
is a topic that receives increasing attention currently.
In a practical setting
in which leader and follower
are actually two distinct decision makers,
it seems to be very natural
to adopt some concept of optimization under uncertainty
and assume that, e.g.,
the leader is uncertain about some parameters of the follower's problem.
A very recent and comprehensive overview of the area of bilevel optimization under uncertainty
is given by Beck, Ljubić, and Schmidt in the survey~\cite{beckljubicschmidt2023}.
It covers both stochastic and robust approaches to model uncertainty,
and it classifies various ways of
where and how uncertainty can come into play in a bilevel optimization problem.
In this article,
we focus on uncertainty in the parameters of the follower's objective function
from the leader's perspective.
The situation can be illustrated
by considering three decision makers:
The \emph{leader} first chooses a solution,
then an \emph{adversary} (of the leader) chooses an objective function
for the follower,
and the \emph{follower} finally computes an optimal solution
according to this objective function
and depending on the leader's choice.
The leader's aim is to optimize her own objective value,
which depends on the follower's choice,
while the adversary has the opposite objective.
Note that \emph{defender-attacker-defender games} present
a similar three-level structure as these robust bilevel optimization problems; see, e.g., \cite{brownetal2006}.
However, there is only one objective function in these problems,
as the same decision maker acts on the first and on the third level.
Moreover, the attack usually does not affect the objective function,
but a combinatorial structure concerning the constraints.

A main goal of this article
is to improve the understanding of
robust bilevel optimization problems
and, in particular,
of the additional complexity
that this way of respecting uncertainty introduces in bilevel optimization problems.
For a more general class of
robust bilevel optimization problems,
this question has been studied in~\cite{buchheimhenkehommelsheim2021}:
It has been shown that
interval uncertainty can turn an NP-equivalent bilevel optimization problem
into a \Sigp{2}-hard robust problem,
while a robust bilevel optimization problem with discrete uncertainty
is at most one level higher in the polynomial hierarchy
compared to the follower's problem.
In this article,
our underlying bilevel optimization problem
is solvable in polynomial time;
therefore,
we can expect to obtain robust problems
that are either polynomial-time solvable
or NP-equivalent.

As a specific robust bilevel optimization problem to investigate,
we develop a natural problem formulation that is based on the single-level \SP{},
which can be seen as one of the most simple combinatorial optimization problems
and which is often studied in the field of robust optimization
in the context of complexity questions.
For robust versions of the \SP{},
see, e.g.,~\cite{kasperskizielinski2009} and \cite[Chapter~3.3]{kouvelisyu1997}.

A related robust bilevel optimization problem,
namely the \RBCKP{},
together with its complexity for different types of uncertainty sets,
has been studied in~\cite{buchheimhenke2022}.
A part of the current work
generalizes and extends the results of~\cite{buchheimhenke2022},
with a stronger emphasis on combinatorial underlying problems.
The (\textsc{Continuous}) \BSP{} with disjoint item sets
is more general than the \BCKP{} in one aspect,
but structurally simpler than it in another;
see Section~\ref{sec_BSP_var} for more details.
For most of the robust settings we consider,
we obtain similar results as in~\cite{buchheimhenke2022},
but we will also see some differences.
In particular,
solving the \RCBSP{}
under discrete uncorrelated uncertainty
requires some new ideas,
which also lead to a pseudopolynomial-time algorithm
for the \RBCKP{};
see Section~\ref{sec_RCBSP_discrete_uncorr} and Appendix \ref{sec_RCBSP_discrete_uncorr_knapsack}.
The latter result answers an open question stated in~\cite{buchheimhenke2022}.

In addition to the results mentioned above,
which concern the \BSP{} with disjoint item sets
that is closely related to~\cite{buchheimhenke2022},
we also study a more general version of the \BSP{}
where some items might be controlled by both decision makers.
We show that it can still be solved in polynomial time
(Theorem~\ref{thm_BSP_alg}),
but its robust version becomes NP-hard in case of discrete uncertainty
(Theorem~\ref{thm_RBSP_hardness}).
However,
we give a 2-approximation algorithm for this case
(see Section~\ref{sec_leader_general_approx})
and also show how the problem can be solved in exponential time
(see Section~\ref{sec_leader_general_exact}).
One of our approaches implies that,
if the number of scenarios in a discrete uncertainty set
is a constant,
the \RBSP{} can be solved in polynomial time
(Theorem~\ref{thm_RBSP_discrete_exact_u}).

Besides the \BCKP{},
which has been studied under robust uncertainty in~\cite{buchheimhenke2022}
and under stochastic uncertainty in~\cite{buchheimhenkeirmai2022},
also other bilevel combinatorial optimization problems have been investigated in the literature,
although not in the uncertain setting we consider here.
The complexity of several bilevel variants of the \KP{}
has been studied in~\cite{capraraetal2014}.
One of them,
which has been introduced in~\cite{mansietal2012},
can be seen as a generalization of our \BSP{} in case of disjoint item sets;
see also Section~\ref{sec_BSP_var_sets}.
The results in~\cite{capraraetal2014,mansietal2012} imply that
it is \Sigp{2}-hard and solvable in pseudopolynomial time.
Other bilevel combinatorial optimization problems
whose complexity has been studied
are, e.g., the bilevel minimum spanning tree problem~\cite{shizengprokopyev2019,buchheimhenkehommelsheim2022}
and the bilevel assignment problem~\cite{gassnerklinz2009,fischermulukwoeginger2022}.
Note that
the bilevel minimum spanning tree problem
and the \BSP{} can both be seen as special cases of a bilevel minimum matroid basis problem.

This article is organized as follows.
In Section~\ref{sec_BSP},
we introduce the \BSP{} and discuss several variants of the problem,
some of which are distinguished later on.
We show how the \BSP{} can be solved in polynomial time
in Section~\ref{sec_BSP_alg}.
The robust problem version is introduced in Section~\ref{sec_RBSP}.
We then present algorithms for the adversary's problem,
for different types of uncertainty sets,
in Section~\ref{sec_adversary},
before turning to the robust leader's problem in Section~\ref{sec_leader}.
Here we give polynomial-time algorithms
for the simpler case of disjoint item sets in Section~\ref{sec_leader_disj}
and show our hardness and algorithmic results for the general case in Section~\ref{sec_leader_general}.
Moreover,
we give some insights regarding problem variants with continuous variables
in Section~\ref{sec_cont}.
This includes an example of a bilevel optimization problem
whose optimal value may not be attained
(Example~\ref{ex_cont_bin_no_opt})
and some structural properties in Section~\ref{sec_CBSP_PLF}
and in the beginning of Section~\ref{sec_RCBSP}.
We then study the \RCBSP{} for different types of uncertainty sets.
Section~\ref{sec_conclusion} concludes.
In Appendix~\ref{sec_RCBSP_discrete_uncorr_knapsack},
a new approach developed in Section~\ref{sec_RCBSP} is finally extended to a pseudopolynomial-time algorithm
for the \RBCKP{} with discrete uncorrelated uncertainty.

Many of the results of Sections~\ref{sec_BSP} to~\ref{sec_leader}
have already appeared in the unpublished Master's thesis~\cite{hartmann2021}.

\section{The Bilevel Selection Problem Without Uncertainty} \label{sec_BSP}

In this section,
we introduce the \BSP{},
which is an example of a combinatorial bilevel optimization problem
that can be solved in polynomial time.
It will serve as a basis for the robust problem versions
that we will investigate in the remainder of this article.
We define the \BSP{} in Section~\ref{sec_BSP_def}
and discuss several variants of it,
together with their relation to bilevel variants of the \KP{},
in Section~\ref{sec_BSP_var},
before presenting a polynomial-time algorithm
in Section~\ref{sec_BSP_alg}.

\subsection{Problem Definition} \label{sec_BSP_def}

We first define the (single-level) \SP{} as follows:
Given a finite set~$\E$ of items,
a number $b \in \{0, \dots, \lvert \E \rvert\}$,
and item costs $c \colon \E \to \Q$,
find a subset~$X \subseteq \E$
such that~$\lvert X \rvert = b$
and $c(X) := \sum_{e \in X} c(e)$ is minimal.
This problem is easy to solve in polynomial time
because it is clearly optimal to select
$b$~items that have the smallest values of~$c$,
for example, by
sorting the items in~$\E$ by increasing values of~$c$
and then selecting the first $b$~items of this order.
This algorithm has running time~$O(n \log n)$,
where~$n = \lvert \E \rvert$.
However,
explicitly sorting the items is not even necessary:
It suffices to find the~$b$-th best item~$e \in \E$ with respect to~$c$,
which amounts to computing a weighted median,
and then select the set of all items with costs at most~$c(e)$.
This can be done in linear time~$O(n)$~\cite{blumetal1973};
see also~\cite[Chapter~17.1]{kortevygen2018}.

In our bilevel version of the \SP{},
we now consider two players,
the leader and the follower,
who select items from different item sets
and according to different item costs,
but subject to a common capacity constraint.
More specifically,
we are given finite leader's and follower's item sets~$\Ell$ and~$\Eff$, respectively,
a number $b \in \{0, \dots, \lvert \Ell \cup \Eff \rvert\}$,
and leader's and follower's item costs $c \colon \Ell \cup \Eff \to \Q$ and $d \colon \Eff \to \Q$, respectively.
The task is to find a subset~$X \subseteq \Ell$ for the leader such that $c(X \cup Y)$ is minimal,
where~$Y \subseteq \Eff \setminus X$ is a follower's subset
that satisfies $\lvert X \cup Y \rvert = b$ and minimizes $d(Y)$.

Note that
we do not require the sets~$\Ell$ and~$\Eff$ to be disjoint in general.
However,
assuming the sets to be disjoint
leads to an important special case,
actually simplifying the problem in some sense.
We denote the cardinalities of the item sets by
$\nll = \lvert \Ell \rvert$, $\nff = \lvert \Eff \rvert$,
and $n = \lvert \Ell \cup \Eff \rvert$
throughout this article.

In the typical form of a b ilevel optimization problem,
the \BSP{} can be written as follows:
\begin{equation} \label{eq_BSP} \tag{BSP}
\begin{aligned}
  \min_X~ & c(X \cup Y)\\
  \st~ & X \subseteq \Ell \\
  & Y \in
  \begin{aligned}[t]
    \argmin_{Y'}~ & d(Y')\\
    \st~ & Y' \subseteq \Eff \setminus X \\
    & \lvert X \cup Y' \rvert = b
  \end{aligned}
\end{aligned}
\end{equation}

The leader can be imagined
to make her choice~$X \subseteq \Ell$ before the follower and,
at the same time,
anticipate the follower's optimal solution~$Y \subseteq \Eff \setminus X$,
while optimizing her objective function~$c$ on all items selected by the leader or the follower.
Leader's subsets~$X \subseteq \Ell$ that do not allow for
a feasible follower's response,
i.e., for some~$Y' \subseteq \Eff \setminus X$ such that $\lvert X \cup Y' \rvert = b$,
are considered infeasible for the leader's problem.
However,
the requirement~$b \in \{0, \dots, \lvert \Ell \cup \Eff \rvert\}$ ensures that
a feasible (leader's and corresponding follower's) solution always exists.

From the follower's perspective,
a feasible leader's solution~$X \subseteq \Ell$ is already fixed and
the task is to select a subset~$Y'$ of $b - \lvert X \rvert$~items from the set~$\Eff \setminus X$,
while minimizing the total costs~$d(Y')$ of the selected items.
This amounts to a standard single-level \SP{}
and can thus be solved in a greedy way by sorting the follower's items by their costs~$d$,
or by computing a weighted median;
see above.

Note that
the follower's objective function considers only the items selected by the follower,
while the leader's costs of all items selected by either player appear in the leader's objective.
This is also due to the fact that
the leader's solution~$X$ is fixed from the follower's perspective.
If there were additional follower's item costs for the items in~$X$,
they would only amount to a constant term in his objective
and would therefore not change the problem.

\begin{remark} \label{rem_opt_pess}
  The above definition of the \BSP{} and also the formulation in~\eqref{eq_BSP}
  are imprecise if the optimal follower's solution is not unique.
  As usual in bilevel optimization,
  the optimistic and the pessimistic setting can be distinguished,
  i.e., the follower can be assumed to decide either in favor or to the disadvantage of the leader
  among his optimal solutions.
  Here, this corresponds to
  the follower minimizing or maximizing~$c(Y')$ as a secondary objective, respectively.
  For the follower's greedy algorithm
  this means that
  $c$ is used as a secondary criterion
  when non-uniqueness arises in
  sorting the follower's items with respect to~$d$.
  More precisely,
  the required order can be defined such that
  an item~$e_1 \in \Eff$ precedes an item~$e_2 \in \Eff$
  if either $d(e_1) < d(e_2)$ or $d(e_1) = d(e_2)$ and $c(e_1) < c(e_2)$ [$c(e_1) > c(e_2)$]
  in the optimistic [pessimistic] setting.

  This already indicates that,
  for the \BSP{},
  there is not a significant difference between the two settings
  from an algorithmic perspective.
  In fact,
  as soon as there is any fixed order
  according to which the follower selects his items greedily,
  the problem is well-defined and
  the follower's optimal solution can always be assumed to be unique.
  Therefore,
  we will often implicitly assume
  that a fixed follower's greedy order is given,
  such that it will usually not be required
  to discuss the optimistic and the pessimistic setting
  in detail.
  
  In fact,
  from the leader's perspective,
  only this order of the follower's items is relevant,
  but not the precise follower's item costs.
  Accordingly,  
  when constructing an instance,
  we will sometimes only define a follower's greedy order
  instead of a function~$d$;
  see, e.g., Theorem~\ref{thm_RBSP_hardness} and Examples~\ref{ex_approx_discrete} and~\ref{ex_approx_interval}.
  Observe that,
  given a follower's greedy order,
  an appropriate function~$d$ (with polynomial-size values) can easily be constructed in polynomial time,
  e.g., by setting its values to subsequent natural numbers in this order.
\end{remark}

\subsection{Problem Variants and Relation to Knapsack Problems} \label{sec_BSP_var}

In this section,
we introduce and discuss several assumptions and modifications
that can be made for the \BSP{}
as we defined it in Section~\ref{sec_BSP_def}.
This provides context on related problems in the literature
and motivates which problem variants will be distinguished in the rest of this article.

\subsubsection{Item Sets} \label{sec_BSP_var_sets}

As mentioned in Section~\ref{sec_BSP_def},
we will mainly distinguish the special case of the \BSP{}
where $\Ell \cap \Eff = \emptyset$
from the general one
(see Section~\ref{sec_leader}).
This special case of the \BSP{}
can be considered to be a special case of the \BKP{}
that has been studied in \cite{mansietal2012,capraraetal2014}.
In fact,
the polynomial-time algorithm
presented for the \BSP{}
in Section~\ref{sec_BSP_alg},
in case of disjoint item sets,
can be seen as a special case of the pseudopolynomial-time approaches
in~\cite{mansietal2012,capraraetal2014}.
However,
our algorithm also works for the general \BSP{},
where the situation is more involved.

Moreover,
one could think of a simpler bilevel variant of the \SP{}
in which only the follower selects items in the sense of a single-level \SP{},
while the leader just decides on the number~$b$ of items to select
and gets some value (which might be positive or negative) from each selected item.
This problem variant can be seen as a special case
of the \BKP{}
in which the leader provides the capacity for the follower's knapsack problem.
This problem
has been studied in \cite{demperichter2000,capraraetal2014},
and the continuous version of it is the underlying \BCKP{} in~\cite{buchheimhenke2022,buchheimhenkeirmai2022}.
In fact,
this bilevel selection problem
is equivalent to a special case of the \BSP{} as defined above,
with disjoint item sets:
Consider $\Ell$ and $\Eff$ to be disjoint
and the leader's costs of all leader's items to be equal to zero.
Then, by selecting some subset $X \subseteq \Ell$ of her items,
the leader just determines the capacity $b - \lvert X \rvert$
for the follower
and pays her item costs of all items selected by the follower.
Therefore,
the algorithmic results we obtain in Sections~\ref{sec_BSP_alg}, \ref{sec_leader_disj}, and~\ref{sec_cont}
can also be applied to this bilevel selection problem,
and in fact are similar to the ones in~\cite{buchheimhenke2022}.
In Appendix~\ref{sec_RCBSP_discrete_uncorr_knapsack},
we extend our new algorithms
for the \RCBSP{}
with discrete uncorrelated uncertainty
to the \RBCKP{}.

It is also worth mentioning another special case
regarding the two item sets,
namely the one where $\Ell \subseteq \Eff$.
This assumption can be seen as reasonable and convenient
because it ensures that
there is always a feasible follower's solution,
no matter which (at most~$b$) items the leader has selected.
In the general setting,
the leader might have to be careful about
selecting enough items from~$\Ell \setminus \Eff$
in order to obtain a feasible solution.
However,
assuming $\Ell \subseteq \Eff$
is not a restriction
because an instance of the general \BSP{}
can be transformed into an equivalent instance satisfying this assumption.
For this,
add all items in $\Ell \setminus \Eff$
to $\Eff$ and set their follower's item costs
to some constant larger than all other follower's item costs.
Then the follower will not select any of these items
unless the leader chooses a solution that is infeasible in the original instance.

\subsubsection{Continuous Variables} \label{sec_BSP_var_cont}

Another detail that can be altered in the definition of the \BSP{},
in order to obtain other possibly interesting variants of the problem,
is the integrality of the leader's and the follower's decisions.
More precisely,
one could allow the two players (or only one of them) to select fractions of items.
This is related to the situation of the \BCKP{} in~\cite{buchheimhenke2022}.
We will discuss this setting in more detail in Section~\ref{sec_cont}.

\subsubsection{Maximization} \label{sec_BSP_var_max}

The \SP{} and also the \BSP{}
are defined as minimization problems according to Section~\ref{sec_BSP_def},
but could analogously be defined as maximization problems,
i.e., the decision makers could maximize the total value of their selection of items
instead of minimizing the total costs.
Then the problems can be more directly seen as special cases of the classical \KP{} and a bilevel variant of it;
see also Section~\ref{sec_BSP_var_sets}.

Although we will work with the minimization problem version in the following,
almost all of our arguments and approaches
are also valid for the maximization version of the problem.
Only the approximation procedure in Section~\ref{sec_leader_general_approx}
appears to be specific to the minimization problem
because it is restricted to nonnegative item costs;
see also Section~\ref{sec_BSP_var_costs}.

\subsubsection{Capacity Constraint} \label{sec_BSP_var_constr}

In relation to the classical \KP{},
it might be more natural to replace ``$= b$'' by ``$\le b$'' in the capacity constraint
(or by ``$\ge b$'' in the minimization version).
In the maximization version,
this means that 
the follower is not forced to use the provided capacity completely
if this would mean to select items having a negative value for him.
Accordingly,
in the minimization version,
the follower is allowed to select all items of negative cost,
even if there are more than~$b$ of them.
From the algorithmic perspective,
this is only a minor change:
In an optimal follower's solution,
no [all] items of negative value [cost] are selected
in the minimization [maximization] version,
and the greedy approach on all other items is still the best strategy.
This insight enables to model this problem variant
in terms of our original \BSP{},
e.g., by removing items for which the decision is clear or adding dummy items,
and possibly adjusting the capacity.

We focus on the problem variant with an equality constraint in this article
because our arguments often become easier
when the number of items in a solution
is known in advance.

\subsubsection{Item Costs} \label{sec_BSP_var_costs}

On a related note,
one could also investigate
how the \BSP{} changes
when only allowing for nonnegative leader's and/or follower's item costs
(or nonnegative item values in the maximization version),
in contrast to arbitrary rational values
as introduced so far.
In fact,
the setting of arbitrary item costs can be easily reduced to any of the restricted settings
by adding appropriate constants~$C \in \Qgt$ and~$D \in \Qgt$
to all leader's and follower's item costs, respectively.
The resulting problem is equivalent because,
by the capacity constraint~$\lvert X \cup Y' \rvert = b$,
this changes both the leader's and the follower's objective function value
of any given solution
(in the latter case assuming that some fixed leader's solution~$X$ is given)
only by a constant.
Note that
the distinction between equality and inequality constraints
from Section~\ref{sec_BSP_var_constr}
is not necessary
if all costs are nonnegative
because we may assume that exactly $b$~items are selected anyway.

In Section~\ref{sec_leader_general_approx},
we will deal with approximation algorithms
and, for this purpose,
assume all leader's item costs to be nonnegative.
In terms of approximation results,
this is a restriction
because the above construction
does not preserve any approximation guarantees.

\subsection{A Polynomial-Time Algorithm} \label{sec_BSP_alg}

In order to develop a polynomial-time algorithm for the \BSP{},
we first focus on the follower's problem.
As already noted in Section~\ref{sec_BSP_def},
the follower solves a single-level \SP{},
which can be done in linear time
by selecting the best follower's items
according to his (and possibly the leader's, in case of non-uniqueness) item costs.
From the follower's perspective,
there is no difference among the problem variants concerning the item sets
(see Section~\ref{sec_BSP_var_sets})
because
the leader's solution~$X$ is already fixed
and the item set the follower chooses from is given by~$\Eff \setminus X$,
regardless of the relation between~$\Ell$ and~$\Eff$.

We now turn to the leader's perspective
and first look at the special case with disjoint item sets.
In this case,
it is easy to see that
the only property of the leader's solution
that is relevant for the follower
is the number~$\lvert X \rvert$ of items she selects.
Therefore,
for any fixed partition of the total capacity~$b$
into a leader's capacity~$\bll$ and a follower's capacity~$\bff$,
each of the two players will select their number of items from their set
using a greedy strategy.
The optimal partition of the capacity
can be determined by enumerating all possible partitions.
This leads to Algorithm~\ref{alg_BSP}.

\begin{algorithm}
  \caption{Algorithm for the \BSP{} with $\Ell \cap \Eff = \emptyset$ or $\Ell \subseteq \Eff$}
  \label{alg_BSP}

  \Input{finite sets $\Ell$ and $\Eff$ with $\Ell \cap \Eff = \emptyset$ or $\Ell \subseteq \Eff$, $\nll = \lvert \Ell \rvert$, $\nff = \lvert \Eff \rvert$, and $n = \lvert \Ell \cup \Eff \rvert$, $b \in \{0, \dots, n\}$, $c \colon \Ell \cup \Eff \to \Q$, $d \colon \Eff \to \Q$}

  \Output{an optimal solution $(X, Y)$ of the \BSP{}}

  compute a bijection $\pll \colon \{1, \dots, \nll\} \to \Ell$ such that $c(\pll(1)) \le \dots \le c(\pll(\nll))$ \label{alg_BSP_leader_bijection}

  compute a bijection $\pff \colon \{1, \dots, \nff\} \to \Eff$ such that $d(\pff(1)) \le \dots \le d(\pff(\nff))$ and $c(\pff(i)) \le c(\pff(i + 1))$ [or $c(\pff(i)) \ge c(\pff(i + 1))$] for all $i \in \{1, \dots, \nff - 1\}$ with $d(\pff(i)) = d(\pff(i + 1))$ in the optimistic [pessimistic] setting \label{alg_BSP_follower_bijection}

  $\bllm := \max\{0, b - \nff\}$ \label{alg_BSP_bl_lb}

  $\bllp := \min\{b, \nll\}$ \label{alg_BSP_bl_ub}

  \For{$\bll = \bllm, \dots, \bllp$ \label{alg_BSP_loop}}
  {
    $X_{\bll} := \{\pll(1), \dots, \pll(\bll)\}$ \label{alg_BSP_leader_sol}

    $\bff := b - \bll$ \label{alg_BSP_bf}

    $Y_{\bll} := \{\pff(1), \dots, \pff(m)\} \setminus X_{\bll}$, where $m \in \{0, \dots, \nff\}$ is chosen such that $\lvert Y_{\bll} \rvert = \bff$ \label{alg_BSP_follower_sol}
  }

  \Return{$\argmax\{c(X \cup Y) \mid (X, Y) \in \{(X_{\bllm}, Y_{\bllm}), \dots, (X_{\bllp}, Y_{\bllp})\}\}$} \label{alg_BSP_return}
\end{algorithm}

\begin{theorem} \label{thm_BSP_disj_alg}
  Algorithm~\ref{alg_BSP} solves the \BSP{} in case of $\Ell \cap \Eff = \emptyset$
  in time~$O(n \log n)$.
\end{theorem}

\begin{proof}
  From the follower's perspective,
  only the number~$\lvert X \rvert$ of items selected by the leader is relevant,
  not the specific leader's solution~$X$,
  because the item set in the follower's \SP{} is always~$\Eff$
  in case of disjoint sets~$\Ell$ and~$\Eff$,
  and only the number~$b - \lvert X \rvert$ of items to select from~$\Eff$
  is influenced by the leader.
  Hence,
  for any fixed cardinality of the leader's solution,
  selecting the corresponding number of items greedily
  from her item set~$\Ell$
  is optimal for the leader.
  The algorithm now enumerates all feasible numbers~$\bll$ of leader's items to select,
  together with the corresponding greedy leader's and follower's solutions, respectively.
  Note that
  the order of the follower's items determined by~$\pff$
  correctly represents the optimistic and the pessimistic setting;
  see also Remark~\ref{rem_opt_pess}.
  Moreover,
  $\bllm$ and~$\bllp$ are defined
  such that exactly the numbers~$\bll$ are enumerated
  that lead to feasible solutions,
  given that exactly~$b$ items have to be selected in total.

  The running time of Algorithm~\ref{alg_BSP} is dominated by
  sorting the leader's and the follower's items
  in Lines~\ref{alg_BSP_leader_bijection} and~\ref{alg_BSP_follower_bijection},
  which can each be done in time~$O(n \log n)$.
  Indeed,
  the loop in Lines~\ref{alg_BSP_loop}~to~\ref{alg_BSP_follower_sol}
  can be implemented to
  run in linear time
  by updating~$X_{\bll}$ and~$Y_{\bll}$ in every iteration.
  Note that
  $\Ell$ and~$\Eff$,
  and hence also~$X_{\bll}$ and~$\Eff$,
  are disjoint in the current setting
  such that~$Y_{\bll}$ is simply chosen as $\{\pff(1), \dots, \pff(\bff)\}$
  in Line~\ref{alg_BSP_follower_sol}.
\end{proof}

\begin{remark} \label{rem_BSP_alg_time}
  As mentioned in Section~\ref{sec_BSP_def},
  a single-level \SP{} can also be solved in linear time,
  without the need to sort all items.
  Hence,
  a single iteration of the loop in Algorithm~\ref{alg_BSP}
  could be implemented to run in linear time,
  even without Lines~\ref{alg_BSP_leader_bijection}~and~\ref{alg_BSP_follower_bijection}.
  However,
  the overall algorithm becomes faster in general
  if we precompute the greedy orders,
  as this results in a running time of~$O(n \log n)$
  instead of~$O(n^2)$.
  The reason is that
  we can use the same orders for all iterations,
  which correspond to \SP{}s on the same item sets,
  but with different capacities.
\end{remark}

Next, we will show that,
also in the general case,
the \BSP{} can be solved in polynomial time.
In fact,
we can still apply Algorithm~\ref{alg_BSP},
although this is less obvious than in the case of disjoint item sets.
We prove this for the case of $\Ell \subseteq \Eff$,
which is equivalent to the general setting;
see Section~\ref{sec_BSP_var_sets}.

\begin{theorem} \label{thm_BSP_alg}
  Algorithm~\ref{alg_BSP} solves the \BSP{}
  in case of $\Ell \subseteq \Eff$
  in time~$O(n \log n)$.
\end{theorem}

\begin{proof}
  We first argue that
  the analysis of the running time from the proof of Theorem~\ref{thm_BSP_disj_alg}
  is also valid here.
  Indeed,
  it is still true that
  the loop in Lines~\ref{alg_BSP_loop}~to~\ref{alg_BSP_follower_sol}
  can be implemented to run in linear time
  by updating~$X_{\bll}$ and~$Y_{\bll}$ in every iteration.
  Here,
  the new follower's solution~$Y_{\bll}$ is obtained from~$Y_{\bll - 1}$ by
  either removing the item~$\pll(\bll)$ that was just added to the leader's solution if $\pll(\bll) \in Y_{\bll - 1}$,
  or otherwise removing the item~$\pff(m)$ from the previous iteration, i.e., decreasing~$m$ by one.
  
  It remains to show that
  Algorithm~\ref{alg_BSP} is still correct for the setting of $\Ell \subseteq \Eff$.
  Note that,
  if $\Ell \subseteq \Eff$,
  all sets~$X \subseteq \Ell$ with~$\lvert X \rvert \le b$
  are feasible leader's solutions
  because the follower is always able to select the remaining number of items,
  assuming $b \le n = \nff$.
  In particular,
  we always have~$\bllm = 0$.

  Let~$\pll$ be the bijection computed in Line~\ref{alg_BSP_leader_bijection} of Algorithm~\ref{alg_BSP}.
  Given some set~$X \subseteq \Ell$, we define
  \begin{align*}
    i_1(X) &:=
             \begin{cases}
               0 & \text{ if } X = \emptyset\\
               \max\{i \in \{1, \dots, \nll\} : \pll(i) \in X\} & \text{ otherwise}
             \end{cases} \\
    i_2(X) &:=
             \begin{cases}
               \nll + 1 & \text{ if } X = \Ell \\
               \min\{i \in \{1, \dots, \nll\} : \pll(i) \notin X\} & \text{ otherwise}
             \end{cases}
  \end{align*}
 and look at the difference~$D(X) := i_1(X) - i_2(X)$.
 We first observe that,
 for all leader's solutions~$X_{\bll}$ in Algorithm~\ref{alg_BSP},
 we have~$D(X_{\bll}) = \bll - (\bll + 1) = -1$,
 and for all feasible leader's solutions~$X \subseteq \Ell$
 that do not have this structure,
 $D(X) > 0$ holds.

 Let~$X$ be an optimal leader's solution with minimum value~$D(X)$.
 If~$D(X) = -1$,
 Algorithm~\ref{alg_BSP} inspects the solution~$X$
 and thus returns an optimal solution.
 Otherwise,
 we will construct another optimal solution~$\overline{X}$
 with~$D(\overline{X}) < D(X)$,
 contradicting the assumption that~$D(X)$ is minimal.
 Hence,
 we will assume $D(X) > 0$ from now on.

 Let~$e_1 := \pll(i_1(X)) \in X$ and~$e_2 := \pll(i_2(X)) \in \Ell \setminus X$,
 which is well-defined
 because~$\emptyset \subset X \subset \Ell$
 in the case where $X$ is feasible but not considered by the algorithm.
 As we assume~$D(X) > 0$,
 we have~$i_1(X) > i_2(X)$
 and hence $c(e_1) \ge c(e_2)$.
 We consider three operations that can be applied to~$X$:
 \begin{itemize}
 \item removing~$e_1$,
 \item adding~$e_2$, and
 \item both removing~$e_1$ and adding~$e_2$.
 \end{itemize}
 Note that
 all three operations reduce the value~$D(X)$:
 Removing~$e_1$ decreases~$i_1(X)$
 and does not change~$i_2(X)$ if $i_1(X) > i_2(X)$ was true before.
 Analogously,
 adding~$e_2$ increases~$i_2(X)$
 and does not change~$i_1(X)$.
 When performing both actions one after the other,
 $D(X)$ is decreased twice if $D(X) > 0$ still holds after the first one.
 Otherwise,
 $D(X) = -1$ already holds after the first action
 and still after the second one.

 We now distinguish four cases and show that,
 in each of them,
 one of the three operations leads to an optimal solution again.
 Let~$Y$ denote the optimal follower's choice corresponding to~$X$.

 \emph{Case 1:}
 $e_2 \in Y$.
 In this case,
 $\lvert X \rvert < b$ holds
 and hence, it is still feasible for the leader to choose~$X \cup \{e_2\}$.
 This leads to the follower choosing~$Y \setminus \{e_2\}$
 since he can select one item less,
 while one of the items he selected before is not available anymore,
 so the selection of the other items is still optimal.
 Thus,
 the overall solution~$X \cup Y$ does not change
 and~$X \cup \{e_2\}$ is still optimal for the leader.

 From now on, we assume~$e_2 \notin Y$.
 For the following cases,
 consider the optimal follower's solution~$\overline{Y}$
 corresponding to the leader's choice~$\overline{X} = X \setminus \{e_1\}$.
 The follower chooses the same items as before
 and one additional one,
 which we call~$\overline{e}$,
 i.e., $\overline{Y} = Y \cup \{\overline{e}\}$.

 \emph{Case 2:}
 $\overline{e} = e_1$.
 The overall solution~$X \cup Y$ does not change
 when the leader removes~$e_1$ from her solution~$X$
 because the follower will take it instead.
 Hence,
 $\overline{X}$ is optimal for the leader.

 \emph{Case 3:}
 $\overline{e} = e_2$.
 This means that,
 in the overall solution~$X \cup Y$,
 the item~$e_1$ is replaced by~$e_2$
 when the leader chooses~$X \setminus \{e_1\}$ instead of~$X$.
 As~$c(e_1) \ge c(e_2)$,
 this does not make the solution worse for the leader,
 so it is still optimal.

 \emph{Case 4:}
 $\overline{e} \notin \{e_1, e_2\}$.
 Hence,
 when the leader chooses~$\overline{X}$,
 neither~$e_1$ nor~$e_2$ occur in the overall solution.
 If the leader adds~$e_2$ to her solution~$\overline{X}$ now,
 the follower will leave out~$\overline{e}$ again,
 but not change his solution apart from that,
 i.e., he will choose~$Y$ again.
 Therefore,
 the leader choosing~$X \setminus \{e_1\} \cup \{e_2\}$
 leads to an overall solution~$X \cup Y \setminus \{e_1\} \cup \{e_2\}$,
 which is still optimal for the leader due to~$c(e_1) \ge c(e_2)$.
\end{proof}

As argued in Section~\ref{sec_BSP_var_sets},
the assumption $\Ell \subseteq \Eff$ is not a relevant restriction.
Accordingly,
with a small modification,
Algorithm~\ref{alg_BSP}
also solves the \BSP{}
in the general case:

\begin{corollary} \label{cor_BSP_alg}
  The \BSP{}
  can be solved
  in time~$O(n \log n)$.
\end{corollary}

\section{Definition of the Robust Bilevel Selection Problem} \label{sec_RBSP}

In this section,
we start to investigate the \BSP{} under uncertainty.
More precisely,
we assume that
the follower's objective function is uncertain from the leader's perspective,
and apply the concept of robust optimization,
i.e., we assume that
an uncertainty set~$\U \subseteq \R^\Eff$
is given,
from which an adversary (of the leader) draws a function~$d \colon \Eff \to \R$ of follower's item costs
such that the result is as bad as possible for the leader.
As an extension of the formulation in~\eqref{eq_BSP},
the \RBSP{} can be written as follows:
\begin{equation} \label{eq_RBSP} \tag{RBSP}
\begin{aligned}
  \min_{X \subseteq \Ell}~ \max_{d \in \U}~& c(X \cup Y)\\
  \st~ & Y \in
  \begin{aligned}[t]
    \argmin_{Y'}~ & d(Y')\\
    \st~ & Y' \subseteq \Eff \setminus X \\
    & \lvert X \cup Y' \rvert = b
  \end{aligned}
\end{aligned}
\end{equation}

The situation of the three decision makers
can now be imagined as follows:
From the follower's perspective,
nothing changes compared to the version without uncertainty,
as the leader's and the adversary's decisions are already fixed at this point.
Hence,
the follower still solves a single-level \SP{}.
The adversary chooses an objective function~$d$ for the follower from a given uncertainty set~$\U$.
We observe that,
by doing so,
he influences the order of the follower's items
that the follower will then use for his greedy solution.
The adversary's objective is
to choose~$d$ such that the resulting follower's optimal solution~$Y$
has the worst possible costs~$c(Y)$ for the leader.
Finally,
the leader's task is
to anticipate all of this
when choosing her own solution~$X$ in the beginning
such that the resulting overall set~$X \cup Y$ of selected items
has the minimal possible costs~$c(X \cup Y)$ for her.

The adversary's impact on the leader
heavily depends on the type of the uncertainty set~$\U$.
In fact,
uncertainty sets of a specific structure
can lead to specific structures of the possible follower's greedy orders
and therefore of the worst-case follower's reactions
that the leader has to take into account.
We will discuss the adversary's problem in Section~\ref{sec_adversary},
for discrete uncertainty, interval uncertainty, and discrete uncorrelated uncertainty.
In Section~\ref{sec_leader},
we then turn to the leader's problem.
Here, the general case turns out to be significantly harder
than the special case where $\Ell \cap \Eff = \emptyset$;
see Sections~\ref{sec_leader_disj} and~\ref{sec_leader_general}.
Note that
the distinction between these two cases
is only relevant from the leader's perspective
because, after a leader's solution~$X$ is fixed,
the adversary and the follower just operate on the remaining follower's items~$\Eff \setminus X$.
Therefore,
it is reasonable to first discuss the adversary's problem
in Section~\ref{sec_adversary},
and distinguish the two cases regarding the item sets only in Section~\ref{sec_leader}.

\section{The Adversary's Problem} \label{sec_adversary}

In this section,
we always assume a fixed feasible leader's solution~$X \subseteq \Ell$ to be given,
and consider the adversary's problem of choosing a scenario from the uncertainty set~$\U$
that is a worst possible one for the leader.
This problem highly depends on the structure of the uncertainty set
and we investigate it for three common types of uncertainty sets
in the following.
The easiest case from the adversary's perspective is the one
of discrete uncertainty in Section~\ref{sec_adversary_discrete}.
More work is required to handle the case of interval uncertainty.
However,
we still achieve a polynomial-time algorithm
for this setting
in Section~\ref{sec_adversary_interval}.
The strategy is very similar to the one for handling interval uncertainty
in the \RBCKP{} in~\cite{buchheimhenke2022}.
Finally,
the case of discrete uncorrelated uncertainty
is investigated in Section~\ref{sec_adversary_discrete_uncorr}.
It turns out to be equivalent to interval uncertainty,
now in contrast to \cite{buchheimhenke2022}
where the adversary's problem is NP-hard
in case of discrete uncorrelated uncertainty.

\subsection{Discrete Uncertainty} \label{sec_adversary_discrete}

First,
we focus on the case of discrete uncertainty,
where the uncertainty set~$\U$ is a finite set
that is explicitly given in the input.
In this case,
it is easy to see
that the adversary's problem can be solved in polynomial time
by enumerating all scenarios:

\begin{theorem} \label{thm_RBSP_discrete_adversary}
  For any fixed feasible leader's solution of the \RBSP{}
  with a discrete uncertainty set~$\U$,
  the adversary's problem can be solved in time~$O(\lvert \U \rvert n)$.
\end{theorem}

\begin{proof}
  Let~$\bff = b - \bll$ be the capacity
  that is to be filled by the follower,
  where~$\bll = \lvert X \rvert$ is the capacity
  the given leader's solution $X \subseteq \Ell$ uses.
  The adversary's task is
  to choose follower's item costs $d \in \U$
  such that the resulting follower's solution is a worst possible one for the leader.
  We know that the follower's problem,
  for fixed item costs,
  is a single-level \SP{} on $\Eff \setminus X$ with capacity~$\bff$.
  Such a problem can be solved in linear time;
  see Section~\ref{sec_BSP_def}.
  When the adversary enumerates all scenarios in~$\U$
  and computes the resulting follower's solution,
  together with its leader's costs,
  for each of them,
  he thus achieves a running time of~$O(\lvert \U \rvert n)$.
\end{proof}

This result is in line with the observation
in~\cite{buchheimhenkehommelsheim2021} that,
also in more general robust bilevel optimization problems,
the adversary's problem is on the same level of the polynomial hierarchy
as the follower's problem
in case of discrete uncertainty,
since the adversary can always enumerate all scenarios
and solve the corresponding follower's problems.

\subsection{Interval Uncertainty} \label{sec_adversary_interval}

In the case of interval uncertainty,
i.e., when the uncertainty set has the form
$\U = \prod_{e \in \Eff} [d^-(e),d^+(e)]$
with given functions~$d^-, d^+ \in \Q^\Eff$
that satisfy $d^-(e) \le d^+(e)$ for all~$e \in \Eff$,
we again present a polynomial-time algorithm
for the adversary's problem.
However,
in contrast to the typical setting in robust optimization,
it is not true in general for our type of robust bilevel optimization problems
that, e.g., the worst case is obtained when all values attain their upper bounds;
see also~\cite{buchheimhenkehommelsheim2021}.

As already noted in Section~\ref{sec_RBSP},
the scenarios arising from an uncertainty set in the \RBSP{}
affect the greedy order of the follower's items
the follower then uses to determine his solution.
For the adversary and the leader,
only this order resulting from a scenario is relevant,
but not the actual follower's item costs.
For this reason,
we study
which orders can be induced by the scenarios in an interval uncertainty set.
We observe that
the possible orders correspond to an interval order;
see Lemma~\ref{lem_RBSP_interval_orders}.
For more details on interval orders,
we refer to, e.g., \cite{moehring1989} and the references therein.

The connection between follower's solutions and (prefixes of) interval orders,
and also the strategy we will then use for the adversary's algorithm,
are very similar to how the case of interval uncertainty
is handled in Section 4 of \cite{buchheimhenke2022} for the \RBCKP{},
and in particular for the adversary's problem there.
Therefore,
we only briefly repeat the arguments
and omit the details here.
In fact,
the current situation is simpler than the one in \cite{buchheimhenke2022}
because, in contrast to the setting of~\cite{buchheimhenke2022}, we are not allowed to select fractions of items here
and therefore do not have to care about
the special role of the fractional item in a follower's solution.

For the sake of a simpler exposition,
we assume that
there are no one-point intersections
between intervals involved in~$\U$.
More precisely,
we assume that,
for all $e_1, e_2 \in \Eff$,
we have $d^-(e_1) \neq d^+(e_2)$.
Handling the case with one-point intersections
requires to deal with some technical details
related to the optimistic and pessimistic setting
more carefully;
see Remark~2 in~\cite{buchheimhenke2022}.

Intuitively,
items whose follower's costs lie in disjoint intervals in~$\U$,
always have the same order in the follower's greedy order,
independently of the adversary's decision.
On the other hand,
if the intervals corresponding to two follower's items intersect,
then the adversary can decide on the order in which the items will be considered by the follower.
This property gives some insight into the structure of the follower's optimal solutions,
which can formally be expressed as follows:

\begin{lemma} \label{lem_RBSP_interval_orders}
  Given an interval uncertainty set $\U = \prod_{e \in \Eff} [d^-(e),d^+(e)]$,
  define a binary relation~$\prec_\U$ on~$\Eff$
  such that, for $e_1, e_2 \in \Eff$,
  we have $e_1 \prec_\U e_2$
  if and only if $d^+(e_1) < d^-(e_2)$.
  Then the relation~$\prec_\U$
  is an interval order.
  In the \RBSP{}
  with the interval uncertainty set~$\U$,
  the set of greedy orders of the follower's items
  that arise from the scenarios in~$\U$
  is exactly the set of linear extensions
  of the interval order~$\prec_\U$ on $\Eff$.
  Moreover,
  the set of follower's optimal solutions,
  given a fixed feasible leader's solution~$X \subseteq \Ell$
  and any scenario in~$\U$,
  is exactly the set of prefixes~$\Effn$
  of the restriction of $\prec_\U$ to $\Eff \setminus X$
  that satisfy $\lvert \Effn \rvert = b - \lvert X \rvert$.
\end{lemma}

Using Lemma~\ref{lem_RBSP_interval_orders},
we could now compute the interval order and its linear extensions explicitly
in order to solve the adversary's problem.
However,
this would not give a polynomial-time algorithm in general
because the number of linear extensions can be exponential.

Nevertheless,
we can solve the adversary's problem in polynomial time.
For this,
we use that
it can be seen as a special case of the
\emph{precedence constraint knapsack problem}
and that this problem can be solved in pseudopolynomial time
by an algorithm described in~\cite{woeginger2003}
if the precedence constraints correspond to an interval order;
see \cite{buchheimhenke2022} for more details.
The pseudopolynomial-time algorithm runs in polynomial time here
because we consider a selection problem instead of a knapsack problem,
i.e., the item sizes are all~$1$ in our case.
This results in Algorithm~\ref{alg_RBSP_interval_adversary}
for solving the adversary's problem.

In Line~\ref{alg_RBSP_interval_adversary_subroutine}
of Algorithm~\ref{alg_RBSP_interval_adversary},
an appropriate subroutine is called
for solving a \SP{} in linear time.
The input of the subroutine
consists of a finite item set,
a number of items to select
and a function of item costs;
see also the definition of the \SP{} in Section~\ref{sec_BSP_def}.

\begin{algorithm}
  \caption{Algorithm for the adversary's problem of the \RBSP{} with interval uncertainty}
  \label{alg_RBSP_interval_adversary}

  \Input{finite sets $\Ell$ and $\Eff$,
    $b \in \{0, \dots, \lvert \Ell \cup \Eff \rvert\}$, $c \colon \Ell \cup \Eff \to \Q$, an interval uncertainty set~$\U$ given by $d^-, d^+ \colon \Eff \to \Q$ with $d^-(e) \le d^+(e)$ for all~$e \in \Eff$, a feasible leader's solution~$X \subseteq \Ell$}

  \Output{an optimal adversary's solution~$d \in \U$ (i.e., $d \colon \Eff \to \Q$ with $d^-(e) \le d(e) \le d^+(e)$ for all~$e \in \Eff$) of the \RBSP{}}

  $\E := \Eff \setminus X$

  \If{$\E = \emptyset$}
  {
    $Y_{\overline{e}} := \emptyset$ \label{alg_RBSP_interval_adversary_Y_trivial}
  }
  \Else{
    $\bff := b - \lvert X \rvert$

    $\overline{\E} := \emptyset$

    \For{$\overline{e} \in \E$ \label{alg_RBSP_interval_adversary_loop}}
    {
      $\E_{\overline{e}}^- := \{e \in \E \mid d^+(e) < d^-(\overline{e})\}$ \label{alg_RBSP_interval_adversary_predecessors}

      $\E_{\overline{e}}^0 := \{e \in \E \mid d^-(e) \le d^-(\overline{e}) \le d^+(e)\}$ \label{alg_RBSP_interval_adversary_incomparable}

      \If{$\lvert \E_{\overline{e}}^- \rvert \le \bff$ and $\lvert \E_{\overline{e}}^0 \rvert \ge \bff - \lvert \E_{\overline{e}}^- \rvert$ \label{alg_RBSP_interval_adversary_if}}
      {
        $Y_{\overline{e}}^0 := \SP{}(\E_{\overline{e}}^0, \bff - \lvert \E_{\overline{e}}^- \rvert, -c \restriction \E_{\overline{e}}^0)$ \label{alg_RBSP_interval_adversary_subroutine}

        $Y_{\overline{e}} := \E_{\overline{e}}^- \cup Y_{\overline{e}}^0$

        $\overline{\E} := \overline{\E} \cup \{\overline{e}\}$ \label{alg_RBSP_interval_adversary_loop_end}
      }
    }

    select $\overline{e} \in \argmax\{c(Y_e) \mid e \in \overline{\E}\}$ arbitrarily \label{alg_RBSP_interval_adversary_best_iteration}
  }

  \Return{$d \colon \Eff \to \Q$ with $d(e) := d^-(e)$ for all $e \in Y_{\overline{e}}$ and $d(e) := d^+(e)$ for all $e \in \Eff \setminus Y_{\overline{e}}$ \label{alg_RBSP_interval_adversary_return}}
\end{algorithm}

Very similarly to Lemma~1 in \cite{buchheimhenke2022},
one can now prove:

\begin{theorem} \label{thm_RBSP_interval_adversary}
  For any fixed feasible leader's solution of the \RBSP{}
  with an interval uncertainty set,
  Algorithm~\ref{alg_RBSP_interval_adversary} solves
  the adversary's problem in time~$O(n^2)$.
\end{theorem}

\subsection{Discrete Uncorrelated Uncertainty} \label{sec_adversary_discrete_uncorr}

We now turn to the setting of discrete uncorrelated uncertainty,
in which there is a finite number of possible follower's item costs for each follower's item,
independently of each other,
i.e., $\U = \prod_{e \in \Eff} \U_e$ with finite sets~$\U_e \subseteq \Q$ for all~$e \in \Eff$.

For every such uncertainty set,
its convex hull is an interval uncertainty set
and can be written as $\conv(\U) = \prod_{e \in \Eff} [d^-(e), d^+(e)]$,
where $d^-(e)$ and~$d^+(e)$ are the minimal and maximal value in $\U_e$, respectively,
for all~$e \in \Eff$.
Recall that,
in our setting of robust bilevel optimization problems,
one cannot replace an uncertainty set by its convex hull in general
without changing the problem;
see, e.g., \cite{buchheimhenkehommelsheim2021}.
However,
for the \RBSP{},
we will see in the following
that a discrete uncorrelated uncertainty set
can still be replaced by its convex hull.
We emphasize that this
is very specific to the \BSP{}
and that even closely related problems
do not have this property,
in particular the continuous variant of the \RBSP{}
and the \RBCKP{};
see Section~\ref{sec_RCBSP_discrete_uncorr}
and~\cite{buchheimhenke2022}.

First observe that
the adversary's problem with $\conv(\U)$ is a relaxation of the one with~$\U$.
Moreover,
the optimal adversary's solutions
that are computed by Algorithm~\ref{alg_RBSP_interval_adversary}
in the case of interval uncertainty
always attain an endpoint of each of the intervals.
Hence,
when replacing a discrete uncorrelated uncertainty set~$\U$ by its convex hull
as described above
and solve the resulting problem with interval uncertainty
using Algorithm~\ref{alg_RBSP_interval_adversary},
then only the endpoints of the intervals~$[d^-(e), d^+(e)]$ are relevant
for the optimal adversary's solutions.
As these points are contained in the original sets~$\U_e$,
this implies that
the computed adversary's solutions
are also valid,
and therefore also optimal,
for the original problem with discrete uncorrelated uncertainty.
Hence,
the two types of uncertainty sets are equivalent
for the \RBSP{},
and
we have shown:

\begin{theorem} \label{thm_RBSP_discrete_uncorr_equiv}
  The \RBSP{} with a discrete uncorrelated uncertainty set
  can be linearly reduced to the corresponding problem with an interval uncertainty set.
  The same holds for the adversary's problems,
  for any fixed feasible leader's solution.
\end{theorem}

This implies that
Theorem~\ref{thm_RBSP_interval_adversary}
also holds in the setting of discrete uncorrelated uncertainty.

\begin{remark} \label{rem_RBSP_discrete_uncorr_orders}
The set of greedy orders of the follower's items
that the adversary can enforce
cannot be described in terms of partial orders here,
as it was the case for interval uncertainty;
see Lemma~\ref{lem_RBSP_interval_orders}.
As an example,
let $\Eff = \{e_1, e_2, e_3\}$
with $\U_{e_1} = \{1, 4\}$, $\U_{e_2} = \{2\}$, and $\U_{e_3} = \{3\}$.
Then the only relation
that is true for every possible follower's greedy order
is that $e_2$ precedes $e_3$,
but the linear extension $e_2 \prec e_1 \prec e_3$ of this partial order
cannot be enforced.
However,
every prefix of $e_2 \prec e_1 \prec e_3$
is also the prefix of one of the two linear orders
that the adversary can produce.
In fact,
the sets of optimal follower's solutions
that the adversary can enforce,
for any fixed feasible leader's solution,
are always the same
for $\U$ and $\conv(\U)$,
which is an intuitive reason why
the two types of uncertainty sets are equivalent.
\end{remark}

\section{The Robust Leader's Problem} \label{sec_leader}

After having seen polynomial-time algorithms
for the adversary's problem in the previous section,
we now turn to the leader's perspective in the \RBSP{}.
The case of disjoint item sets turns out to be significantly easier here than the general one.
In Section~\ref{sec_leader_disj},
we will show how to derive polynomial-time algorithms
for the leader's problem
from the ones for the adversary's problem,
for discrete uncertainty, interval uncertainty, and discrete uncorrelated uncertainty.
In Section~\ref{sec_leader_general},
we then prove that the \RBSP{} is NP-hard in general,
for discrete uncertainty,
and show how to approximate it
and how to solve it in exponential time.

\subsection{The Special Case of Disjoint Item Sets} \label{sec_leader_disj}

The setting of disjoint leader's and follower's item sets
was already easier to understand
than the general one
in the \BSP{} without uncertainty,
although the same algorithm was able to solve both variants of the problem;
see Section~\ref{sec_BSP_alg}.
In particular,
it was clear there that
the capacity~$b$ has to be split between leader and follower somehow
and that each of the players solves a single-level \SP{} on their own set of items then.
Hence,
the algorithm mainly needed to determine (by enumeration)
how the leader can split the capacity optimally.
In the robust setting,
this basic idea remains valid.
Therefore,
for a fixed choice of $\bll$,
the leader's decision can again be seen as a single-level \SP{}
that she can solve in a greedy way,
and the best choice of $\bll$
can again be determined by enumeration.

On the follower's side of the problem,
now also the adversary comes into play.
We do not only solve one single-level \SP{} for the follower,
but solve the adversary's problem of determining a follower's objective function
that results in the worst possible follower's solution for the leader.
Here, the polynomial-time algorithms
that we have presented for the adversary's problem in Section~\ref{sec_adversary}
will be useful.
The leader's enumeration approach then leads to
a polynomial-time algorithm
whenever we can solve the adversary's problem in polynomial time:

\begin{theorem} \label{thm_RBSP_disj_adversary}
  Consider the \RBSP{} with $\Ell \cap \Eff = \emptyset$,
  and with any type of uncertainty set.
  Suppose that the corresponding adversary's problem can be solved in time at most~$A(I)$,
  given an instance~$I$ of the \RBSP{}%
  \footnote{Note that parts of the structure of the instance~$I$ as well as the function~$A$ may depend on the type of uncertainty set.}
  together with any leader's solution.
  Then the \RBSP{},
  given an instance~$I$ with $n$~items,
  can be solved in time $O(n A(I))$.
\end{theorem}

\begin{proof}
  The leader enumerates all feasible values of~$\bll$
  that determine the capacity used by herself.
  These are clearly the integer values in the same range as in Algorithm~\ref{alg_BSP},
  and their number is $O(n)$.
  For each of them,
  the leader solves a single-level \SP{} on $\Ell$ with capacity~$\bll$
  in order to determine her solution,
  and, in addition,
  solves the adversary's problem for this fixed leader's choice.
  The former can be done in time~$O(n)$
  and the latter in time at most~$A(I)$.
  The best solution computed in this way
  clearly gives an optimal leader's solution.
  Since the running time~$A(I)$ for solving the adversary's problem
  cannot be expected to be faster than~$O(n)$,
  which is the time required to solve the follower's problem
  for a given scenario,
  the resulting running time can be written as~$O(n A(I))$.
\end{proof}

\begin{remark} \label{rem_RBSP_disj_adversary_time}
  For computing all potential leader's solutions,
  the algorithm described in the proof of Theorem~\ref{thm_RBSP_disj_adversary}
  requires a running time of $O(n)$
  in each of the $O(n)$ iterations,
  i.e., a total running time of $O(n^2)$.
  This part of the algorithm's total running time
  can be reduced to $O(n \log n)$
  by presorting the leader's items once in the beginning of the algorithm,
  as in Algorithm~\ref{alg_BSP};
  see also Remark~\ref{rem_BSP_alg_time}.
  However,
  this does not improve the stated total running time of $O(n A(I))$.
  Moreover,
  when applying Theorem~\ref{thm_RBSP_disj_adversary}
  for specific types of uncertainty sets below,
  we do not only use it as a black box
  and improve on this running time anyway,
  by doing similar preprocessings
  that depend on the type of uncertainty set
  and enable to solve the adversary's problems faster.
\end{remark}

We can now combine Theorem~\ref{thm_RBSP_disj_adversary}
with the algorithms for the adversary's problem from Section~\ref{sec_adversary},
for the different types of uncertainty sets:

\begin{corollary} \label{cor_RBSP_disj_discrete_leader}
  The \RBSP{} with $\Ell \cap \Eff = \emptyset$,
  and with a discrete uncertainty set~$\U$,
  can be solved in time~$O(\lvert \U \rvert n \log n)$.
\end{corollary}

\begin{proof}
  The existence of a polynomial-time algorithm
  follows from Theorems~\ref{thm_RBSP_discrete_adversary}~and~\ref{thm_RBSP_disj_adversary}.
  The desired running time of this algorithm
  can be achieved by implementing it
  as a variant of Algorithm~\ref{alg_BSP}
  (see also Remark~\ref{rem_RBSP_disj_adversary_time}),
  as follows:
  Precompute the bijection~$\pll$
  and bijections~$\pffd$ for all scenarios~$d \in \U$.
  This takes time~$O(\lvert \U \rvert n \log n)$.
  The enumeration of values~$\bll$
  and the computation of the corresponding leader's solutions~$X_\bll$
  is done as in Algorithm~\ref{alg_BSP}.
  To determine the worst-case follower's solution~$Y_\bll$
  in every iteration,
  we compute~$Y_\bll^d$ from~$\pffd$ for each scenario~$d \in \U$
  and choose one with the maximal costs~$c(Y_\bll^d)$.
  Each of the $O(n)$ iterations of the loop
  can thus be implemented to run in time~$O(\lvert \U \rvert)$.
\end{proof}

\begin{corollary} \label{cor_RBSP_disj_interval_leader}
  The \RBSP{} with $\Ell \cap \Eff = \emptyset$,
  and with interval uncertainty or discrete uncorrelated uncertainty,
  can be solved in time~$O(n^2 \log n)$.
\end{corollary}

\begin{proof}
  The existence of a polynomial-time algorithm
  directly follows from Theorems~\ref{thm_RBSP_interval_adversary}~and~\ref{thm_RBSP_disj_adversary}.
  These results imply a running time of~$O(n^3)$.
  However,
  as indicated in Remark~\ref{rem_RBSP_disj_adversary_time},
  this can be improved to~$O(n^2 \log n)$
  by presorting not only the leader's,
  but also the follower's items.
  More precisely,
  we compute the sets $\E_{\overline{e}}^-$ and $\E_{\overline{e}}^0$,
  together with an order of the items in $\E_{\overline{e}}^0$
  by their leader's item costs $c$,
  for all $\overline{e} \in \Eff$,
  in the beginning of the algorithm,
  before enumerating the splittings of the capacity $b$
  into leader's and follower's capacities $\bll$ and~$\bff$, respectively.
  This requires a running time of $O(n^2 \log n)$
  and enables to implement
  every iteration of the loop in Algorithm~\ref{alg_RBSP_interval_adversary}
  to run in constant time.
  Thus,
  the overall running time is $O(n^2 \log n)$.

  Theorem~\ref{thm_RBSP_discrete_uncorr_equiv} implies that
  the statement
  also holds in the setting of discrete uncorrelated uncertainty.
\end{proof}

\subsection{The General Case} \label{sec_leader_general}

In the general case,
the \RBSP{} becomes much more involved
than in the special case where $\Ell \cap \Eff = \emptyset$.
In fact,
we will prove in Section~\ref{sec_leader_general_hardness}
that the leader's problem is strongly NP-hard
in the general case.
We will, however,
suggest some algorithms to deal with the general \RBSP{}:
The polynomial-time algorithm presented in Section~\ref{sec_leader_general_approx}
leads to a leader's objective function value
that is at most by a factor of~2 worse than the optimum.
Besides approximation algorithms,
another way to deal with NP-hard problems is
to solve them in superpolynomial time.
We will analyze two such algorithms
for the \RBSP{}
in Section~\ref{sec_leader_general_exact}.
One of them shows that
the problem can be solved in polynomial time
given a constant number of scenarios
(Theorem~\ref{thm_RBSP_discrete_exact_u}).

\subsubsection{NP-Hardness} \label{sec_leader_general_hardness}

When the leader's and the follower's item sets are not disjoint,
it is not clear anymore that
a greedy decision on the leader's items is optimal.
This can make the leader's problem significantly harder than the adversary's problem,
in contrast to Theorem~\ref{thm_RBSP_disj_adversary}
for the disjoint setting,
which showed that the leader's problem can be polynomially reduced to the adversary's problem in this case.
Theorem~\ref{thm_RBSP_hardness} will show that
the general \RBSP{} with discrete uncertainty
is NP-hard,
while the corresponding adversary's problem is solvable
in polynomial time
by Theorem~\ref{thm_RBSP_discrete_adversary}.
Note that
the latter implies that
the leader's problem is NP-easy,
as it cannot be more than one level harder
than the adversary's problem,
i.e., the evaluation of the leader's objective function;
see also~\cite{buchheimhenkehommelsheim2021}.
The \RBSP{} with discrete uncertainty is therefore NP-equivalent.

For the following proof,
we will use a reduction
from the well-known strongly NP-hard \VCP{}~\cite{gareyjohnson1979}.
In this problem,
we are given an undirected graph~$G = (V, E)$,
and the task is to find a vertex cover,
i.e., a vertex set~$X \subseteq V$ such that,
for every edge~$e \in E$,
at least one endpoint of~$e$ is contained in~$X$,
of minimal cardinality.

\begin{theorem} \label{thm_RBSP_hardness}
  The \RBSP{} with discrete uncertainty
  is strongly NP-hard.
\end{theorem}

\begin{proof}
  Consider an instance of the \VCP{},
  consisting of an undirected graph~$G = (V, E)$ with $n \in \N$~vertices~$V = \{v_1, \dots, v_n\}$
  and $m \in \N$~edges~$E = \{e_1, \dots, e_m\}$.
  Without loss of generality,
  we may assume that the minimal cardinality of a vertex cover in this instance
  is at least $2$ and at most $n - 1$.
  We build an instance of the \RBSP{} as follows:
  $\Ell = V$, $\Eff = V \cup \{h_1, \dots, h_{n + 3}\}$,
  $b = n + 1$, and
  \begin{equation*}
    c(h) =
    \begin{cases}
      1 & \text{ for } h \in V \cup \{h_{n + 1}, h_{n + 2}\} \\
      n & \text{ for } h = h_{n + 3} \\
      0 & \text{ for } h \in \{h_1, \dots, h_n\}.
    \end{cases}
  \end{equation*}
  The uncertainty set~$\U = \{d_1, \dots, d_{m + 1}\}$ consists of a scenario~$d_i$ for every edge~$e_i \in E$ of~$G$
  and one additional scenario~$d_{m + 1}$.
  Instead of defining the scenarios, i.e., follower's objective functions, $d_i$ explicitly,
  it suffices to give orders of the items that are used in the follower's greedy algorithm in the respective scenarios.
  It is always possible to define a function~$d_i$ for which this order is the unique optimal one for the follower
  and, in particular, to compute such a function~$d_i$ with polynomial-size values in polynomial time;
  see also Remark~\ref{rem_opt_pess}.
  Note that the follower never selects more than $b = n + 1$ items,
  so that we can neglect the items after the first $n + 1$ ones
  in the follower's greedy order.
  For any $i \in \{1, \dots, m\}$,
  the scenario~$d_i$,
  corresponding to the edge~$e_i = \{v_j, v_k\} \in E$ with $j < k$,
  is defined by the order
  $h_{n + 1}, h_{n + 2}, v_1, \dots, v_{j - 1}, v_{j + 1}, \dots, v_{k - 1}, v_{k + 1}, \dots, v_n, h_{n + 3}$.
  The scenario~$d_{m + 1}$ is defined by the order
  $h_{n + 3}, h_1, \dots, h_n$.

  We will prove that
  every optimal leader's solution~$X \subseteq \Ell = V$ is a minimum vertex cover in~$G$.
  First,
  let~$X \subseteq \Ell$ not be a vertex cover,
  i.e., there is an edge~$e_i = \{v_j, v_k\} \in E$ with~$j < k$
  such that~$v_j \notin X$ and~$v_k \notin X$.
  By definition of the corresponding scenario~$d_i$,
  all items selected by the leader in~$X$
  are among the first $n + 1$~items in the follower's greedy order.
  Hence,
  when the follower completes the leader's solution~$X$ greedily to $b = n + 1$~items in total,
  this leads to a follower's solution~$Y$
  with~$X \cup Y$ consisting of exactly the first $n + 1$~items in the follower's greedy order,
  i.e., $X \cup Y = \{h_{n + 1}, h_{n + 2}, v_1, \dots, v_{j - 1}, v_{j + 1}, \dots, v_{k - 1}, v_{k + 1}, \dots, v_n, h_{n + 3}\}$,
  which has a leader's objective function value of $c(X \cup Y) = 2n$.
  Thus,
  in the robust setting,
  every leader's solution~$X$ that is not a vertex cover
  leads to an objective function value of $2n$,
  which is the worst possible value from the leader's perspective.

  Now we turn to leader's solutions~$X \subseteq \Ell$ that are vertex covers.
  By definition,
  for every edge~$e_i \in E$,
  at least one of its endpoints is in~$X$.
  Accordingly,
  for every scenario~$d_i$,
  at least one item that is not among the first $n + 1$~items of the follower's greedy order in this scenario
  is selected by the leader.
  When the follower now greedily completes the leader's solution~$X$ to $n + 1$~items in total,
  he will never select the item~$h_{n + 3}$,
  which is the $(n + 1)$-th item in his greedy order.
  In fact,
  we get $X \cup Y \subseteq V \cup \{h_{n+1}, h_{n+2}\}$.
  Thus,
  in each of the scenarios~$d_1, \dots, d_m$,
  the leader's objective value is~$n + 1$.
  In scenario~$d_{m + 1}$,
  the follower always selects item~$h_{n + 3}$
  and possibly some items with leader's cost~$0$.
  For the leader's solution~$X$ and the corresponding follower's solution~$Y$,
  the leader's objective value is thus
  $c(X \cup Y) = \lvert X \rvert + n$.
  This implies that the adversary can be assumed to always select scenario~$d_{m + 1}$
  when the leader has selected a vertex cover~$X$,
  and that, as a leader's solution, any vertex cover is better than any set that is not a vertex cover.
  Finally,
  optimizing the leader's objective value among the solutions that are vertex covers
  is equivalent to finding a vertex cover~$X \subseteq V$ of minimal size~$\lvert X \rvert$.

  Note that
  we have indeed shown strong NP-hardness,
  since the \VCP{} does not have any numerical parameters
  and all values of the constructed functions~$c$ and $d_i \in \U$ have polynomial size.
\end{proof}

For other types of uncertainty sets,
it remains an open question
whether the \RBSP{} is NP-hard.
We conjecture that this is the case,
in particular for interval uncertainty.
In view of the algorithmic results in Sections~\ref{sec_adversary} and~\ref{sec_leader_disj},
interval uncertainty does not seem to be easier
than discrete uncertainty
for the \RBSP{}.
Also the more general complexity results in \cite{buchheimhenkehommelsheim2021} indicate that
interval uncertainty
is in some sense harder than discrete uncertainty
in our robust bilevel setting.

Because of Theorem~\ref{thm_RBSP_hardness},
we cannot hope for a polynomial-time exact algorithm
(unless P~$=$~NP)
for the general \RBSP{} with discrete uncertainty.
Therefore,
we develop approximation algorithms
and exponential-time exact algorithms
for the \RBSP{}
in Sections~\ref{sec_leader_general_approx} and~\ref{sec_leader_general_exact}, respectively.

\subsubsection{Approximation Algorithms} \label{sec_leader_general_approx}

In the setting of disjoint item sets,
we have seen that
the leader's problem can be solved in polynomial time
for all considered types of uncertainty sets,
by reducing it to the adversary's problem;
see Theorem~\ref{thm_RBSP_disj_adversary}.
In this algorithm,
the leader enumerates
how many items she selects herself
and selects them greedily
for every fixed number.
We will now see that
this algorithm does not yield an optimal solution anymore
in the general setting.
This is in contrast to the problem without uncertainty,
where the same idea could be used
for the disjoint and for the general case;
see Section~\ref{sec_BSP_alg}.
However,
the difference is not surprising here
in view of the result of Theorem~\ref{thm_RBSP_hardness}
that the general \RBSP{}
with discrete uncertainty
is strongly NP-hard.
But
the algorithm
that is exact in the disjoint case
now turns out to be a $2$-approximation algorithm,
as we will show in Theorem~\ref{thm_RBSP_approx_adversary}.
In fact,
we again prove this
for any type of uncertainty set
for which the adversary's problem can be solved in polynomial time,
like in Theorem~\ref{thm_RBSP_disj_adversary}.
The results in this section are originally due to Hartmann~\cite{hartmann2021},
while the proofs have been revised and simplified by the author of this article.

We need to emphasize that,
in order to make reasonable statements
about approximating objective function values,
we have to assume that
they are nonnegative.
More precisely,
we assume, for the following statements, that
the function~$c \colon \Ell \cup \Eff \to \Qge$
of leader's item costs
only attains nonnegative values.
Recall from Section~\ref{sec_BSP_var_costs} that
all leader's item costs can be shifted arbitrarily by a constant,
without changing the optimal solutions.
However,
this cannot be done without loss of generality here,
as the resulting constant offset in the objective values
has an impact on approximation factors.

\begin{theorem} \label{thm_RBSP_approx_adversary}
  Consider the \RBSP{},
  with any type of uncertainty set.
  Suppose that the corresponding adversary's problem can be solved in time at most~$A(I)$,
  given an instance~$I$ of the \RBSP{}
  together with any leader's solution.
  Then the problem of 2-approximating the \RBSP{},
  given an instance~$I$ with $n$~items,
  can be solved in time~$O(n A(I))$.
\end{theorem}

\begin{proof}
  We assume for this proof
  that $\Ell \subseteq \Eff$,
  which is not a restriction
  compared to the general \BSP{};
  see Section~\ref{sec_BSP_var_sets}.
  
  The leader enumerates all values $\bll \in \{0, \dots, \min\{b, \nll\}\}$
  and solves a single-level \SP{} on~$\Ell$ with capacity~$\bll$ in each iteration.
  Moreover,
  she determines an optimal solution of the adversary's problem
  corresponding to this leader's solution.
  The best solution achieved throughout these iterations is returned.
  This is the same algorithm
  as in Theorem~\ref{thm_RBSP_disj_adversary},
  and also its running time is the same.

  We will now prove that
  the algorithm always returns a solution with leader's costs
  that are at most twice the costs of an optimal solution.
  For this,
  denote the leader's solution in iteration~$\bll$ of the algorithm by~$X_\bll$,
  the corresponding follower's response in any scenario~$d \in \U$ by~$Y_\bll^d$,
  and the follower's response in the worst-case scenario by~$Y_\bll$,
  i.e., choose $Y_\bll \in \{Y_\bll^d \mid d \in \U\}$
  with $c(Y_\bll) = \max \{c(Y_\bll^d) \mid d \in \U\}$.
  Moreover,
  denote an optimal leader's solution by~$X^*$,
  the corresponding follower's response in any scenario~$d \in \U$ by~$Y^d$,
  and the follower's response in the worst-case scenario by~$Y^*$.

  For the rest of the proof,
  choose $\bll \in \{\lvert X^* \rvert, \dots, \min\{b, \nll\}\}$ to be minimal
  such that $\lvert X_\bll \cap Y^d \rvert \le \bll - \lvert X^* \rvert$
  for all~$d \in \U$.
  This is always possible because
  $\lvert X^* \rvert \le \min\{b, \nll\}$ and
  the inequality is always satisfied
  for~$\bll = \min\{b, \nll\}$.
  Indeed,
  since $X^*$ and~$Y^d$ are disjoint
  for all~$d \in \U$
  as corresponding leader's and follower's solutions,
  we then have either $\lvert X_b \cap Y^d \rvert \le \lvert Y^d \rvert = b - \lvert X^* \rvert$
  or $\lvert X_\nll \cap Y^d \rvert = \lvert \Ell \cap Y^d \rvert \le \lvert \Ell \setminus X^* \rvert = \nll - \lvert X^* \rvert$
  for all~$d \in \U$.

  If~$\bll = 0$,
  we must have $X^* = X_\bll = \emptyset$,
  which means that the algorithm finds an optimal solution.
  Therefore,
  we assume $\bll \ge 1$
  in the following.

  Due to the minimality of~$\bll$,
  we have $\bll = \lvert X^* \rvert$
  or there is some~$d \in \U$
  with $\lvert X_{\bll - 1} \cap Y^d \rvert > \bll - 1 - \lvert X^* \rvert$,
  which is equivalent to
    \begin{equation} \label{thm_RBSP_approx_adversary_proof_ineq} \tag{\textasteriskcentered}
    \lvert X_{\bll - 1} \cap Y^d \rvert \ge \bll - \lvert X^* \rvert.
  \end{equation}
  In fact,
  $\bll = \lvert X^* \rvert$ also implies \eqref{thm_RBSP_approx_adversary_proof_ineq},
  even for all~$d \in \U$.
  Hence,
  there is always some~$d \in \U$
  for which \eqref{thm_RBSP_approx_adversary_proof_ineq} is satisfied.
  Fix such a scenario~$d$ in the following.

  In order to prove the required approximation guarantee,
  we will show that
  the leader's costs of
  both the leader's and the follower's solution
  in iteration~$\bll$
  can be bounded from above
  by the costs of an optimal solution,
  i.e.,
  $c(X_\bll) \le c(X^*) + c(Y^*)$
  and $c(Y_\bll) \le c(X^*) + c(Y^*)$.
  Together,
  this proves that
  the total costs of the solution
  that the algorithm returns is at most
  $c(X_\bll) + c(Y_\bll) \le 2(c(X^*) + c(Y^*))$.

  We first focus on bounding $c(X_\bll)$.
  For this,
  we partition the set~$X_\bll$ into
  the two disjoint sets
  $B := X_{\bll - 1} \cap Y^d$
  and $A := X_\bll \setminus B$.
  Note that
  $B \subseteq X_\bll$
  because $X_\bll$ consists of
  $X_{\bll - 1}$ and one additional item
  by the construction of the leader's solutions in the algorithm.
  Clearly,
  $c(B) \le c(Y^d) \le c(Y^*)$ holds,
  so it remains to show $c(A) \le c(X^*)$.

  Since $X_\bll$ is selected greedily from~$\Ell$ according to the costs~$c$,
  we have~$c(e_1) \le c(e_2)$
  for all $e_1 \in X_\bll$ and all~$e_2 \in \Ell \setminus X_\bll$,
  and in particular for all $e_1 \in A \setminus X^* \subseteq X_\bll$
  and all $e_2 \in X^* \setminus A \subseteq \Ell \setminus X_\bll$.
  The last inclusion holds because
  $X^* \cap Y^d = \emptyset$ for the follower's response~$Y^d$ to the leader's solution~$X^*$,
  which implies $X^* \cap B = \emptyset$ and hence $X^* \cap X_\bll \subseteq A$.

  Moreover,
  we can use \eqref{thm_RBSP_approx_adversary_proof_ineq}
  to derive
  $\lvert A \rvert = \lvert X_\bll \rvert - \lvert B \rvert = \bll - \lvert X_{\bll - 1} \cap Y^d \rvert \le \lvert X^* \rvert$
  and hence $\lvert A \setminus X^* \rvert \le \lvert X^* \setminus A \rvert$.
  Together with the previous paragraph,
  this yields $c(A) = c(A \setminus X^*) + c(A \cap X^*) \le c(X^* \setminus A) + c(A \cap X^*) = c(X^*)$.
  This concludes the first part of the proof with $c(X_\bll) = c(A) + c(B) \le c(X^*) + c(Y^*)$.

  For the second part,
  we partition also the set~$Y_\bll$
  into two disjoint sets $A' := Y_\bll \cap X^*$ and $B' := Y_\bll \setminus X^*$.
  We directly observe $c(A') \le c(X^*)$,
  so it remains to prove $c(B') \le c(Y^*)$.

  Let $d' \in \U$ be the worst-case scenario in the considered solution of the algorithm,
  i.e., such that $Y_\bll^{d'} = Y_\bll$.
  We derive $\lvert B' \rvert = \lvert Y_\bll \setminus X^* \rvert \le \lvert Y_\bll \rvert = b - \bll$
  and,
  by the definition of~$\bll$,
  $\lvert Y^{d'} \setminus X_\bll \rvert = \lvert Y^{d'} \rvert - \lvert Y^{d'} \cap X_\bll \rvert \ge \lvert Y^{d'} \rvert - (\bll - \lvert X^* \rvert) = b - \bll$.
  Thus,
  $\lvert B' \rvert \le \lvert Y^{d'} \setminus X_\bll \rvert$.

  We know that
  $Y_\bll^{d'}$ is selected greedily from~$\E \setminus X_\bll$ by the follower,
  according to the costs~$d'$
  (and possibly the leader's costs~$c$ as a secondary criterion
  in the optimistic or pessimistic setting).
  This means that
  $Y_\bll^{d'}$ consists of the first $\lvert Y_\bll^{d'} \rvert = b - \bll$ items
  of the corresponding order
  of the set~$\E \setminus X_\bll$.
  By removing $X^*$ from the sorted set,
  we obtain that
  $B' = Y_\bll^{d'} \setminus X^*$ consists of the first $\lvert B' \rvert$ items
  of the same order of the set~$\E \setminus (X_\bll \cup X^*)$.

  Analogously,
  $Y^{d'}$ is selected greedily from~$\E \setminus X^*$ by the follower,
  i.e., again according to~$d'$ and possibly $c$ as a secondary criterion.
  We may assume that the same order as above is used.
  We now remove $X_\bll$ and derive that
  $Y^{d'} \setminus X_\bll$ consists of the first $\lvert Y^{d'} \setminus X_\bll \rvert$ items
  of the same order of the set~$\E \setminus (X_\bll \cup X^*)$.

  Together with $\lvert B' \rvert \le \lvert Y^{d'} \setminus X_\bll \rvert$,
  this implies that $B' \subseteq Y^{d'} \setminus X_\bll$
  and thus $c(B') \le c(Y^{d'}) \le c(Y^*)$,
  which concludes the second part of the proof.
\end{proof}

Combining Theorem~\ref{thm_RBSP_approx_adversary} with the algorithms for the adversary's problem
obtained in Section~\ref{sec_adversary}
leads to the following results
for our specific types of uncertainty sets.

\begin{corollary} \label{cor_RBSP_discrete_approx}
  The \RBSP{}
  with discrete uncertainty
  can be $2$-approximated in time~$O(\lvert \U \rvert n \log n)$.
\end{corollary}

\begin{proof}
  A polynomial-time 2-approximation algorithm
  directly follows from Theorems~\ref{thm_RBSP_discrete_adversary}~and~\ref{thm_RBSP_approx_adversary}.
  Analogously to Corollary~\ref{cor_RBSP_disj_discrete_leader},
  its running time can be improved from $O(\lvert \U \rvert n^2)$ to $O(\lvert \U \rvert n \log n)$
  by precomputing the bijections for all scenarios.
  Note that,
  in every iteration,
  we might have to remove the item that the leader adds to her solution
  from the items the follower chooses from,
  for each scenario,
  like in the general version of Algorithm~\ref{alg_BSP};
  see also the proof of Theorem~\ref{thm_BSP_alg}.
\end{proof}

\begin{corollary} \label{cor_RBSP_interval_approx}
  The
  \RBSP{} with interval uncertainty or discrete uncorrelated uncertainty
  can be 2-approximated in time~$O(n^2 \log n)$.
\end{corollary}

\begin{proof}
  The existence of a polynomial-time 2-approximation algorithm
  for the case of interval uncertainty
  directly follows from Theorems~\ref{thm_RBSP_interval_adversary}~and~\ref{thm_RBSP_approx_adversary}.
  Analogously to Corollary~\ref{cor_RBSP_disj_interval_leader},
  it can be implemented in a faster running time of $O(n^2 \log n)$
  by making use of a preprocessing.
  Similarly to the proof of Corollary~\ref{cor_RBSP_discrete_approx},
  the precomputed sets the follower's solutions are derived from
  might be affected by the different leader's solutions here
  and therefore need to be updated
  in each iteration.
  
  Theorem~\ref{thm_RBSP_discrete_uncorr_equiv} implies that
  the statement
  also holds in the setting of discrete uncorrelated uncertainty.
\end{proof}

The following examples
show that
the approximation factor of~2 is actually tight for the given algorithm
in case of discrete uncertainty and interval uncertainty.

\begin{example} \label{ex_approx_discrete}
  Consider the following family of instances of the \RBSP{} with discrete uncertainty:
  Let $n \in \N$ with $n \ge 4$ and $\varepsilon \in \Qgt$,
  and set $b = n - 2$.
  Define the item sets as $\Ell = \{e_1, e_2\}$ and $\Eff = \{e_1, e_2, e_3, \dots, e_n\}$,
  with leader's item costs~$c(e_1) = 1 - \varepsilon$, $c(e_2) = 1$, $c(e_3) = 3$, and $c(e) = 0$ for all other items~$e \in \{e_4, \dots, e_n\}$,
  and follower's item costs according to the two scenarios~$\U = \{d_1, d_2\}$
  such that the follower's greedy order in scenario~$d_1$ is $e_4, \dots, e_n, e_3, e_1, e_2$,
  and in scenario~$d_2$ it is $e_2, e_4, \dots, e_n, e_1, e_3$.
  Recall that it suffices to define the scenarios via the follower's greedy orders
  because the actual costs in the follower's objective are not important for the leader;
  see Remark~\ref{rem_opt_pess}.
  
  The algorithm from Corollary~\ref{cor_RBSP_discrete_approx}
  considers the three leader's selections~$X_0 = \emptyset$, $X_1 = \{e_1\}$, and $X_2 = \{e_1, e_2\}$.
  They lead to worst-case follower's responses $Y_0 = \{e_4, \dots, e_n, e_3\}$ (in scenario~$d_1$), $Y_1 = \{e_2, e_4, \dots, e_{n-1}\}$ (in scenario~$d_2$), and $Y_2 = \{e_4, \dots, e_{n-1}\}$ (in both scenarios), respectively.
  The resulting leader's objective function values are~$3$, $2 - \varepsilon$, and $2 - \varepsilon$.
  Hence, the leader will choose either $X_1$ or $X_2$ and achieve a cost value of $2 - \varepsilon$.
  The optimal leader's solution $X^* = \{e_2\}$, however,
  leads to the worst-case follower's response~$Y^* = \{e_4, \dots, e_{n-1}\}$ (in both scenarios)
  and therefore to a leader's objective value of only~$1$.
  This shows that,
  for any number of items,
  the solution returned by the algorithm
  might achieve an objective value of almost twice the optimal value.
  Depending on the choice
  the algorithm makes in case the greedy order of the leader's items is not unique,
  the same behavior might even occur for $\varepsilon = 0$,
  which would deliver a factor of exactly~$2$.
\end{example}

\begin{example} \label{ex_approx_interval}
  Consider the family of instances from Example~\ref{ex_approx_discrete} again,
  but instead of the discrete uncertainty set~$\U = \{d_1, d_2\}$,
  we now work with an interval uncertainty set~$\U$
  according to which the following three follower's greedy orders are possible:
  $e_4, \dots, e_{n - 1}, e_2, e_n, e_3, e_1$,
  or $e_4, \dots, e_{n - 1}, e_n, e_2, e_3, e_1$,
  or $e_4, \dots, e_{n - 1}, e_n, e_3, e_2, e_1$.
  This can be achieved by fixing the follower's item costs of all items except for $e_2$
  (i.e., using intervals of length~$0$)
  and defining an appropriate interval for the cost of $e_2$.
  Analogously to Example~\ref{ex_approx_discrete},
  it can be easily checked that the algorithm returns either $X_1 = \{e_1\}$ or $X_2 = \{e_1, e_2\}$
  as a leader's solution, each achieving a leader's objective function value of $2 - \varepsilon$,
  while the optimal leader's solution~$X^* = \{e_2\}$
  leads to a cost value of only~$1$.
\end{example}

\subsubsection{Exact Algorithms} \label{sec_leader_general_exact}

We now turn towards
exact algorithms for the general \RBSP{}
that have exponential running time
and can be interesting from the perspective of parameterized complexity
(see, e.g.,~\cite{cyganetal2015}).

Clearly,
an exact algorithm can be achieved
by an enumeration approach:

\begin{theorem} \label{thm_RBSP_exact}
  Consider the \RBSP{},
  with any type of uncertainty set.
  Suppose that the corresponding adversary's problem can be solved in time at most~$A(I)$,
  given an instance~$I$ of the \RBSP{}
  together with any leader's solution.
  Then the \RBSP{},
  given an instance~$I$ with $\nll$ leader's items and capacity~$b$,
  can be solved in time $O(\min\{2^\nll, \nllb\} A(I))$.
\end{theorem}

\begin{proof}
  The \RBSP{} can be solved by enumerating all feasible leader's solutions~$X \subseteq \Ell$
  and solving the adversary's problem for each of them.
  The latter is done by an algorithm specific for the given type of uncertainty set.
  This algorithm for the adversary's problem
  can be assumed to return a corresponding follower's solution~$Y \subseteq \Eff$ as well.
  (If not, we can solve the follower's problem in time $O(n)$
  for the given leader's solution and adversary's choice.)
  By comparing the leader's costs of all resulting solutions~$X \cup Y$,
  an optimal leader's solution can be determined.

  For the running time,
  note that the number of feasible leader's solutions~$X \subseteq \Ell$
  is bounded by~$2^\nll$ and,
  if~$b \le \nll$,
  by~$\sum_{\bll = 0}^b \binom{\nll}{\bll} \in O(\nllb)$
  because the leader can select at most $b$~items in a feasible solution.
\end{proof}

This implies that the \RBSP{} is fixed-parameter tractable (FPT) in the parameter~$\nll$
as well as slice-wise polynomial (XP) with respect to the parameter~$b$,
whenever the adversary's problem can be solved fast enough
(e.g., in polynomial time).
For the uncertainty sets
that we have investigated in Section~\ref{sec_adversary}
(see Theorems~\ref{thm_RBSP_discrete_adversary},~\ref{thm_RBSP_interval_adversary}, and~\ref{thm_RBSP_discrete_uncorr_equiv}),
we obtain:

\begin{corollary} \label{cor_RBSP_discrete_exact}
  The \RBSP{} with a discrete uncertainty set~$\U$,
  can be solved in time~$O(\min\{2^\nll, \nllb\} \lvert \U \rvert n)$.
\end{corollary}

\begin{corollary} \label{cor_RBSP_interval_exact}
  The \RBSP{} with interval uncertainty or discrete uncorrelated uncertainty
  can be solved in time~$O(\min\{2^\nll, \nllb\} n^2)$.
\end{corollary}

For discrete uncertainty,
we also derive a more complex algorithm
which proves that
the problem is slice-wise polynomial (XP) in the parameter~$\lvert \U \rvert$,
i.e., for every constant number of scenarios,
the problem is solvable in polynomial time;
see Theorem~\ref{thm_RBSP_discrete_exact_u}.
We develop the idea of the algorithm in the following.
A detailed description can also be found in~\cite{hartmann2021}.
We assume here that
$\Ell \subseteq \Eff$,
which is not a restriction;
see Section~\ref{sec_BSP_var_sets}.

To illustrate the idea,
we first make some observations
about the \BSP{} without uncertainty,
and with the assumption~$\Ell \subseteq \Eff$.
Recall that,
from the follower's perspective,
a greedy order of his items~$\Eff$ can be fixed,
independently of the leader's choice,
and we may assume that the follower always selects a prefix of this order,
possibly after removing the items that have already been selected by the leader.
We utilized this also in Algorithm~\ref{alg_BSP}.
Hence,
every overall solution~$X \cup Y$ can be assumed to consist of
a prefix of the follower's greedy order
(some part of which might have been selected by the leader and the rest by the follower)
and possibly some single leader's items
that appear later in the follower's greedy order.
Therefore,
we could also find a feasible overall solution~$X \cup Y$
that is best possible for the leader
as follows:
Enumerate all prefixes~$\overline{Y} \subseteq \Eff$
of the follower's greedy order
and, for each of them,
check whether the leader can select some solution~$X \subseteq \Ell$
such that, together with the follower's response~$Y$ to $X$,
the overall solution~$X \cup Y$ consists of the prefix~$\overline{Y}$
and $b - \lvert \overline{Y} \rvert$ additional leader's items.
If this is possible,
find the best such leader's choice~$X$ according to the leader's objective~$c$.
The best solution obtained in this enumeration process
must be an optimal leader's solution.

The subproblem for a fixed prefix~$\overline{Y} \subseteq \Eff$
is easy to solve:
Clearly,
the leader must select $b - \lvert \overline{Y} \rvert$ items from $\Ell \setminus \overline{Y}$.
If this set does not contain sufficiently many items,
then there is no feasible solution
of the desired structure
corresponding to the prefix~$\overline{Y}$.
In addition,
the leader may select any number of items from $\Ell \cap \overline{Y}$,
but this has no influence on the overall solution~$X \cup Y$
because the follower's optimal response is always $Y = \overline{Y} \setminus X$.
The selection of $b - \lvert \overline{Y} \rvert$ items from $\Ell \setminus \overline{Y}$
can be made in a greedy way according to the leader's objective~$c$
because there is no difference in the follower's reaction among different such selections.
Therefore,
each iteration in the above algorithm boils down to a single-level \SP{}.

In fact,
this algorithm is very similar to Algorithm~\ref{alg_BSP}
because it also enumerates greedy choices of different sizes,
for the leader and for the follower.
Due to the different perspective of starting with the follower's instead of the leader's prefixes,
only the items that could be selected by either player
are handled differently:
Observe that the leader's greedy solutions that are enumerated in Algorithm~\ref{alg_BSP}
correspond to the solutions of the single-level \SP{}s here,
together with a specific additional choice from $\Ell \cap \overline{Y}$,
which does not influence the overall solution.

We now turn to the robust setting again.
The different scenarios now correspond to different follower's greedy orders,
and the adversary may choose between them.
However,
in any given scenario,
the above structural observations still apply,
and a generalization of the enumeration procedure described above
now solves the \RBSP{}.
Instead of one prefix~$\overline{Y} \subseteq \Eff$ of a single follower's greedy order,
we now consider in each iteration
a combination of prefixes of the greedy orders in all scenarios,
i.e., some collection~$(\overline{Y_d})_{d \in \U}$ such that $\overline{Y_d} \subseteq \Eff$
is a prefix of the follower's greedy order in scenario~$d$, for all $d \in \U$.
The problem that has to be solved in a given iteration
is then to find the best leader's selection~$X \subseteq \Ell$
(or decide that there is no such~$X$)
such that,
for every scenario~$d \in \U$,
the overall solution~$X \cup Y_d$ in this scenario
(i.e., $Y_d$ is the follower's optimal solution in scenario~$d$)
consists of the prefix~$\overline{Y_d}$ of the follower's greedy order corresponding to this scenario
and $b - \lvert \overline{Y_d} \rvert$ additional leader's items.

This subproblem can now be solved by another enumeration
and several single-level \SP{}s on appropriate subsets of~$\Ell$.
We illustrate this for the setting of two scenarios~$\U = \{d_1, d_2\}$:
Given prefixes~$\overline{Y_{d_1}}, \overline{Y_{d_2}} \subseteq \Eff$,
in analogy to the situation without uncertainty above,
the leader has to select $b - \lvert \overline{Y_{d_1}} \rvert$ items from $\Ell \setminus \overline{Y_{d_1}}$
and $b - \lvert \overline{Y_{d_2}} \rvert$ items from $\Ell \setminus \overline{Y_{d_2}}$.
Since these two selections are not independent in general,
we reformulate them as follows.
If the leader selects $b_1$ items from $\Ell \setminus (\overline{Y_{d_1}} \cup \overline{Y_{d_2}})$,
$b_2$ items from $(\Ell \cap \overline{Y_{d_1}}) \setminus \overline{Y_{d_2}}$,
$b_3$ items from $(\Ell \cap \overline{Y_{d_2}}) \setminus \overline{Y_{d_1}}$,
and $b_4$ items from $\Ell \cap \overline{Y_{d_1}} \cap \overline{Y_{d_1}}$,
then this leads to the desired situation
if and only if $b_1 + b_3 = b - \lvert \overline{Y_{d_1}} \rvert$
and $b_1 + b_2 = b - \lvert \overline{Y_{d_2}} \rvert$.
Note that it does not matter
how many items the leader selects that are contained in all considered prefixes,
i.e., what the number~$b_4$ is.
The subproblem can now be solved by enumerating
all combinations of values for $b_1$, $b_2$, and $b_3$
that satisfy these conditions
and solving the three corresponding single-level \SP{}s.

It remains to analyze the running time of the algorithm.
In the beginning,
the follower's greedy orders in all scenarios are computed in time~$O(\lvert \U \rvert n \log n)$.
In the outer enumeration,
we consider $O(b)$ possible prefixes for each scenario,
which gives $O(b^{\lvert \U \rvert})$ combinations $(\overline{Y_d})_{d \in \U}$ of prefixes in total.
In the subproblem,
the enumeration comprises~$O(2^{\lvert \U \rvert})$ values~$b_i$,
one for each subset $\overline{\U} \subseteq \U$ of the scenarios,
describing the number of items the leader selects
from the set $(\Ell \cap \bigcap_{d \in \U \setminus \overline{\U}} \overline{Y_d}) \setminus \bigcup_{d \in \overline{\U}} \overline{Y_d}$.
Each value~$b_i$ can attain $O(b)$ many different values,
resulting in $O(b^{2^{\lvert \U \rvert}})$ cases that have to be checked.
For fixed values~$b_i$,
we check $O(\lvert \U \rvert)$ conditions on sums of these values
and, if they are satisfied,
we solve $O(2^{\lvert \U \rvert})$ many single-level \SP{}s,
each of them in time~$O(n)$.
In order to evaluate the leader's objective function value of such a solution,
we can compute its leader's costs in all scenarios in $O(\lvert \U \rvert n)$.

In summary,
the described algorithm gives the following result:

\begin{theorem} \label{thm_RBSP_discrete_exact_u}
  The \RBSP{} with a discrete uncertainty set~$\U$
  of constant size~$\lvert \U \rvert$
  can be solved in time polynomial in $n$.
\end{theorem}

\section{Problem Variants with Continuous Variables} \label{sec_cont}

As already mentioned in Section~\ref{sec_BSP_var_cont},
we now study variants of the \BSP{} with continuous decisions.
As a combinatorial optimization problem,
the \BSP{} can be formulated in terms of binary variables
in a straightforward way,
and it could be interesting to consider also the corresponding variant with continuous variables,
i.e., the setting in which the leader and/or the follower
are allowed to select fractions of items.
Corresponding to the formulation in~\eqref{eq_BSP},
we can write the (basic) continuous variant of the \BSP{} as follows:
\begin{equation} \label{eq_CBSP} \tag{CBSP}
\begin{aligned}
  \min_x~ & \sum_{e \in \Ell} c(e) x_e + \sum_{e \in \Eff} c(e) y_e\\
  \st~ & x \in [0, 1]^\Ell \\
  & y \in
  \begin{aligned}[t]
    \argmin_{y'}~ & \sum_{e \in \Eff} d(e) y'_e \\
    \st~ & x_e + y'_e \le 1 & \forall e \in \Ell \cap \Eff \\
    & \sum_{e \in \Ell} x_e + \sum_{e \in \Eff} y'_e = b \\
    & y' \in [0, 1]^\Eff
  \end{aligned}
\end{aligned}
\end{equation}
Observe that
the version of~\eqref{eq_CBSP} with binary instead of continuous variables
is equivalent to~\eqref{eq_BSP}.
Note that the constraint that the follower's solution must be disjoint from the leader's one
($Y' \subseteq \Eff \setminus X$ in~\eqref{eq_BSP})
is modeled by the inequalities~$x_e + y'_e \le 1$ for all~$e \in \Ell \cap \Eff$ here.

As explained in Section~\ref{sec_BSP_var_sets},
the bilevel problem where the leader sets the capacity for the follower's \SP{}
can be modeled as a special case of the \BSP{},
and this is also true in the continuous variant.
Here, the resulting problem can be seen as a special case of the \BCKP{}
that is considered in \cite{buchheimhenke2022}.
Accordingly,
most of the insights presented in this section
are generalizations and variations
of results in \cite{buchheimhenke2022}.
However,
we also emphasize an important difference in complexity
between the \RCBSP{}
and the \RBCKP{}
for the setting of discrete uncorrelated uncertainty;
see Section~\ref{sec_RCBSP_discrete_uncorr}.

Observe that
the follower's problem in~\eqref{eq_CBSP},
for a fixed feasible leader's solution~$x \in [0, 1]^\Ell$,
corresponds to a \CSP{}
(in case of disjoint item sets $\Ell$ and $\Eff$)
or possibly to a \CKP{}
(in the general case)
because the leader might have selected fractions of some follower's items,
hence leaving items of different reduced sizes
for the follower to choose from.
However,
in both cases,
the follower can still be assumed
to solve his problem greedily
and select items
according to a fixed order,
i.e., he selects the largest prefix of this order
that fits entirely into the given capacity
and possibly a fraction of the next item
such that the capacity is filled completely.
Hence,
such a greedy solution
can be assumed to select at most one item fractionally.
In case of a \CKP{},
the greedy order is determined by
the ratios of item costs and item sizes~\cite{dantzig1957}.
For more details on the situation of a \CKP{} in the follower's problem,
we also refer to~\cite{buchheimhenke2022}.

In addition to the original \BSP{}
and the continuous version~\eqref{eq_CBSP},
two further problem variants can be obtained by
relaxing the integrality
only in the leader's or only in the follower's decision.
In the following,
we will compare all four versions of the \BSP{}.

First note that
the setting in which the leader takes a binary decision
and the follower is allowed to select fractions of items
is in fact equivalent to the original \BSP{}:
From the follower's perspective,
an integer number of items is left to select
and none of his items has been selected fractionally by the leader.
Following the greedy approach,
the follower will therefore
select a binary solution.

We now turn to the problem version
in which the leader can select fractions of items,
but the follower has binary variables.
First focus on the special case of disjoint item sets again.
As in the original \BSP{}
(see Section~\ref{sec_BSP_alg}),
the only relevant interaction between leader and follower
then is the number $\bll = \sum_{e \in \Ell} x_e$ of items the leader selects.
Although $\bll$ is not automatically an integer here,
the leader has to choose her solution
such that $\bll$ is an integer
in order to enable the follower to complete the selection
to exactly $b$ items in total.
As before,
it is optimal for the leader and for the follower
to select items greedily,
given a fixed partition of the total capacity,
in case of disjoint item sets.
Analogously to above,
this means that the leader can be assumed to take a binary decision.
Thus,
if the leader's and the follower's item sets are disjoint,
allowing only the leader to select fractions of items
does not change the problem compared to the original variant.

If the item sets are not disjoint,
the situation is different.
In fact,
we might obtain a bilevel optimization problem
whose optimal value is not attained;
see Example~\ref{ex_cont_bin_no_opt}.
This is an interesting property
that (mixed-integer) bilevel optimization problems
can have,
even if the corresponding single-level problems
do not indicate such a behavior.
The reason here is that
the binary follower's problem
might introduce an unnatural discontinuity into the leader's objective.
We refer to, e.g., \cite{vicentesavardjudice1996}
for more examples of such a behavior
in bilevel optimization problems with continuous leader's variables
and integer follower's variables.

\begin{example} \label{ex_cont_bin_no_opt}
  Consider the problem~\eqref{eq_CBSP}
  with continuous leader's, but binary follower's variables.
  Let~$\Ell = \{e_1, e_2\}$, $\Eff = \{e_1, e_2, e_3\}$, $b = 2$, $c(e_1) = 1$, $c(e_2) = c(e_3) = 0$, $d(e_1) = 0$, $d(e_2) = d(e_3) = 1$.
  For the leader,
  the overall solution consisting of~$e_2$ and~$e_3$ would be best possible,
  with total costs of~$0$.
  However,
  a solution without~$e_1$ is not possible
  because the follower prefers~$e_1$ over~$e_2$ and~$e_3$
  and the leader cannot select~$e_3$ herself.
  But,
  as the leader is allowed to make continuous decisions and the follower is not,
  it suffices that
  the leader selects a small fraction of~$e_1$
  in order to ``block'' it for the follower.
  However,
  the leader has to select an integer number of items in total
  in order to ensure feasibility.
  This can be achieved by selecting
  a corresponding fraction of~$e_2$.
  More precisely,
  for every~$\varepsilon \in (0,1)$,
  the leader's solution~$x = (\varepsilon, 1 - \varepsilon, 0)$
  induces the follower's response~$y = (0, 0, 1)$
  and hence a leader's objective value of~$\varepsilon$.
  Thus,
  it can get arbitrarily close to~$0$, but cannot attain~$0$.
  This implies that
  the problem does not have an optimal leader's solution.
\end{example}

Because of the above observations,
the setting in which both decision makers
take a continuous decision
is the most interesting to consider further.
Therefore,
we focus on the setting of~\eqref{eq_CBSP} again
for the remainder of this section.

In fact,
for the setting without uncertainty,
we argue that
the continuous problem~\eqref{eq_CBSP} is equivalent to the original \BSP{},
while this will not be the case anymore
when we add robustness;
see Example~\ref{ex_RCBSP_not_equiv}.
To show the equivalence of the two problem versions
without uncertainty,
first consider the special case of disjoint item sets again.
As before,
the leader only influences the follower
through the (not necessarily integer) number $\bll = \sum_{e \in \Ell} x_e$ of items she selects
and, for fixed~$\bll$,
it is optimal for each player
to solve the corresponding single-level \CSP{}
on their own item set.
Hence,
every possible such selection of leader's and follower's items
can be seen as a convex combination
of two integer solutions.
This will also be illustrated in Section~\ref{sec_CBSP_PLF}.
Thus,
there is always an optimal solution that is integer.
If the item sets are not disjoint,
a similar argument can be applied
because, in an optimal solution, we may assume that
the sets of items the two players select (possibly fractionally)
are disjoint.

For the remainder of this section,
we focus on the presumably easier special case of $\Ell \cap \Eff = \emptyset$ again.
From now on,
we will also denote the problem~\eqref{eq_CBSP}
with this assumption
as the \CBSP{}.

We will give a better understanding
of the structure of the \CBSP{} in Section~\ref{sec_CBSP_PLF},
using piecewise linear functions.
Based on this,
we investigate the robust problem in Section~\ref{sec_RCBSP}.
Here,
we show that the problem
is not equivalent to the binary problem version
in the sense we described it above for the setting without uncertainty.
However,
we also show that the robust problem
with discrete uncertainty, interval uncertainty, and discrete uncorrelated uncertainty
can be solved in polynomial time.

\subsection{Structure of the Leader's Objective Function} \label{sec_CBSP_PLF}

We have already argued above that
the \CBSP{} (without uncertainty)
is equivalent to the \BSP{}.
In the case of disjoint item sets,
this can also be understood
by investigating the structure of the leader's objective function in more detail.
Moreover,
the following insights will be useful in view of the robust problem versions
we will study in Section~\ref{sec_RCBSP}.

As before,
in case of disjoint item sets,
it is clear that,
for any fixed splitting of the capacity~$b$
into capacities~$\bll$ and~$\bff$
(which are not necessarily integers now)
for the leader and the follower, respectively,
it is optimal for each of the players
to choose their solution greedily.
We will now look at the leader's objective function value
resulting from any feasible choice of~$\bll$.

\begin{figure}
  \begin{tabular}{ccc}
  \begin{subfigure}[t]{0.31\textwidth}
    \centering
    \begin{tikzpicture}[scale=0.6]
      \draw[-latex] (-0.1,0) node[left] {$0$} -- (5.5,0) node[right] {$\bll$};
      \draw[-latex] (0,-4.5) -- (0,2.5) node[above] {$c(x)$};

      \draw[dotted] (1,-4.5) -- (1,2.5);
      \draw[dotted] (2,-4.5) -- (2,2.5);
      \draw[dotted] (3,-4.5) -- (3,2.5);
      \draw[dotted] (4,-4.5) -- (4,2.5);
      \draw[dotted] (5,-4.5) -- (5,2.5);

      \draw (1,-0.1) node[below] {$1$} -- (1,0.1);
      \draw (2,-0.1) node[below] {$2$} -- (2,0.1);
      \draw (3,-0.1) node[below] {$3$} -- (3,0.1);
      \draw (4,-0.1) node[below] {$4$} -- (4,0.1);
      \draw (5,-0.1) node[below] {$5$} -- (5,0.1);

      \draw (-0.1,-4) node[left] {$-4$} -- (0.1,-4);
      \draw (-0.1,-3) node[left] {$-3$} -- (0.1,-3);
      \draw (-0.1,-2) node[left] {$-2$} -- (0.1,-2);
      \draw (-0.1,-1) node[left] {$-1$} -- (0.1,-1);
      \draw (-0.1,1) node[left] {$1$} -- (0.1,1);
      \draw (-0.1,2) node[left] {$2$} -- (0.1,2);

      \draw[dashed] (0,0) -- (1,-1);
      \draw (1,-1) -- (2,-2) -- (3,-2) -- (4,1);
    \end{tikzpicture}
    \subcaption{Leader's objective of her own items,
      $\PLFl(\Ell, c \restriction \Ell, \bllm, \bllp)$}
    \label{fig_CBSP_leader_obj_leader_own}
  \end{subfigure}
  &
  \begin{subfigure}[t]{0.31\textwidth}
    \centering
    \begin{tikzpicture}[scale=0.6]
      \draw[-latex] (-0.1,0) node[left] {$0$} -- (5.5,0) node[right] {$\bff$};
      \draw[-latex] (0,-4.5) -- (0,2.5) node[above] {$d(y)$};

      \draw[dotted] (1,-4.5) -- (1,2.5);
      \draw[dotted] (2,-4.5) -- (2,2.5);
      \draw[dotted] (3,-4.5) -- (3,2.5);
      \draw[dotted] (4,-4.5) -- (4,2.5);
      \draw[dotted] (5,-4.5) -- (5,2.5);

      \draw (1,-0.1) node[below] {$1$} -- (1,0.1);
      \draw (2,-0.1) node[below] {$2$} -- (2,0.1);
      \draw (3,-0.1) node[below] {$3$} -- (3,0.1);
      \draw (4,-0.1) node[below] {$4$} -- (4,0.1);
      \draw (5,-0.1) node[below] {$5$} -- (5,0.1);

      \draw (-0.1,-4) node[left] {$-4$} -- (0.1,-4);
      \draw (-0.1,-3) node[left] {$-3$} -- (0.1,-3);
      \draw (-0.1,-2) node[left] {$-2$} -- (0.1,-2);
      \draw (-0.1,-1) node[left] {$-1$} -- (0.1,-1);
      \draw (-0.1,1) node[left] {$1$} -- (0.1,1);
      \draw (-0.1,2) node[left] {$2$} -- (0.1,2);
      
      \draw[dashed] (0,0) -- (1,-2);
      \draw (1,-2) -- (2,-2) -- (3,-1) -- (4,0);
    \end{tikzpicture}
    \subcaption{Follower's objective of his own items,
      $\PLFl(\Eff, (d, -c \restriction \Eff), \bffm, \bffp)$}
    \label{fig_CBSP_leader_obj_follower_own}
  \end{subfigure}
  &
  \begin{subfigure}[t]{0.31\textwidth}
    \centering
    \begin{tikzpicture}[scale=0.6]
      \draw[-latex] (-0.1,0) node[left] {$0$} -- (5.5,0) node[right] {$\bff$};
      \draw[-latex] (0,-4.5) -- (0,2.5) node[above] {$c(y)$};

      \draw[dotted] (1,-4.5) -- (1,2.5);
      \draw[dotted] (2,-4.5) -- (2,2.5);
      \draw[dotted] (3,-4.5) -- (3,2.5);
      \draw[dotted] (4,-4.5) -- (4,2.5);
      \draw[dotted] (5,-4.5) -- (5,2.5);

      \draw (1,-0.1) node[below] {$1$} -- (1,0.1);
      \draw (2,-0.1) node[below] {$2$} -- (2,0.1);
      \draw (3,-0.1) node[below] {$3$} -- (3,0.1);
      \draw (4,-0.1) node[below] {$4$} -- (4,0.1);
      \draw (5,-0.1) node[below] {$5$} -- (5,0.1);

      \draw (-0.1,-4) node[left] {$-4$} -- (0.1,-4);
      \draw (-0.1,-3) node[left] {$-3$} -- (0.1,-3);
      \draw (-0.1,-2) node[left] {$-2$} -- (0.1,-2);
      \draw (-0.1,-1) node[left] {$-1$} -- (0.1,-1);
      \draw (-0.1,1) node[left] {$1$} -- (0.1,1);
      \draw (-0.1,2) node[left] {$2$} -- (0.1,2);

      \draw[dashed] (0,0) -- (1,1);
      \draw (1,1) -- (2,-2) -- (3,0) -- (4,-1);
    \end{tikzpicture}
    \subcaption{Leader's objective of the follower's items, depending on $\bff$}
    \label{fig_CBSP_leader_obj_leader_followers_bf}
  \end{subfigure}
  \\
  \begin{subfigure}[t]{0.31\textwidth}
    \centering
    \begin{tikzpicture}[scale=0.6]
      \draw[-latex] (-0.1,0) node[left] {$0$} -- (5.5,0) node[right] {$\bll$};
      \draw[-latex] (0,-4.5) -- (0,2.5) node[above] {$c(y)$};

      \draw[dotted] (1,-4.5) -- (1,2.5);
      \draw[dotted] (2,-4.5) -- (2,2.5);
      \draw[dotted] (3,-4.5) -- (3,2.5);
      \draw[dotted] (4,-4.5) -- (4,2.5);
      \draw[dotted] (5,-4.5) -- (5,2.5);

      \draw (1,-0.1) node[below] {$1$} -- (1,0.1);
      \draw (2,-0.1) node[below] {$2$} -- (2,0.1);
      \draw (3,-0.1) node[below] {$3$} -- (3,0.1);
      \draw (4,-0.1) node[below] {$4$} -- (4,0.1);
      \draw (5,-0.1) node[below] {$5$} -- (5,0.1);

      \draw (-0.1,-4) node[left] {$-4$} -- (0.1,-4);
      \draw (-0.1,-3) node[left] {$-3$} -- (0.1,-3);
      \draw (-0.1,-2) node[left] {$-2$} -- (0.1,-2);
      \draw (-0.1,-1) node[left] {$-1$} -- (0.1,-1);
      \draw (-0.1,1) node[left] {$1$} -- (0.1,1);
      \draw (-0.1,2) node[left] {$2$} -- (0.1,2);

      \draw[dashed] (5,0) -- (4,1);
      \draw (4,1) -- (3,-2) -- (2,0) -- (1,-1);
    \end{tikzpicture}
    \subcaption{Leader's objective of the follower's items, depending on $\bll$,\newline
      $\PLFf(\Eff, (d, -c \restriction \Eff), c \restriction \Eff, \bllm, \bllp, b)$}
    \label{fig_CBSP_leader_obj_leader_followers_bl}
  \end{subfigure}
  &
  \begin{subfigure}[t]{0.31\textwidth}
    \centering
    \begin{tikzpicture}[scale=0.6]
      \draw[-latex] (-0.1,0) node[left] {$0$} -- (5.5,0) node[right] {$\bll$};
      \draw[-latex] (0,-4.5) -- (0,2.5) node[above] {$c(x) + c(y)$};

      \draw[dotted] (1,-4.5) -- (1,2.5);
      \draw[dotted] (2,-4.5) -- (2,2.5);
      \draw[dotted] (3,-4.5) -- (3,2.5);
      \draw[dotted] (4,-4.5) -- (4,2.5);
      \draw[dotted] (5,-4.5) -- (5,2.5);

      \draw (1,-0.1) node[below] {$1$} -- (1,0.1);
      \draw (2,-0.1) node[below] {$2$} -- (2,0.1);
      \draw (3,-0.1) node[below] {$3$} -- (3,0.1);
      \draw (4,-0.1) node[below] {$4$} -- (4,0.1);
      \draw (5,-0.1) node[below] {$5$} -- (5,0.1);

      \draw (-0.1,-4) node[left] {$-4$} -- (0.1,-4);
      \draw (-0.1,-3) node[left] {$-3$} -- (0.1,-3);
      \draw (-0.1,-2) node[left] {$-2$} -- (0.1,-2);
      \draw (-0.1,-1) node[left] {$-1$} -- (0.1,-1);
      \draw (-0.1,1) node[left] {$1$} -- (0.1,1);
      \draw (-0.1,2) node[left] {$2$} -- (0.1,2);

      \draw (1,-2) -- (2,-2) -- (3,-4) -- (4,2);
    \end{tikzpicture}
    \subcaption{Leader's total objective
      (sum of the functions in Figures~\ref{fig_CBSP_leader_obj_leader_own} and \ref{fig_CBSP_leader_obj_leader_followers_bl})}
    \label{fig_CBSP_leader_obj_leader_total}
  \end{subfigure}
  \end{tabular}
  \caption{Example of leader's and follower's objective functions
    of the \CBSP{},
    depending on the splitting of the total capacity~$b$
    into leader's and follower's capacities~$\bll$ and~$\bff$, respectively.
    Let $\Ell = \{e_1, e_2, e_3, e_4\}$,
    $\Eff = \{e_5, e_6, e_7, e_8\}$,
    $b = 5$,
    $c(e_1) = -1$, $c(e_2) = -1$, $c(e_3) = 0$, $c(e_4) = 3$,
    $c(e_5) = 1$, $c(e_6) = -3$, $c(e_7) = 2$, $c(e_8) = -1$,
    $d(e_5) = -2$, $d(e_6) = 0$, $d(e_7) = 1$, and $d(e_8) = 1$.
    We adopt the pessimistic setting,
    which is reflected in the order of the items $e_7$ and $e_8$.\\
    The terms $c(x)$, $c(y)$, and $d(y)$ are used as shortcuts
    for the objective function values $\sum_{e \in \Ell} c(e) x_e$, $\sum_{e \in \Eff} c(e) y_e$, and $\sum_{e \in \Eff} d(e) y_e$,
    where $x \in [0, 1]^\Ell$ and $y \in [0, 1]^\Eff$
    are optimal leader's and follower's solutions
    for given values of~$\bll$ and~$\bff$, respectively.\\
    The dashed linear pieces are parts of the functions
    that correspond to infeasible overall solutions
    because both leader and follower have to select at least one item
    in order to reach the total desired capacity of $b = 5$.
    More formally,
    we have $\bllm = \bffm = 1$ and $\bllp = \bffp = 4$ here.
  }
  \label{fig_CBSP_leader_obj}
\end{figure}

First focus on the leader
only solving her own \CSP{} on~$\Ell$
with a varying capacity~$\bll \in [\bllm,\bllp]$,
where~$\bllm = \max\{0, b - \nff\}$ and~$\bllp = \min\{b, \nll\}$ are defined
as in Lines~\ref{alg_BSP_bl_lb} and~\ref{alg_BSP_bl_ub} of Algorithm~\ref{alg_BSP}.
Note that,
as in the original \BSP{},
feasible solutions for the bilevel problem arise from exactly these choices of~$\bll$.
The order of the items in~$\Ell$
can be assumed to be fixed,
independent of~$\bll$,
and given by sorting the item costs~$c(e)$ non-decreasingly.
For any~$\bll$,
the leader selects a \emph{fractional prefix}
of this order,
i.e., she selects the first~$\lfloor \bll \rfloor$ items of the order completely
and a fraction $\bll - \lfloor \bll \rfloor$ of the next one.
For the total costs of the selected items,
this means that
it changes linearly between any two integer values of~$\bll$,
where the slope of this linear piece is given by the current fractional item's cost.
At an integer~$\bll$,
the function proceeds continuously,
but the slope changes to the next item's cost.
This results in a continuous piecewise linear function
with vertices at integer values of~$\bll$ and
linear pieces corresponding to the items.
Because of the greedy order of the items,
the slopes of the linear pieces are sorted non-decreasingly.
Hence,
the function is convex.
An example of such a function is displayed in Figure~\ref{fig_CBSP_leader_obj_leader_own}.
We denote the function
that is constructed in this way
by~$\PLFl(\Ell, c \restriction \Ell, \bllm, \bllp)$.
The data that are required to define it
are the set~$\Ell$ of items,
the function $c \colon \Ell \to \Q$ of item costs
that we sort the items by,
and the range $[\bllm, \bllp]$ of the piecewise linear function,
corresponding to the range of capacities that we allow.
When writing $\PLFl(\E, c', \bm, \bp)$
for a finite set~$\E$,
a function~$c' \colon \E \to \Q$,
and numbers~$\bm, \bp \in \Nn$,
we always assume that $0 \le \bm \le \bp \le \lvert \E \rvert$,
i.e., that the function is well-defined on the whole range.
The piecewise linear function~$\PLFl(\E, c', \bm, \bp)$
can be computed in time~$O(n' \log n')$,
where $n' = \lvert \E \rvert$,
because we need to sort the items~$e \in \E$ by their costs~$c'(e)$,
and it can be stored as a list of its vertices,
which are given by the function's values
in the endpoints and the integers in the range~$[\bllm, \bllp]$.

For the follower,
the setting is analogous:
He solves a \CSP{} on~$\Eff$
with varying capacity~$\bff \in [\bffm, \bffp]$,
where $\bffm = b - \bllp = \max\{0, b - \nll\}$ and~$\bffp = b - \bllm = \min\{b, \nff\}$.
We can assume that
the items in~$\Eff$ are sorted
by their follower's item costs $d(e)$ in non-decreasing order
(and, in case of non-uniqueness,
additionally respecting the leader's item costs
according to the optimistic or pessimistic setting;
see Remark~\ref{rem_opt_pess}).
Therefore,
his objective function,
viewed as a function of~$\bff$,
has the same structural properties
as the part of the leader's objective function
that results from the items selected by the leader,
as described above.
For an example,
see Figure~\ref{fig_CBSP_leader_obj_follower_own}.
Following the notation introduced above,
this function can be denoted by~$\PLFl(\Eff, (d, \pm c \restriction \Eff), \bffm, \bffp)$.
By $(d, \pm c \restriction \Eff)$
we mean that the items~$e \in \Eff$ are sorted by their costs~$d(e)$
and, as a secondary criterion,
by the leader's costs~$c(e)$,
in a non-decreasing or a non-increasing order,
in the optimistic or the pessimistic setting,
respectively.

However,
we are mainly interested in the contribution
the follower's solution makes to the leader's objective function.
This means that
we need to take the leader's instead of the follower's item costs
of the follower's items into account
for the leader's objective function,
while their order still depends on the follower's item costs.
This does not change that
the function is continuous and piecewise linear with vertices at integer values of~$\bff$,
but the slopes of the linear pieces now correspond to the leader's item costs.
Therefore,
the function is not convex anymore in general.
See Figure~\ref{fig_CBSP_leader_obj_leader_followers_bf}
for the resulting function in the example.

Moreover,
in order to combine the two parts of the leader's objective function,
we need to use~$\bll = b - \bff$ as a parameter instead of~$\bff$,
also for the function described in the previous paragraph.
This corresponds to mirroring it
such that left and right are swapped.
In terms of the follower's greedy algorithm,
we can imagine that
he starts at~$\bll = \bllm$
by selecting his $\bffp = b - \bllm$ best items
and then removing them greedily in the reverse order.
The transformed function in the example
is displayed in Figure~\ref{fig_CBSP_leader_obj_leader_followers_bl}.
We will denote this function by~$\PLFf(\Eff, (d, \pm c \restriction \Eff), c \restriction \Eff, \bllm, \bllp, b)$.
When writing $\PLFf(\E, (d', \pm c'), c'', \bm, \bp, b')$,
the required data now consists of
the finite set~$\E$ of items and their costs we sort by,
in the form $(d', \pm c')$
for functions~$d', c' \colon \E \to \Q$ as above,
the function~$c'' \colon \E \to \Q$ of item costs
that we use to determine the slopes of the linear pieces,
the range $[\bm, \bp]$ of the piecewise linear function,
with $\bm, \bp \in \Nn$,
and the total capacity~$b' \in \Nn$
that is required to perform the swapping correctly.
In order to get a well-defined function,
we assume $0 \le b' - \bp \le b' - \bm \le \lvert \E \rvert$ here.
As above,
this piecewise linear function can be computed in time~$O(n' \log n')$,
where~$n' = \lvert \E \rvert$.

Finally,
the leader's objective, as a function of~$\bll$,
is given by the sum of the two functions
that describe her objective function values
resulting from the greedy choices of herself and the follower, respectively.
As a sum of two such functions,
it is also a continuous piecewise linear function
with vertices at integer values of~$\bll$.
See Figure~\ref{fig_CBSP_leader_obj_leader_total}
for the total leader's objective function in the example.
Given the two functions
as described above,
their sum can be computed in time~$O(n)$.

Optimal leader's solutions now correspond to minima of this function.
Because of its continuous piecewise linear structure,
the function always has a minimum at an integer value of~$\bll$.
Together with integer optimal solutions
of the resulting leader's and follower's \CSP{}s,
this yields an optimal solution of the \CBSP{}
that is also valid for the binary problem version.

Note that,
also in \cite{buchheimhenke2022},
the leader's objective function
of the \BCKP{} is described as a piecewise linear function.
Since the leader directly controls the follower's capacity in this problem,
only one function as in Figure~\ref{fig_CBSP_leader_obj_leader_followers_bf}
is required,
while we here need to combine it with another piecewise linear function
corresponding to the leader's items.
Recall also the relation of the two problems
described in Section~\ref{sec_BSP_var_sets}.
Another difference is related to
the \CSP{} being the special case of the \CKP{}
where all item sizes are one:
The linear pieces all have a width of~1 here
(i.e., there might be a vertex at every integer value of $\bll$ or $\bff$),
while the widths correspond to the item sizes
in the \BCKP{}.
Accordingly,
the slopes are ratios of item values and item sizes
in the \BCKP{},
but only the item costs here.

\subsection{The Robust Continuous Bilevel Selection Problem} \label{sec_RCBSP}

We now extend the general structural observations of Section~\ref{sec_CBSP_PLF}
to the corresponding robust problem.
We will first describe some general insights
and show in Example~\ref{ex_RCBSP_not_equiv}
that the \RCBSP{} is not equivalent to the binary version,
in contrast to the problem without uncertainty.
Afterwards,
we will investigate the \RCBSP{}
for specific types of uncertainty sets:
In Sections~\ref{sec_RCBSP_discrete}, \ref{sec_RCBSP_interval}, and~\ref{sec_RCBSP_discrete_uncorr},
we give polynomial-time algorithms
for the \RCBSP{}
with discrete uncertainty, interval uncertainty, and discrete uncorrelated uncertainty, respectively.
While the results for discrete and interval uncertainty
make use of similar ideas as before and in~\cite{buchheimhenke2022},
a modified approach is necessary in case of discrete uncorrelated uncertainty.
The latter also provides a pseudopolynomial-time algorithm
for the \RBCKP{} with discrete uncorrelated uncertainty,
which answers an open question stated in~\cite{buchheimhenke2022};
see Appendix~\ref{sec_RCBSP_discrete_uncorr_knapsack}.

Regarding the representation of the leader's objective function
from Section~\ref{sec_CBSP_PLF},
observe that
the function describing the leader's costs of the items selected by herself
is not affected
when introducing uncertainty,
whereas the function describing the follower's selection
can look different in every scenario.
Recall that
the linear pieces of this function correspond to the selected items,
where their slopes are given by the leader's item costs.
Hence,
when the adversary's decision affects the greedy order of the follower's items,
this means that
the order of the linear pieces in the function changes
(while the function still remains continuous).
Extending the example from Figure~\ref{fig_CBSP_leader_obj},
we refer to Figure~\ref{fig_RCBSP_leader_obj_leader_followers_bl_scenario2}
for an example of such a function in a different scenario,
i.e., only the order of the linear pieces is different
compared to the function in Figure~\ref{fig_CBSP_leader_obj_leader_followers_bl}.

For understanding the overall leader's objective function in the robust setting,
remember the order in which the three decisions are made:
First,
the leader determines~$\bll$ and her own selection of items,
which corresponds to the first function
(Figure~\ref{fig_CBSP_leader_obj_leader_own}).
Second, the adversary chooses a scenario
and, with that, one of the other piecewise linear functions
(Figures~\ref{fig_CBSP_leader_obj_leader_followers_bl} and~\ref{fig_RCBSP_leader_obj_leader_followers_bl_scenario2} in our example).
The follower's decision is already implicit
when the adversary's decision is interpreted as
choosing one of the piecewise linear functions.
Since the adversary decides after $\bll$~is fixed,
the adversary's and the follower's contribution
to the leader's objective function
is the pointwise maximum of the piecewise linear functions corresponding to the scenarios;
see Figure~\ref{fig_RCBSP_leader_obj_leader_followers_bl_total}.
The overall leader's objective function can thus be written as
the sum of a piecewise linear function and a pointwise maximum of the other piecewise linear functions;
see Figure~\ref{fig_RCBSP_leader_obj_leader_total}.
Again, optimal leader's solutions correspond to minima of this function.

\begin{figure}
  \begin{tabular}{ccc}
  \begin{subfigure}[t]{0.31\textwidth}
    \centering
    \begin{tikzpicture}[scale=0.6]
      \draw[-latex] (-0.1,0) node[left] {$0$} -- (5.5,0) node[right] {$\bll$};
      \draw[-latex] (0,-4.5) -- (0,2.5) node[above] {$c(y)$};

      \draw[dotted] (1,-4.5) -- (1,2.5);
      \draw[dotted] (2,-4.5) -- (2,2.5);
      \draw[dotted] (3,-4.5) -- (3,2.5);
      \draw[dotted] (4,-4.5) -- (4,2.5);
      \draw[dotted] (5,-4.5) -- (5,2.5);

      \draw (1,-0.1) node[below] {$1$} -- (1,0.1);
      \draw (2,-0.1) node[below] {$2$} -- (2,0.1);
      \draw (3,-0.1) node[below] {$3$} -- (3,0.1);
      \draw (4,-0.1) node[below] {$4$} -- (4,0.1);
      \draw (5,-0.1) node[below] {$5$} -- (5,0.1);

      \draw (-0.1,-4) node[left] {$-4$} -- (0.1,-4);
      \draw (-0.1,-3) node[left] {$-3$} -- (0.1,-3);
      \draw (-0.1,-2) node[left] {$-2$} -- (0.1,-2);
      \draw (-0.1,-1) node[left] {$-1$} -- (0.1,-1);
      \draw (-0.1,1) node[left] {$1$} -- (0.1,1);
      \draw (-0.1,2) node[left] {$2$} -- (0.1,2);

      \draw[dashed] (5,0) -- (4,-1);
      \draw (4,-1) -- (3,0) -- (2,2) -- (1,-1);
    \end{tikzpicture}
    \subcaption{Leader's objective of the follower's items, in scenario~$d_2$,\newline
      $\PLFf(\Eff, (d_2, -c \restriction \Eff), c \restriction \Eff, \bllm, \bllp, b)$}
    \label{fig_RCBSP_leader_obj_leader_followers_bl_scenario2}
  \end{subfigure}
  &
  \begin{subfigure}[t]{0.31\textwidth}
    \centering
    \begin{tikzpicture}[scale=0.6]
      \draw[-latex] (-0.1,0) node[left] {$0$} -- (5.5,0) node[right] {$\bll$};
      \draw[-latex] (0,-4.5) -- (0,2.5) node[above] {$c(y)$};

      \draw[dotted] (1,-4.5) -- (1,2.5);
      \draw[dotted] (2,-4.5) -- (2,2.5);
      \draw[dotted] (3,-4.5) -- (3,2.5);
      \draw[dotted] (4,-4.5) -- (4,2.5);
      \draw[dotted] (5,-4.5) -- (5,2.5);

      \draw (1,-0.1) node[below] {$1$} -- (1,0.1);
      \draw (2,-0.1) node[below] {$2$} -- (2,0.1);
      \draw (3,-0.1) node[below] {$3$} -- (3,0.1);
      \draw (4,-0.1) node[below] {$4$} -- (4,0.1);
      \draw (5,-0.1) node[below] {$5$} -- (5,0.1);

      \draw (-0.1,-4) node[left] {$-4$} -- (0.1,-4);
      \draw (-0.1,-3) node[left] {$-3$} -- (0.1,-3);
      \draw (-0.1,-2) node[left] {$-2$} -- (0.1,-2);
      \draw (-0.1,-1) node[left] {$-1$} -- (0.1,-1);
      \draw (-0.1,1) node[left] {$1$} -- (0.1,1);
      \draw (-0.1,2) node[left] {$2$} -- (0.1,2);

      \draw[dashed] (5,0) -- (4,1);
      \draw (4,1) -- (3,-2) -- (2,0) -- (1,-1);

      \draw[dashed] (5,0) -- (4,-1);
      \draw (4,-1) -- (3,0) -- (2,2) -- (1,-1);

      \draw[very thick, dashed] (5,0) -- (4,1);
      \draw[very thick] (4,1) -- (3.5,-0.5) -- (3,0) -- (2,2) -- (1,-1);
    \end{tikzpicture}
    \subcaption{Leader's objective of the follower's items, in the worst case
    (pointwise maximum of the functions in Figures~\ref{fig_CBSP_leader_obj_leader_followers_bl} and~\ref{fig_RCBSP_leader_obj_leader_followers_bl_scenario2})}
    \label{fig_RCBSP_leader_obj_leader_followers_bl_total}
  \end{subfigure}
  &
  \begin{subfigure}[t]{0.31\textwidth}
    \centering
    \begin{tikzpicture}[scale=0.6]
      \draw[-latex] (-0.1,0) node[left] {$0$} -- (5.5,0) node[right] {$\bll$};
      \draw[-latex] (0,-4.5) -- (0,2.5) node[above] {$c(x) + c(y)$};

      \draw[dotted] (1,-4.5) -- (1,2.5);
      \draw[dotted] (2,-4.5) -- (2,2.5);
      \draw[dotted] (3,-4.5) -- (3,2.5);
      \draw[dotted] (4,-4.5) -- (4,2.5);
      \draw[dotted] (5,-4.5) -- (5,2.5);

      \draw (1,-0.1) node[below] {$1$} -- (1,0.1);
      \draw (2,-0.1) node[below] {$2$} -- (2,0.1);
      \draw (3,-0.1) node[below] {$3$} -- (3,0.1);
      \draw (4,-0.1) node[below] {$4$} -- (4,0.1);
      \draw (5,-0.1) node[below] {$5$} -- (5,0.1);

      \draw (-0.1,-4) node[left] {$-4$} -- (0.1,-4);
      \draw (-0.1,-3) node[left] {$-3$} -- (0.1,-3);
      \draw (-0.1,-2) node[left] {$-2$} -- (0.1,-2);
      \draw (-0.1,-1) node[left] {$-1$} -- (0.1,-1);
      \draw (-0.1,1) node[left] {$1$} -- (0.1,1);
      \draw (-0.1,2) node[left] {$2$} -- (0.1,2);

      \draw (1,-2) -- (2,-1) -- (3,-3) -- (3.5,-1) -- (4,2);
    \end{tikzpicture}
    \subcaption{Leader's total objective
      (sum of the functions in Figures~\ref{fig_CBSP_leader_obj_leader_own} and~\ref{fig_RCBSP_leader_obj_leader_followers_bl_total})}
    \label{fig_RCBSP_leader_obj_leader_total}
  \end{subfigure}
  \end{tabular}
  \caption{Example of the leader's objective function
    of the \RCBSP{}
    with discrete uncertainty,
    depending on the leader's capacity~$\bll$.
    Let $\Ell = \{e_1, e_2, e_3, e_4\}$,
    $\Eff = \{e_5, e_6, e_7, e_8\}$,
    $b = 5$,
    $c(e_1) = -1$, $c(e_2) = -1$, $c(e_3) = 0$, $c(e_4) = 3$,
    $c(e_5) = 1$, $c(e_6) = -3$, $c(e_7) = 2$, $c(e_8) = -1$,
    and $\U = \{d_1, d_2\}$ with
    $d_1(e_5) = -2$, $d_1(e_6) = 0$, $d_1(e_7) = 1$, $d_1(e_8) = 1$,
    $d_2(e_5) = -1$, $d_2(e_6) = 3$, $d_2(e_7) = 0$, and $d_2(e_8) = -3$.
    This is the same instance as in Figure~\ref{fig_CBSP_leader_obj}
    without uncertainty,
    but now with an additional second scenario~$d_2$.
    We again adopt the pessimistic setting,
    which is reflected in the order of the items $e_7$ and $e_8$
    in scenario $d_1$.
  }
  \label{fig_RCBSP_leader_obj}
\end{figure}

One can already see in Figures~\ref{fig_RCBSP_leader_obj_leader_followers_bl_total} and~\ref{fig_RCBSP_leader_obj_leader_total} that
the pointwise maximum coming from the robustness
can result in vertices of the piecewise linear functions
being not necessarily at integers.
We will now show that,
for this reason,
there might be no optimal solution that is binary,
in contrast to the setting without uncertainty.
Hence,
the binary and the continuous problem variants are not equivalent in the robust setting.

\begin{example} \label{ex_RCBSP_not_equiv}
  Consider the \RCBSP{} with discrete uncertainty.
  Let~$\Ell = \{e_1, e_2, e_3\}$, $\Eff = \{e_4, e_5, e_6\}$, $b = 3$,
  $c(e) = 0$ for all~$e \in \Ell$, $c(e_4) = -1$, $c(e_5) = 1$, $c(e_6) = 0$,
  and~$\U = \{d_1, d_2\}$
  with $d_1(e_4) = 0$, $d_1(e_5) = 1$, $d_1(e_6) = 2$, $d_2(e_4) = 1$, $d_2(e_5) = 2$, and $d_2(e_6) = 0$.
  Effectively,
  the leader can give any capacity~$\bff = b - \bll \in [0, 3]$ to the follower
  by choosing any fraction~$\bll \in [0,3]$ of her items,
  which have objective function value zero;
  see also Section~\ref{sec_BSP_var_sets}.
  Depending on the follower's objective chosen by the adversary,
  the order of the follower's items in his greedy algorithm
  is either~$e_4, e_5, e_6$ or~$e_6, e_4, e_5$.
  As explained above,
  each of these orders corresponds to a piecewise linear function from the leader's perspective,
  and the leader's objective function is now the pointwise maximum of these;
  see Figure~\ref{fig_RCBSP_not_equiv}.

  \begin{figure}
    \centering
    \begin{tikzpicture}
      \draw[-latex] (-0.1,0) node[left] {$0$} -- (3.5,0) node[right] {$\bll$};
      \draw[-latex] (0,-1.5) node[below] {} -- (0,0.5) node[above] {$c(y)$};

      \draw[dotted] (1,-1.5) -- (1,0.5);
      \draw[dotted] (2,-1.5) -- (2,0.5);
      \draw[dotted] (3,-1.5) -- (3,0.5);

      \draw (1,-0.1) node[below] {$1$} -- (1,0.1);
      \draw (2,-0.1) node[below] {$2$} -- (2,0.1);
      \draw (3,-0.1) node[below] {$3$} -- (3,0.1);

      \draw (-0.1,-1) node[left] {$-1$} -- (0.1,-1);;

      \draw (0,0) -- (1,-1) -- (2,0) -- (3,0);
      \draw (0,0) -- (1,0) -- (2,-1) -- (3,0);
      \draw[very thick] (0,0) -- (1,0) -- (1.5,-0.5) -- (2,0) -- (3,0);
    \end{tikzpicture}
    \caption{The leader's objective function of the \RBSP{} instance described in Example~\ref{ex_RCBSP_not_equiv}}
    \label{fig_RCBSP_not_equiv}
  \end{figure}

  One can easily see that
  the minimum of the leader's objective function is at~$\bll = \frac{3}{2}$,
  where the two piecewise linear functions corresponding to the two scenarios intersect,
  i.e., give the same leader's objective value of~$-\frac{1}{2}$.
  In contrast,
  for every binary leader's decision,
  corresponding to~$\bll \in \{0, 1, 2, 3\}$,
  the adversary can choose~$d \in \U$
  (namely $d_2$ for $\bll \in \{0, 1\}$ and $d_1$ for $\bll \in \{2, 3\}$)
  such that the leader's objective value is always~$0$.
  This shows that
  the optimal leader's solution
  in the continuous problem variant
  is not necessarily binary
  in case of discrete uncertainty.
  Hence,
  the \RCBSP{} is indeed not equivalent
  to the binary \RBSP{}.
\end{example}

  Constructing the leader's objective function
  as described above,
  using the pointwise maximum of
  piecewise linear functions corresponding to the scenarios,
  directly leads to a polynomial-time solution algorithm
  for the \RCBSP{}
  under discrete uncertainty;
  see Section~\ref{sec_RCBSP_discrete}.
  In fact,
  one could solve the problem using the same approach
  for every type of uncertainty set,
  even if~$\U$ is not finite.
  The reason for this is that
  we do not need to enumerate all elements of~$\U$,
  but only all possible follower's greedy orders
  arising from the scenarios,
  corresponding to piecewise linear functions that
  contribute to the pointwise maximum,
  and the number of orders of~$\Eff$ is finite.
  Hence,
  for any uncertainty set,
  the leader's objective function
  can be constructed from finitely many piecewise linear functions.
  However,
  their number is not polynomial in general,
  such that the resulting algorithm is usually not efficient.
  For the cases of interval uncertainty
  and discrete uncorrelated uncertainty,
  we will develop efficient algorithms
  that avoid enumerating all possible follower's greedy orders explicitly,
  using similar ideas as in Section~\ref{sec_adversary_interval};
  see Sections~\ref{sec_RCBSP_interval} and~\ref{sec_RCBSP_discrete_uncorr}.

\subsubsection{Discrete Uncertainty} \label{sec_RCBSP_discrete}

In case of discrete uncertainty,
we have seen in Theorem~\ref{thm_RBSP_discrete_adversary} that
the adversary's problem can be solved in polynomial time
for the binary setting,
by enumerating all scenarios.
We now show that
the same is true for the continuous setting
with disjoint item sets:

\begin{theorem} \label{thm_RCBSP_discrete_adversary}
  For any fixed feasible leader's solution
  of the \RCBSP{}
  with a discrete uncertainty set~$\U$,
  the adversary's problem
  can be solved in time~$O(\lvert \U \rvert n)$.
\end{theorem}

\begin{proof}
  Let $\bll = \sum_{e \in \Ell} x_e$ be the capacity
  that the given leader's solution~$x \in [0, 1]^\Ell$ uses.
  Then $\bff = b - \bll$ is the capacity to be filled by the follower.
  Note that,
  in contrast to Theorem~\ref{thm_RBSP_discrete_adversary},
  $\bll$ and $\bff$ are not necessarily integers now.
  However,
  given disjoint item sets~$\Ell$ and~$\Eff$,
  it is still true that
  the only interaction between leader and follower
  is via the capacities they use.
  Hence,
  the follower's task is to solve a single-level \CSP{} on~$\Eff$ with capacity~$\bff$,
  with respect to the item costs chosen by the adversary.
  As in Theorem~\ref{thm_RBSP_discrete_adversary},
  the adversary can simply enumerate
  all scenarios in~$\U$
  and solve the follower's problem for each of them.
  This can again be done in a running time of~$O(\lvert \U \rvert n)$.
\end{proof}

As a black box,
like in Theorem~\ref{thm_RBSP_disj_adversary} and Corollary~\ref{cor_RBSP_disj_discrete_leader},
Theorem~\ref{thm_RCBSP_discrete_adversary} does not suffice here for
constructing an algorithm for the leader's problem
because there are infinitely many feasible choices of~$\bll$
and Example~\ref{ex_RCBSP_not_equiv} shows that
it is indeed not enough to consider only integers~$\bll$.
However,
the observations about the leader's objective function
in Section~\ref{sec_CBSP_PLF} and in the beginning of Section~\ref{sec_RCBSP}
directly imply a polynomial-time algorithm in case of a discrete uncertainty set,
using piecewise linear functions:

\begin{theorem} \label{thm_RCBSP_discrete_leader}
  The \RCBSP{}
  with a discrete uncertainty set~$\U$,
  can be solved in time~$O(\lvert \U \rvert n \log (\lvert \U \rvert n))$.
\end{theorem}

\begin{proof}
  As described in the beginning of Section~\ref{sec_RCBSP},
  the leader's objective function can be written
  as a sum of two continuous piecewise linear functions
  one of which is a pointwise maximum of $\lvert \U \rvert$~piecewise linear functions.
  In order to solve the leader's problem,
  we can compute all these functions explicitly
  and minimize the resulting continuous piecewise linear function
  by evaluating it in all its vertices.
  This leads to an optimal capacity~$\bll$ the leader should choose,
  together with the corresponding leader's greedy solution.

  Regarding the running time,
  first observe that the function corresponding to the items selected by the leader
  and all functions corresponding to the items selected by the follower in the different scenarios
  can be computed in total time~$O(\lvert \U \rvert n \log n)$
  by sorting the items according to their costs.
  The pointwise maximum of the $\lvert \U \rvert$~piecewise linear functions
  with $O(n)$~vertices each
  can be computed in time~$O(\lvert \U \rvert n \log(\lvert \U \rvert))$:
  Since the vertices of the original functions are all at integers,
  each of the $O(n)$~sections between two subsequent integer values of~$\bll$
  can be considered separately.
  For each of these sections,
  we need to compute the pointwise maximum of $\lvert \U \rvert$~many linear pieces.
  By~\cite{hershberger1989},
  this can be done in $O(\lvert \U \rvert \log(\lvert \U \rvert))$.
  Together,
  this results in a running time of~$O(\lvert \U \rvert n \log (\lvert \U \rvert n))$.
\end{proof}

Note that Theorem~\ref{thm_RCBSP_discrete_leader}
can be seen as a generalization of Theorem~1 in~\cite{buchheimhenke2022};
see also Section~\ref{sec_BSP_var_sets}.

\subsubsection{Interval Uncertainty} \label{sec_RCBSP_interval}

In the setting of interval uncertainty,
almost the same ideas as in Section~\ref{sec_adversary_interval}
also apply to the continuous case.
In particular,
the reasoning of Lemma~\ref{lem_RBSP_interval_orders}
is still true,
and also the idea of the precedence constraint knapsack problem from~\cite{woeginger2003}
is still applicable.
However,
as we are allowed to select fractions of items now,
we have to use an algorithm for the \CSP{} instead of the \SP{} as a subroutine.
Moreover,
a technical complication arises from the fact that,
if a fractional solution is selected,
it becomes important
which item can be the last one of the \emph{initial set}
because this will be the item
that is selected fractionally by the follower afterwards.
This is the same issue that
also appears in case of the \RBCKP{} in~\cite{buchheimhenke2022}
and is described using the term of a \emph{fractional prefix} there.

Since the results obtained in this section
are closely related to the ones in Section~4 of~\cite{buchheimhenke2022},
we omit detailed explanations and proofs here.
Recall from Section~\ref{sec_BSP_var_sets}
that the \RCBSP{} generalizes the
\RBCKP{} from~\cite{buchheimhenke2022} in one aspect, namely the leader controlling her own items
instead of just setting the capacity for the follower's problem,
while the problem is more special in another aspect, namely being based on the \CSP{} instead of the \CKP{}.
While the latter difference only leads to a minor simplification regarding the approach developed in~\cite{buchheimhenke2022},
the former one now makes it necessary to deal with an additional piecewise linear function
corresponding to the leader's items (see Section~\ref{sec_CBSP_PLF} and the beginning of Section~\ref{sec_RCBSP})
when solving the leader's problem.

As in Section~\ref{sec_adversary_interval},
for simplicity,
we assume throughout this section
that there are no one-point intersections between intervals involved in~$\U$,
i.e., that, for all~$e_1, e_2 \in \Eff$,
we have $d^-(e_1) \neq d^+(e_2)$;
see also Remark~2 in~\cite{buchheimhenke2022}.

Translating the idea of Algorithm~\ref{alg_RBSP_interval_adversary}
to the continuous setting
results in Algorithm~\ref{alg_RCBSP_interval_adversary}
for solving the adversary's problem.
Note that
we may assume that
the subroutine uses the standard greedy algorithm
for the \CSP{}
and, in particular, that
it always returns a solution
with at most one nonbinary value.
This assumption
ensures that
there is indeed at most one
$e \in \Eff$ selected fractionally,
i.e., with~$y^{\overline{e}}_e \in (0, 1)$,
in Line~\ref{alg_RCBSP_interval_adversary_return}
of the algorithm.

\begin{algorithm}
  \caption{Algorithm for the adversary's problem of the \RCBSP{} with interval uncertainty}
  \label{alg_RCBSP_interval_adversary}

  \Input{finite sets $\Ell$ and $\Eff$ with $\Eff \neq \emptyset$ and $\Ell \cap \Eff = \emptyset$,
    $b \in \{0, \dots, \lvert \Ell \cup \Eff \rvert\}$, $c \colon \Ell \cup \Eff \to \Q$, an interval uncertainty set~$\U$ given by $d^-, d^+ \colon \Eff \to \Q$ with $d^-(e) \le d^+(e)$ for all~$e \in \Eff$, a feasible leader's solution~$x \in [0, 1]^\Ell$}

  \Output{an optimal adversary's solution~$d \in \U$ (i.e., $d \colon \Eff \to \Q$ with $d^-(e) \le d(e) \le d^+(e)$ for all~$e \in \Eff$) of the \RCBSP{}}

  $\bff := b - \sum_{e \in \Ell} x_e$

    $\overline{\E} := \emptyset$

    \For{$\overline{e} \in \Eff$}
    {
      $\E_{\overline{e}}^- := \{e \in \Eff \mid d^+(e) < d^-(\overline{e})\}$

      $\E_{\overline{e}}^0 := \{e \in \Eff \mid d^-(e) \le d^-(\overline{e}) \le d^+(e)\}$

      \If{$\lvert \E_{\overline{e}}^- \rvert \le \bff$ and $\lvert \E_{\overline{e}}^0 \rvert \ge \bff - \lvert \E_{\overline{e}}^- \rvert$}
      {
        $y^{\overline{e}, 0} := \CSP{}(\E_{\overline{e}}^0, \bff - \lvert \E_{\overline{e}}^- \rvert, -c \restriction \E_{\overline{e}}^0)$

        define $y^{\overline{e}} \in [0, 1]^\Eff$ by $y^{\overline{e}}_e := 1$ for all $e \in \E_{\overline{e}}^-$, $y^{\overline{e}}_e := y^{\overline{e},0}_e$ for all $e \in \E_{\overline{e}}^0$, and $y^{\overline{e}}_e := 0$ for all $e \in \Eff \setminus (\E_{\overline{e}}^- \cup \E_{\overline{e}}^0)$

        $\overline{\E} := \overline{\E} \cup \{\overline{e}\}$
      }
    }

    select $\overline{e} \in \argmax\{\sum_{e' \in \Eff} c(e') y^e_{e'} \mid e \in \overline{\E}\}$ arbitrarily

    $\varepsilon := \frac{1}{2} \min_{e \in \E_{\overline{e}}^0} (d^+(\overline{e}) - d^-(e))$

  \Return{$d \colon \Eff \to \Q$ with $d(e) := d^-(e)$ for all $e \in \{e' \in \Eff \mid y^{\overline{e}}_{e'} = 1\}$, $d(e) := d^+(e)$ for all $e \in  \{e' \in \Eff \mid y^{\overline{e}}_{e'} = 0\}$, and $d(e) := d^-(\overline{e}) + \varepsilon$ for the unique (if it exists) $e \in \Eff$ with $y^{\overline{e}}_e \in (0, 1)$ \label{alg_RCBSP_interval_adversary_return}}
\end{algorithm}

In analogy to Lemma~1 in~\cite{buchheimhenke2022},
one can show:

\begin{theorem} \label{thm_RCBSP_interval_adversary}
  For any fixed feasible leader's solution of the \RCBSP{}
  with interval uncertainty,
  Algorithm~\ref{alg_RCBSP_interval_adversary}
  solves the adversary's problem in time~$O(n^2)$.
\end{theorem}

Observe that,
if $\bff$ is an integer,
which is, in particular, the case
if the leader has taken a binary decision,
then Algorithm~\ref{alg_RCBSP_interval_adversary}
coincides with Algorithm~\ref{alg_RBSP_interval_adversary}
for the binary setting.
Indeed,
the subroutine solving the \CSP{}
is only called for integer capacities then,
and hence solves a binary \SP{}.

In contrast to Corollary~\ref{cor_RBSP_disj_interval_leader},
we cannot directly derive an algorithm
for the leader's problem
from Algorithm~\ref{alg_RCBSP_interval_adversary} here
because it does not suffice to enumerate integer values of~$\bll$ in the continuous setting.
However,
the adversary's algorithm can be generalized to a polynomial-time algorithm for the leader
by combining it with the idea of representing the leader's objective function
by piecewise linear functions;
see Section~\ref{sec_CBSP_PLF}
and the beginning of Section~\ref{sec_RCBSP}.
The resulting Algorithm~\ref{alg_RCBSP_interval_leader}
computes a small number of linear pieces
that might contribute to the pointwise maximum,
using the strategy of Algorithm~\ref{alg_RCBSP_interval_adversary}.
Recall the definitions of $\PLFl$ and $\PLFf$
from Section~\ref{sec_CBSP_PLF}
(with Figure~\ref{fig_CBSP_leader_obj}).

\begin{algorithm}
    \caption{Algorithm for the leader's problem of the \RCBSP{} with interval uncertainty}
  \label{alg_RCBSP_interval_leader}

  \Input{finite sets $\Ell$ and $\Eff$ with $\Eff \neq \emptyset$ and $\Ell \cap \Eff = \emptyset$,
    $b \in \{0, \dots, \lvert \Ell \cup \Eff \rvert\}$, $c \colon \Ell \cup \Eff \to \Q$, an interval uncertainty set~$\U$ given by $d^-, d^+ \colon \Eff \to \Q$ with $d^-(e) \le d^+(e)$ for all~$e \in \Eff$}

  \Output{an optimal leader's solution~$x \in [0, 1]^\Ell$ of the \RCBSP{}}

  $\bllm := \max\{0, b - \lvert \Eff \rvert\}$, $\bffp := b - \bllm$

  $\bllp := \min\{b, \lvert \Ell \rvert\}$, $\bffm := b - \bllp$

  compute $\gll := \PLFl(\Ell, c \restriction \Ell, \bllm, \bllp)$

    $\overline{\E} := \emptyset$

    \For{$\overline{e} \in \Eff$}
    {
      $\E_{\overline{e}}^- := \{e \in \Eff \mid d^+(e) < d^-(\overline{e})\}$

      $\E_{\overline{e}}^0 := \{e \in \Eff \mid d^-(e) \le d^-(\overline{e}) \le d^+(e)\}$

      \If{$\lvert \E_{\overline{e}}^- \rvert \le \bffp$ and $\lvert \E_{\overline{e}}^0 \rvert \ge \bffm - \lvert \E_{\overline{e}}^- \rvert$}
      {
        $\overline{\bffm} := \max\{0, \bffm - \lvert \E_{\overline{e}}^- \rvert\}$

        $\overline{\bffp} := \min\{\bffp - \lvert \E_{\overline{e}}^- \rvert, \lvert \E_{\overline{e}}^0 \rvert\}$
        
        compute $\gffole := \PLFf(\E_{\overline{e}}^0, -c \restriction \E_{\overline{e}}^0, c \restriction \E_{\overline{e}}^0, b - \overline{\bffp}, b - \overline{\bffm}, b)$

        compute $\gffole := \gffole + (- \lvert \E_{\overline{e}}^- \rvert, \sum_{e \in \E_{\overline{e}}^-} c(e))$ (i.e., shift the function~$\gffole$ by~$- \lvert \E_{\overline{e}}^- \rvert$ in $x$~direction and by~$\sum_{e \in \E_{\overline{e}}^-} c(e)$ in $y$~direction)

        $\overline{\E} := \overline{\E} \cup \{\overline{e}\}$
      }
    }

    compute $\gff := \max\{\gffole \mid \overline{e} \in \overline{\E}\}$ (i.e., the pointwise maximum)

    compute $g := \gll + \gff$

    compute some $\blls \in \argmin\{g(\bll) \mid \bll \in [\bllm, \bllp]\}$

  \Return{$\CSP{}(\Ell, \blls, c \restriction \Ell)$}
\end{algorithm}

Similarly to Theorem~2 in~\cite{buchheimhenke2022}
and using the observations from Section~\ref{sec_CBSP_PLF} and the beginning of Section~\ref{sec_RCBSP},
we derive:

\begin{theorem} \label{thm_RCBSP_interval_leader}
  Algorithm~\ref{alg_RCBSP_interval_leader} solves
  the \RCBSP{}
  with interval uncertainty
  in time~$O(n^2 \log n)$.
\end{theorem}

\subsubsection{Discrete Uncorrelated Uncertainty} \label{sec_RCBSP_discrete_uncorr}

In the binary setting,
we have shown that discrete uncorrelated uncertainty is equivalent to interval uncertainty;
see Theorem~\ref{thm_RBSP_discrete_uncorr_equiv}.
In particular,
the optimal adversary's solutions computed by Algorithm~\ref{alg_RBSP_interval_adversary}
for interval uncertainty
were always attained at an endpoint of each interval.
The latter is not true anymore
in Algorithm~\ref{alg_RCBSP_interval_adversary},
the variant of this algorithm
that solves the adversary's problem in the continuous case.
In fact,
this algorithm sets all costs~$d(e)$
to endpoints of the respective intervals,
except for the one that corresponds to the unique item
the follower is supposed to select only a fraction of.
The special adversary's choice of this value is necessary
because this item has to be guaranteed to be the last one the follower selects,
i.e., the last item in the corresponding fractional prefix.
This is also demonstrated in Examples~1 and~2 in~\cite{buchheimhenke2022}.
Hence,
the solution returned by Algorithm~\ref{alg_RCBSP_interval_adversary}
cannot be assumed to be feasible for the adversary's problem
in case of discrete uncorrelated uncertainty in general.
Therefore,
we cannot solve the \RCBSP{} with discrete uncorrelated uncertainty
by simply replacing the uncertainty set by the interval uncertainty set that is its convex hull,
as in Theorem~\ref{thm_RBSP_discrete_uncorr_equiv}.

Nevertheless,
one can solve the \RCBSP{} and its adversary's problem
in case of discrete uncorrelated uncertainty
in polynomial time
using similar ideas
as for interval uncertainty,
as we show in this section.
Moreover,
in Appendix~\ref{sec_RCBSP_discrete_uncorr_knapsack}
we present a generalization
of the ideas,
resulting in
pseudopolynomial-time algorithms
for the adversary's problem and the leader's problem
of the \RBCKP{} with discrete uncorrelated uncertainty,
adding to the complexity results in~\cite{buchheimhenke2022}.

We first study the adversary's problem
of the \RCBSP{} with discrete uncorrelated uncertainty.
Assume that
a feasible leader's solution~$x \in [0, 1]^\Ell$ is given.
As before,
this defines the capacity~$\bff = b - \sum_{e \in \Ell} x_e$
the follower has to fill.
First note that,
in case $\bff$ is an integer,
the follower can be assumed to always make a binary decision
because he solves a \CSP{}
with an integer capacity.
Therefore,
we can handle the adversary's problem
like in the binary case,
i.e., apply Algorithm~\ref{alg_RBSP_interval_adversary},
because of Theorem~\ref{thm_RBSP_discrete_uncorr_equiv}.
For the general case
where $\bff$ is not an integer,
we develop a modified version of this algorithm
and consider the follower's solution
as being composed of a binary selection of $\lfloor \bff \rfloor$~items
and an additional item
of which the follower selects a fraction of $\bff - \lfloor \bff \rfloor$.

Recall that,
in Algorithms~\ref{alg_RBSP_interval_adversary}, \ref{alg_RCBSP_interval_adversary}, and~\ref{alg_RCBSP_interval_leader},
we iterate over all possible \emph{heads} $\overline{e}$
of the (fractional) prefix corresponding to the follower's solution
and use the left endpoint~$d^-(\overline{e})$ of its interval
as a reference point
for defining the item sets~$\E_{\overline{e}}^-$ and~$\E_{\overline{e}}^0$
from which the adversary builds a worst-case follower's solution for the current iteration.
For more details,
we refer to~\cite{buchheimhenke2022}.
Algorithm~\ref{alg_RCBSP_interval_adversary}
assumes that any item in $\E_{\overline{e}}^0$
can be made the fractional item
by selecting the cost~$d^-(\overline{e}) + \varepsilon$ for it.
On a related note,
for a number of intervals that have a proper intersection,
the adversary can enforce any order of the corresponding items.
When we turn to discrete uncorrelated uncertainty,
the adversary's choice is more limited
and fewer follower's greedy orders
are possible;
see also Remark~\ref{rem_RBSP_discrete_uncorr_orders}.
Regarding the possible fractional items,
this is a significant limitation.
Therefore,
we have to take a closer look at the specific cost values
the adversary can select for the fractional item.

Instead of the head,
which is not necessarily the fractional item in Algorithm~\ref{alg_RCBSP_interval_adversary} and in~\cite{buchheimhenke2022},
we now guess the item~$e^* \in \Eff$
that the follower selects only a fraction of,
together with the cost value~$\delta^* \in \U_{e^*}$ the adversary assigns to it.
Note that
there are $\sum_{e \in \Eff} \lvert \U_e \rvert$
such guesses to consider
in total.
The remaining adversary's problem
is then similar to the problem solved in an iteration of Algorithm~\ref{alg_RBSP_interval_adversary},
i.e., it has to be decided
which $\lfloor \bff \rfloor$~items the follower selects before the fractional item.
This can again be solved as a binary \SP{},
while we now use~$\delta^*$ as a reference point
in the definition of the item sets~$\E_{e^*, \delta^*}^-$ and~$\E_{e^*, \delta^*}^0$.
Among all possible choices of~$(e^*, \delta^*)$,
the adversary selects the one
that yields the worst leader's costs
of the resulting follower's solution.
This strategy leads to Algorithm~\ref{alg_RCBSP_discrete_uncorr_adversary}.

\begin{algorithm}
  \caption{Algorithm for the adversary's problem of the \RCBSP{} with discrete uncorrelated uncertainty}
  \label{alg_RCBSP_discrete_uncorr_adversary}

  \Input{finite sets $\Ell$ and $\Eff$ with $\Ell \cap \Eff = \emptyset$,
    $b \in \{0, \dots, \lvert \Ell \cup \Eff \rvert\}$,
    $c \colon \Ell \cup \Eff \to \Q$,
    a discrete uncorrelated uncertainty set~$\U$ given by a finite set~$\U_e$ for every~$e \in \Eff$,
    a feasible leader's solution~$x \in [0, 1]^\Ell$}

  \Output{an optimal adversary's solution~$d \in \U$ (i.e., a value $d(e) \in \U_e$ for every~$e \in \Eff$)
    of the \RCBSP{}}

  $\bff := b - \sum_{e \in \Ell} x_e$

  \If{$\bff = 0$ or $\bff = \lvert \Eff \rvert$}
  {
    \Return{an arbitrary $d \in \U$}
  }
  
  \For{$e \in \Eff$}
    {
      $d^-(e) := \min \U_e$, $d^+(e) := \max \U_e$
    }
    
  $\overline{\E} := \emptyset$
  
  \For{$e^* \in \Eff$}
  {
    \For{$\delta^* \in \U_{e^*}$}
    {
      $\E_{e^*, \delta^*}^- := \{e \in \Eff \mid d^+(e) < \delta^*\}$

      $\E_{e^*, \delta^*}^0 := \{e \in \Eff \mid d^-(e) \le \delta^* \le d^+(e)\} \setminus \{e^*\}$

      \If{$\lvert \E_{e^*, \delta^*}^- \rvert \le \lfloor \bff \rfloor$ and $\lvert \E_{e^*, \delta^*}^0 \rvert \ge \lfloor \bff \rfloor - \lvert \E_{e^*, \delta^*}^- \rvert$ \label{alg_RCBSP_discrete_uncorr_adversary_if}}
      {
        $Y_{e^*, \delta^*}^0 := \SP{}(\E_{e^*, \delta^*}^0, \lfloor \bff \rfloor - \lvert \E_{e^*, \delta^*}^- \rvert, -c \restriction \E_{e^*, \delta^*}^0)$
        
        $Y_{e^*, \delta^*} := \E_{e^*, \delta^*}^- \cup Y_{e^*, \delta^*}^0$
        
        $\overline{\E} := \overline{\E} \cup \{(e^*, \delta^*)\}$
      }
    }
  }

  select $(e^*, \delta^*) \in \argmax\{c(Y_{e, \delta}) + c(e) \cdot (\bff - \lfloor \bff \rfloor) \mid (e, \delta) \in \overline{\E}\}$ arbitrarily

  \Return{$d \colon \Eff \to \Q$ with $d(e^*) := \delta^*$, $d(e) = d^-(e)$ for all $e \in Y_{e^*, \delta^*}$ and $d(e) := d^+(e)$ for all $e \in \Eff \setminus (Y_{e^*, \delta^*} \cup \{e^*\})$}
\end{algorithm}

Again, similarly to the proof of Lemma~1 in~\cite{buchheimhenke2022},
one can show:

\begin{theorem} \label{thm_RCBSP_discrete_uncorr_adversary}
  For any fixed feasible leader's solution of the \RCBSP{}
  with a discrete uncorrelated uncertainty set~$\U = \prod_{e \in \Eff} \U_e$,
  Algorithm~\ref{alg_RCBSP_discrete_uncorr_adversary}
  solves the adversary's problem in time~$O(n u)$,
  where~$u = \sum_{e \in \Eff} \lvert \U_e \rvert$.
\end{theorem}

Observe that,
if $\bff$ is an integer
(and we are not in one of the trivial cases~$\bff = 0$ and~$\bff = \lvert \Eff \rvert$),
then Algorithm~\ref{alg_RCBSP_discrete_uncorr_adversary} selects
a fraction of~$\bff - \lfloor \bff \rfloor = 0$ of the item~$e^*$.
Hence,
the algorithm examines binary prefixes, very similarly to Algorithm~\ref{alg_RBSP_interval_adversary};
only more possible choices of the reference point $\delta^*$ are considered here,
which does not change the result in this case.

In order to devise an algorithm for the leader's problem,
in case of interval uncertainty,
we combined the approach used for solving the adversary's problem
with the idea to represent the leader's objective function by piecewise linear functions;
see Algorithm~\ref{alg_RCBSP_interval_leader}.
Something similar can be done here
in order to generalize Algorithm~\ref{alg_RCBSP_discrete_uncorr_adversary}
to an algorithm that solves the leader's problem.
This results in Algorithm~\ref{alg_RCBSP_discrete_uncorr_leader}.

\begin{algorithm}
  \caption{Algorithm for the leader's problem of the \RCBSP{} with discrete uncorrelated uncertainty}
  \label{alg_RCBSP_discrete_uncorr_leader}

  \Input{finite sets $\Ell$ and $\Eff$ with $\Ell \cap \Eff = \emptyset$,
    $b \in \{0, \dots, \lvert \Ell \cup \Eff \rvert\}$,
    $c \colon \Ell \cup \Eff \to \Q$,
    a discrete uncorrelated uncertainty set~$\U$ given by a finite set~$\U_e$ for every~$e \in \Eff$}

  \Output{an optimal leader's solution~$x \in [0, 1]^\Ell$ of the \RCBSP{}}

  $\bllm := \max\{0, b - \lvert \Eff \rvert\}$, $\bffp := b - \bllm$

  $\bllp := \min\{b, \lvert \Ell \rvert\}$, $\bffm := b - \bllp$

  \If{$\bllm = \bllp$}
  {
    $\blls := \bllm$
  }
  \Else
  {  
    compute $\gll := \PLFl(\Ell, c \restriction \Ell, \bllm, \bllp)$ \label{alg_RCBSP_discrete_uncorr_leader_gll}

    \For{$e \in \Eff$}
    {
      $d^-(e) := \min \U_e$, $d^+(e) := \max \U_e$
    }

    \For{$\bffs \in \{\bffm, \dots, \bffp - 1\}$}
    {
      $\overline{\E} := \emptyset$
      
      \For{$e^* \in \Eff$}
      {
        \For{$\delta^* \in \U_{e^*}$}
        {
          $\E_{e^*, \delta^*}^- := \{e \in \Eff \mid d^+(e) < \delta^*\}$

          $\E_{e^*, \delta^*}^0 := \{e \in \Eff \mid d^-(e) \le \delta^* \le d^+(e)\} \setminus \{e^*\}$

          \If{$\lvert \E_{e^*, \delta^*}^- \rvert \le \bffs$ and $\lvert \E_{e^*, \delta^*}^0 \rvert \ge \bffs - \lvert \E_{e^*, \delta^*}^- \rvert$}
          {
            $Y_{e^*, \delta^*}^0 := \SP{}(\E_{e^*, \delta^*}^0, \bffs - \lvert \E_{e^*, \delta^*}^- \rvert, -c \restriction \E_{e^*, \delta^*}^0)$
            
            $Y_{e^*, \delta^*} := \E_{e^*, \delta^*}^- \cup Y_{e^*, \delta^*}^0$

            $\gffbffsesdeltas :=$ linear piece from $(b - \bffs - 1, c(Y_{e^*, \delta^*}) + c(e^*))$ to $(b - \bffs, c(Y_{e^*, \delta^*}))$ \label{alg_RCBSP_discrete_uncorr_leader_gffbffsesdeltas}
            
            $\overline{\E} := \overline{\E} \cup \{(e^*, \delta^*)\}$
          }
        }
      }
      compute $\gffbffs := \max\{\gffbffsesdeltas \mid (e^*, \delta^*) \in \overline{\E}\}$ (i.e., the pointwise maximum) \label{alg_RCBSP_discrete_uncorr_leader_gffbffs}
    }

    join all $\gffbffs$ (with range $[b - \bffs - 1, b - \bffs]$) to a function~$\gff$ with range $[\bllm, \bllp]$

    compute $g := \gll + \gff$ \label{alg_RCBSP_discrete_uncorr_leader_g}

    compute some $\blls \in \argmin\{g(\bll) \mid \bll \in [\bllm, \bllp]\}$
  }

  \Return{$\CSP{}(\Ell, \blls, c \restriction \Ell)$ \label{alg_RCBSP_discrete_uncorr_leader_return}}
\end{algorithm}

The computation of the function~$\gll$ (Line~\ref{alg_RCBSP_discrete_uncorr_leader_gll}),
representing the leader's costs of the items she selects herself,
and the final steps to determine an optimal solution
(Lines~\ref{alg_RCBSP_discrete_uncorr_leader_g} to~\ref{alg_RCBSP_discrete_uncorr_leader_return})
are as in Algorithm~\ref{alg_RCBSP_interval_leader};
they are independent of the type of uncertainty set.
The function~$\gff$,
representing the leader's costs of the items selected by the follower,
is now computed by a generalization of Algorithm~\ref{alg_RCBSP_discrete_uncorr_adversary}.
Indeed,
for every fixed follower's capacity~$\bff \in [\bffm, \bffp]$,
the value~$\gffbffs(\bff)$, for~$\bffs = \lfloor \bff \rfloor$
(or $\bffs = \bff - 1$ in case $\bff = \bffp$),
and therefore the value $\gff(\bff)$,
corresponds to the worst leader's costs
of a follower's selection
that the adversary can achieve.
The approach to determine this value
is like in Algorithm~\ref{alg_RCBSP_discrete_uncorr_adversary}:
All possible fractional items~$e^* \in \Eff$
and all possible adversary's choices~$\delta^* \in \U_{e^*}$
are enumerated
and a \SP{} is solved in each iteration.
This is independent of the precise fraction~$\bff - \lfloor \bff \rfloor \in [0, 1)$
according to which the follower selects the fractional item.
Therefore,
we consider a linear piece corresponding to all possible fractions
in Line~\ref{alg_RCBSP_discrete_uncorr_leader_gffbffsesdeltas}.
The worst case among all possible choices of $(e^*, \delta^*)$,
which can be chosen individually for each~$\bff$ by the adversary,
now corresponds to the pointwise maximum of these linear pieces
(Line~\ref{alg_RCBSP_discrete_uncorr_leader_gffbffs}).
Finally,
we handle each integer value $\bffs$ separately,
similarly to the binary problem version.
On a related note,
observe that
we solve a binary \SP{} as a subroutine
in Algorithms~\ref{alg_RCBSP_discrete_uncorr_adversary} and~\ref{alg_RCBSP_discrete_uncorr_leader},
and not a \CSP{}
as in Algorithms~\ref{alg_RCBSP_interval_adversary} and~\ref{alg_RCBSP_interval_leader};
see also Remark~\ref{rem_RCBSP_discrete_uncorr_interval} below.

Using the ideas explained above,
one can prove,
again similarly to Theorem~2 in~\cite{buchheimhenke2022}:

\begin{theorem} \label{thm_RCBSP_discrete_uncorr_leader}
  Algorithm~\ref{alg_RCBSP_discrete_uncorr_leader} solves the
  \RCBSP{}
  with a discrete uncorrelated uncertainty set~$\U = \prod_{e \in \Eff} \U_e$
  in time~$O(n u \log (n u))$,
  where~$u = \sum_{e \in \Eff} \lvert \U_e \rvert$.
\end{theorem}

We emphasize that
the term $u$ and therefore the running times stated in Theorems~\ref{thm_RCBSP_discrete_uncorr_adversary} and~\ref{thm_RCBSP_discrete_uncorr_leader}
are polynomial in the input size,
while the total size $\lvert \U \rvert$ of the discrete uncorrelated uncertainty set
is exponential in the input size.
Moreover,
if the uncertainty set is given by sets~$\U_e$ of constant size,
in particular if each $\U_e$ consists of only two values~$d^-(e)$ and~$d^+(e)$,
then the running time of Algorithm~\ref{alg_RCBSP_discrete_uncorr_adversary}
is~$O(n^2)$,
and the running time of Algorithm~\ref{alg_RCBSP_discrete_uncorr_leader}
is~$O(n^2 \log n)$,
like in case of interval uncertainty;
see Theorems~\ref{thm_RBSP_interval_adversary} and~\ref{thm_RCBSP_interval_adversary},
Corollary~\ref{cor_RBSP_disj_interval_leader} and Theorem~\ref{thm_RCBSP_interval_leader}.

Observe that
$\gffbffs$ is defined for all $\bffs \in \{\bffm, \dots, \bffp - 1\}$
and that the function $\gff$ is continuous,
although both is not directly obvious from Algorithm~\ref{alg_RCBSP_discrete_uncorr_leader}:
The alternative way of describing the function $\gff$,
as a pointwise maximum of continuous piecewise linear functions
for all (exponentially many) scenarios in~$\U$
(see the beginning of Section~\ref{sec_RCBSP}),
demonstrates that this is true.
In fact,
similarly to Algorithm~\ref{alg_RCBSP_interval_leader},
Algorithm~\ref{alg_RCBSP_discrete_uncorr_leader}
provides a way to determine~$\gff$
by computing only a polynomial number of linear pieces
that might contribute to the pointwise maximum.

\begin{remark} \label{rem_RCBSP_discrete_uncorr_interval}
  The approach we developed for discrete uncorrelated uncertainty
  is similar to the one for interval uncertainty,
  but there are some differences
  we need to point out.
  As mentioned above,
  the heads~$\overline{e}$
  enumerated in Algorithms~\ref{alg_RCBSP_interval_adversary} and~\ref{alg_RCBSP_interval_leader}
  are not necessarily the fractional items in the resulting follower's solutions;
  they only provide useful reference points
  in order to find all relevant fractional prefixes.
  Accordingly,
  we solve a \CSP{} in a subroutine there
  and determine the fractional item during that.
  In contrast,
  we fix the fractional item~$e^*$ in Algorithms~\ref{alg_RCBSP_discrete_uncorr_adversary} and~\ref{alg_RCBSP_discrete_uncorr_leader},
  and solve a binary \SP{} in a subroutine,
  in order to determine the remaining items of the fractional prefix.
  In fact,
  we could use the latter approach also for solving the \RCBSP{} with interval uncertainty
  in polynomial time,
  by enumerating the possible fractional items
  together with a reasonable finite set of possible cost values the adversary may assign to them
  (depending on the endpoints of the other intervals).
  However,
  this would not give any better results than
  Theorems~\ref{thm_RCBSP_interval_adversary} and~\ref{thm_RCBSP_interval_leader}.
  
  In the corresponding versions of the \RBCKP{},
  the difference between the two approaches is more significant:
  The \CKP{} can be solved in polynomial time,
  while the binary \KP{} cannot.
  Therefore,
  the \RBCKP{} with interval uncertainty
  can be solved in polynomial time
  only by applying the strategy that solves a \CKP{} in each iteration;
  see~\cite{buchheimhenke2022}.
  In case of discrete uncorrelated uncertainty,
  this polynomial-time approach does not work,
  analogously to the discussion in the beginning of Section~\ref{sec_RCBSP_discrete_uncorr}.
  In fact,
  the \RBCKP{} has been shown to be weakly NP-hard in~\cite{buchheimhenke2022}.
  However,
  the new approach we developed in this section
  can be generalized to
  a pseudopolynomial-time algorithm
  for the \RBCKP{};
  see Appendix~\ref{sec_RCBSP_discrete_uncorr_knapsack}.
\end{remark}

\begin{remark} \label{rem_RCBSP_discrete_uncorr_opt_pess}
  In Sections~\ref{sec_adversary_interval} and~\ref{sec_RCBSP_interval},
  for interval uncertainty,
  we assumed that
  there are no one-point intersections between
  the intervals in the uncertainty set,
  in order to avoid discussing technical details related to the optimistic and the pessimistic case.
  For discrete uncorrelated uncertainty,
  one could assume
  that the sets~$\U_e$ are pairwise disjoint.
  However,
  analogously to interval uncertainty,
  Algorithms~\ref{alg_RCBSP_discrete_uncorr_adversary} and~\ref{alg_RCBSP_discrete_uncorr_leader},
  as they are stated above,
  are also correct for the pessimistic setting
  without this assumption,
  while an adjusted definition of the item sets~$\E_{e^*, \delta^*}^-$ and~$\E_{e^*, \delta^*}^0$
  is required for the optimistic setting,
  similarly to Remark~2 in~\cite{buchheimhenke2022}.
\end{remark}

\section{Conclusion} \label{sec_conclusion}

We have investigated the complexity of robust bilevel optimization
at the example of the polynomial-time solvable \BSP{},
for several problem variants.
With this,
we add to the still relatively limited amount of literature
on bilevel optimization under uncertainty,
in particular regarding its complexity.

Firstly,
our results generalize and extend
the structural and algorithmic results of~\cite{buchheimhenke2022}
to a related setting.
Here,
we place a stronger focus on combinatorial underlying problems,
using the fundamental \SP{} as a basis,
which is often used to investigate the complexity of robust optimization concepts.
In Section~\ref{sec_RCBSP},
we generalized the approach of viewing
the leader's objective function as a piecewise linear function,
which has also been used in~\cite{buchheimhenke2022}.

In relation to~\cite{buchheimhenke2022},
we emphasize the following insights regarding the complexity
of the \RBCKP{}
and the \RCBSP{},
which is simpler in terms of the items not having different sizes.
In case of discrete uncorrelated uncertainty,
the leader's as well as the adversary's problem
have been shown to be NP-hard in~\cite{buchheimhenke2022},
while the corresponding \RCBSP{}
is still solvable in polynomial time.
In Appendix~\ref{sec_RCBSP_discrete_uncorr_knapsack},
we generalize our new approach
to a pseudopolynomial-time algorithm
for the \RBCKP{}
and thus showed that the problem with discrete uncorrelated uncertainty
is in fact only weakly NP-hard, but not strongly NP-hard,
answering an open question stated in~\cite{buchheimhenke2022}.

Furthermore,
we complemented the algorithmic results
that are related to~\cite{buchheimhenke2022}
with a complexity study of the more general (\textsc{Robust}) \BSP{}
in which the leader's and the follower's item sets are not necessarily disjoint.
We have shown that the \RBSP{} is strongly NP-hard in general,
but can be 2-approximated in polynomial time.
Moreover,
we investigated exponential-time solution approaches
for this problem,
including an algorithm
that runs in polynomial time
if the number of scenarios in the uncertainty set
is a constant.

Another interesting observation is that
the variant of the \BSP{} with continuous leader's variables
and binary follower's variables
yields a bilevel optimization problem
whose optimal value may not be attained
(see Example~\ref{ex_cont_bin_no_opt}).
This demonstrates that,
in general,
bilevel optimization problems
containing integer and continuous variables
should be handled with care
in this respect.
It could be interesting to study this problem variant further, e.g.,
in view of computing the optimal value or computing feasible solutions that approach the optimal value.

It remains a task for future work
to investigate whether
the general \RBSP{} is also NP-hard
for other types of uncertainty,
as it is for discrete uncertainty.
Also in case of disjoint item sets,
it could be interesting to study other types of uncertainty sets,
e.g., budgeted, polytopal, or ellipsoidal uncertainty,
like in~\cite{buchheimhenke2022}.
If the \RBSP{} could be solved in polynomial time here,
then it would be again easier than the \RBCKP{},
like for discrete uncorrelated uncertainty,
and polynomial-time algorithms
could again be helpful to develop pseudopolynomial-time algorithms
for the \RBCKP{} in the corresponding cases.

We have defined the \SP{} and the \BSP{} as minimization problems
in this article.
In relation to the knapsack problem of~\cite{buchheimhenke2022},
one could also formulate them as maximization problems.
For most of our results,
this does not make a difference,
but the proof of the approximation result in Theorem~\ref{thm_RBSP_approx_adversary}
is only valid for the minimization version.
This suggests the open question
whether a similar result can also be obtained
for the corresponding maximization problem.
Concerning Theorem~\ref{thm_RBSP_approx_adversary},
one could also investigate
if a better approximation factor can be achieved by another algorithm
or if any lower bound for a possible approximation guarantee can be obtained.
Note that
the reduction in the proof of Theorem~\ref{thm_RBSP_hardness}
does not imply any statement about the approximability of the problem.

In Section~\ref{sec_cont},
we mostly focused on the problem variant with disjoint item sets.
One could also study the general \RCBSP{} with arbitraty item sets.
However,
the structural insights regarding piecewise linear functions
do not seem to be easily applicable to this setting,
and we conjecture that
it becomes significantly harder,
similarly to the binary case.

As a generalization of the combinatorial (\textsc{Robust}) \BSP{},
also the corresponding \BKP{} could be interesting to consider further
in view of its complexity.
While the \BKP{} with disjoint item sets
has been proved to be weakly \Sigp{2}-hard in~\cite{capraraetal2014},
it is open whether it is hard in the weak or in the strong sense
in the more general case of arbitrary item sets.
In fact,
this problem could be related to the bilevel knapsack problem
of~\cite{denegre2011}
that has been shown to be strongly NP-hard
in~\cite{capraraetal2014}.
The \BKP{} could also be a good starting point
to investigate the additional complexity
that robustness introduces
in a harder underlying bilevel optimization problem
because its structure is still relatively easy to understand.

Regarding the \RBCKP{},
the NP-hardness results obtained in~\cite{buchheimhenke2022}
for some types of uncertainty sets,
in particular for discrete uncorrelated uncertainty,
all concern both the leader's and the adversary's problem.
In view of multilevel problems
and the polynomial hierarchy
(see, e.g.,~\cite{jeroslow1985}),
it is possible that
the robust leader's problems
are actually one level harder than the adversary's problems
and thus \Sigp{2}-hard.
Exploring whether this is the case
could yield a better understanding
of problems on the second level of the polynomial hierarchy
and of robust bilevel optimization.

\appendix

\section{Appendix: A Pseudopolynomial-Time Algorithm for the Robust Bilevel Continuous Knapsack Problem} \label{sec_RCBSP_discrete_uncorr_knapsack}

We now adapt the algorithms
for the \RCBSP{} with discrete uncorrelated uncertainty
that we presented in Section~\ref{sec_RCBSP_discrete_uncorr},
in order to apply them to the related \RBCKP{} from~\cite{buchheimhenke2022}.
As indicated in Remark~\ref{rem_RCBSP_discrete_uncorr_interval},
this will lead to the first pseudopolynomial-time algorithm
for this problem,
while its weak NP-hardness has been proved in Theorem~3 of~\cite{buchheimhenke2022}.
Also the adversary's problem is weakly NP-hard
(see Theorem~4 of~\cite{buchheimhenke2022}),
and we give a pseudopolynomial-time algorithm for it
as well.

Recall that,
in the \RBCKP{},
the leader does not control any items,
but directly sets the capacity for the follower's \CKP{}.
Moreover,
the items do not only have leader's and follower's costs (or values),
but also different sizes,
as usual in the context of knapsack problems.
For more details and other (minor) differences
between the \RCBSP{} and the \RBCKP{} from~\cite{buchheimhenke2022},
we refer to Section~\ref{sec_BSP_var}.

In the following,
we state the \RBCKP{} again
in the variant that has been studied in~\cite{buchheimhenke2022},
while adjusting the notation to the one used in this article.
We are given a finite set~$\E$ of items,
item sizes~$a \colon \E \to \Qge$,
leader's item values $c \colon \E \to \Q$,
and capacity bounds~$\bm, \bp \in \Q$ with $0 \le \bm \le \bp \le a(\E)$.
The follower's item values $d \colon \E \to \Qgt$ are subject to uncertainty,
which here means that
we are given a discrete uncorrelated uncertainty set~$\U = \prod_{e \in \E} \U_e$
with finite sets~$\U_e \subseteq \Qgt$.
The \RBCKP{} can now be formulated as follows:
\begin{equation} \label{eq_RBCKP} \tag{RBCKP}
\begin{aligned}
  \max_{b \in [\bm, \bp]}~ \min_{d \in \U}~& \sum_{e \in \E} c(e) y_e\\
  \st~ & y \in
  \begin{aligned}[t]
    \argmax_{y'}~ & \sum_{e \in \E} d(e) y'_e \\
    \st~ & \sum_{e \in \E} a(e) y'_e \le b \\
    & y' \in [0, 1]^\E
  \end{aligned}
\end{aligned}
\end{equation}

Observe that
the follower can always be assumed to fill the capacity~$b$ completely here
because the follower's item values are positive and $b \le a(\E)$.
Therefore, the constraint $\sum_{e \in \E} a(e) y'_e \le b$
could be equivalently replaced by $\sum_{e \in \E} a(e) y'_e = b$.

In view of the pseudopolynomial-time algorithms
that we will discuss,
we assume from now on that the item sizes~$a \colon \E \to \Nn$ are all integers
(which can be achieved by scaling the instance if necessary).
The total size~$a(\E) = \sum_{e \in \E} a(e)$ of all items
will then appear in the running times of the algorithms.

The item sizes are the main complication
of the \RBCKP{} compared to the \RCBSP{},
and we need to take care of them in several ways:

Firstly,
the follower's greedy algorithm
depends on the order of the item profits~$d(e)/a(e)$
instead of only the item values~$d(e)$.
This is reflected in the definitions of the sets~$\E_{e^*, \delta^*}^-$ and~$\E_{e^*, \delta^*}^0$
in Algorithms~\ref{alg_RBCKP_adversary} and~\ref{alg_RBCKP_leader},
analogously to Algorithms~1 and~2 in \cite{buchheimhenke2022}
for the adversary's problem and the leader's problem
of the \RBCKP{} with interval uncertainty.

Secondly,
when fixing the fractional item
of the follower's \CKP{},
the fraction with which it is selected
is not uniquely determined by the capacity~$b$,
in contrast to the case where the follower solves a \CSP{}.
If $b$ is not an integer,
in the latter case,
the fraction is always given by $b - \lfloor b \rfloor$,
and we know that $b$~items are selected before the fractional item.
We made use of this property in Algorithm~\ref{alg_RCBSP_discrete_uncorr_adversary}.
In the more general case of a \CKP{},
the items selected before the fractional item~$e^*$
are known to have an integer total size as well
(as we assume all item sizes to be integers),
but this can be any integer number~$b^*$ from $\lfloor b - a(e^*) \rfloor + 1$ to $\lfloor b \rfloor$ now,
corresponding to fractions~$y_{e^*} = \frac{1}{a(e^*)} (i + (b - \lfloor b \rfloor))$ for $i \in \{0, \dots, a(e^*) - 1\}$
(in reverse order).
Therefore,
we enumerate all these possibilities for~$b^*$
in addition to the possible fractional items~$e^*$
and their follower's item values~$\delta^*$
in Algorithm~\ref{alg_RBCKP_adversary}.
Note that Line~\ref{alg_RBCKP_adversary_bs} of Algorithm~\ref{alg_RBCKP_adversary}
combines this with the capacity condition
that is used also in previous versions of the algorithm;
see, e.g., Line~\ref{alg_RCBSP_discrete_uncorr_adversary_if} of Algorithm~\ref{alg_RCBSP_discrete_uncorr_adversary}.
When solving the leader's problem,
we have to enumerate all reasonable integer capacities~$b^*$ anyway
(see Algorithm~\ref{alg_RCBSP_discrete_uncorr_leader})
and do this in a similar way also in Algorithm~\ref{alg_RBCKP_leader}.

For fixed~$e^*$, $\delta^*$, and~$b^*$,
the remaining task is now to solve a binary \KP{}
instead of a binary \SP{}
as in Algorithms~\ref{alg_RCBSP_discrete_uncorr_adversary} and~\ref{alg_RCBSP_discrete_uncorr_leader}.
More precisely,
the subroutine corresponds to a binary \KP{}
with an equality constraint
because the items in the fractional prefix before the fractional item
need to have a total size of exactly~$b^*$
according to the above considerations.
If this problem is infeasible,
then we skip the current value of~$b^*$.

Formally,
the problem that the subroutine solves
can be stated as follows:
Given a finite set~$\E'$ of items,
item sizes~$a' \colon \E' \to \Nn$,
a capacity~$b' \in \{0, \dots, a'(\E')\}$,
and item values~$c' \colon \E' \to \Q$,
select a subset~$Y \subseteq \E'$ such that
$a'(Y) = b'$ and $c'(Y)$ is maximized,
or decide that there is no such~$Y$.
We write $\KP{}(\E', a', b', c')$ when we call this subroutine
and assume that it returns an optimal solution~$Y$
if there is one,
or the string~``infeasible''
otherwise.
The subroutine can be implemented to run in time~$O(\lvert \E' \rvert b')$
using dynamic programming;
see, e.g., \cite{kellererpferschypisinger2004,mansietal2012}.

We first adapt Algorithm~\ref{alg_RCBSP_discrete_uncorr_adversary}
to the adversary's problem of the \RBCKP{}.
Respecting all aspects explained above,
we obtain Algorithm~\ref{alg_RBCKP_adversary}.

\begin{algorithm}
  \caption{Algorithm for the adversary's problem of the \RBCKP{} with discrete uncorrelated uncertainty}
  \label{alg_RBCKP_adversary}

  \Input{a finite set $\E$,
    $a \colon \E \to \Nn$,
    $c \colon \E \to \Q$,
    $\bm, \bp \in \Q$ with $0 \le \bm \le \bp \le a(\E)$,
    a discrete uncorrelated uncertainty set~$\U$ given by a finite set~$\U_e \subseteq \Qgt$ for every~$e \in \E$,
    a feasible leader's solution~$b \in [\bm, \bp]$}

  \Output{an optimal adversary's solution~$d \in \U$ (i.e., a value $d(e) \in \U_e$ for every~$e \in \E$)
    of the \RBCKP{}}

  \If{$b = 0$ or $b = a(\E)$}
  {
    \Return{an arbitrary $d \in \U$}
  }
  
  \For{$e \in \E$}
    {
      $d^-(e) := \min \U_e$, $d^+(e) := \max \U_e$
    }
    
  $\overline{\E} := \emptyset$
  
  \For{$e^* \in \E$}
  {
    \For{$\delta^* \in \U_{e^*}$}
    {
      $\E_{e^*, \delta^*}^- := \{e \in \E \mid d^-(e)/a(e) > \delta^*/a(e^*)\}$

      $\E_{e^*, \delta^*}^0 := \{e \in \E \mid d^-(e)/a(e) \le \delta^*/a(e^*) \le d^+(e)/a(e)\} \setminus \{e^*\}$
      
      \For{$b^* \in \{\max\{a(\E_{e^*, \delta^*}^-), \lfloor b - a(e^*) \rfloor + 1\}, \dots, \min\{a(\E_{e^*, \delta^*}^- \cup \E_{e^*, \delta^*}^0), \lfloor b \rfloor\}\}$ \label{alg_RBCKP_adversary_bs}}
      {
        $Y_{e^*, \delta^*, b^*}^0 := \KP{}(\E_{e^*, \delta^*}^0, a \restriction \E_{e^*, \delta^*}^0, b^* - a(\E_{e^*, \delta^*}^-), -c \restriction \E_{e^*, \delta^*}^0)$
        
        \If{$Y_{e^*, \delta^*, b^*}^0 \neq $ ``infeasible''}
        {
          $Y_{e^*, \delta^*, b^*} := \E_{e^*, \delta^*}^- \cup Y_{e^*, \delta^*, b^*}^0$
          
          $\overline{\E} := \overline{\E} \cup \{(e^*, \delta^*, b^*)\}$ \label{alg_RBCKP_adversary_add_to_olE}
        }
      }
    }
  }

  select $(e^*, \delta^*, b^*) \in \argmin\{c(Y_{e, \delta, \overline{b}}) + c(e)/a(e) \cdot (b - \overline{b}) \mid (e, \delta, \overline{b}) \in \overline{\E}\}$ arbitrarily

  \Return{$d \colon \E \to \Q$ with $d(e^*) := \delta^*$, $d(e) = d^+(e)$ for all $e \in Y_{e^*, \delta^*, b^*}$ and $d(e) := d^-(e)$ for all $e \in \E \setminus (Y_{e^*, \delta^*, b^*} \cup \{e^*\})$}
\end{algorithm}

Theorem~\ref{thm_RCBSP_discrete_uncorr_adversary},
together with the insights explained above,
implies the correctness of Algorithm~\ref{alg_RBCKP_adversary}.
For the running time,
note that
Lines~\ref{alg_RBCKP_adversary_bs} to~\ref{alg_RBCKP_adversary_add_to_olE}
can be implemented using a single dynamic program
because a \KP{} on the same item set
is solved for several capacities.
This shows:

\begin{theorem} \label{thm_RBCKP_adversary}
  For any fixed feasible leader's solution of the \RBCKP{}
  with a discrete uncorrelated uncertainty set~$\U = \prod_{e \in \E} \U_e$,
  Algorithm~\ref{alg_RBCKP_adversary} solves the adversary's problem
  in time~$O(n u A)$,
  where~$n = \lvert \E \rvert$,~$u = \sum_{e \in \E} \lvert \U_e \rvert$, and~$A = a(\E)$.
\end{theorem}

We now turn to the leader's algorithm
and adapt Algorithms~\ref{alg_RCBSP_discrete_uncorr_leader} and~\ref{alg_RBCKP_adversary}
to the leader's problem of the \RBCKP{}.
This results in Algorithm~\ref{alg_RBCKP_leader}.

\begin{algorithm}
  \caption{Algorithm for the leader's problem of the \RBCKP{} with discrete uncorrelated uncertainty}
  \label{alg_RBCKP_leader}

  \Input{a finite set $\E$,
    $a \colon \E \to \Nn$,
    $c \colon \E \to \Q$,
    $\bm, \bp \in \Q$ with $0 \le \bm \le \bp \le a(\E)$,
    a discrete uncorrelated uncertainty set~$\U$ given by a finite set~$\U_e \subseteq \Qgt$ for every~$e \in \E$}

  \Output{an optimal leader's solution~$b \in [\bm, \bp]$ of the \RBCKP{}}

  \For{$e \in \E$}
  {
    $d^-(e) := \min \U_e$, $d^+(e) := \max \U_e$
  }

  $\overline{\E} := \emptyset$
  
  \For{$e^* \in \E$}
  {
    \For{$\delta^* \in \U_{e^*}$}
    {
      $\E_{e^*, \delta^*}^- := \{e \in \E \mid d^-(e)/a(e) > \delta^*/a(e^*)\}$

      $\E_{e^*, \delta^*}^0 := \{e \in \E \mid d^-(e)/a(e) \le \delta^*/a(e^*) \le d^+(e)/a(e)\} \setminus \{e^*\}$

      \For{$b^* \in \{a(\E_{e^*, \delta^*}^-), \dots, a(\E_{e^*, \delta^*}^- \cup \E_{e^*, \delta^*}^0)\}$}
      {
        $Y_{e^*, \delta^*, b^*}^0 := \KP{}(\E_{e^*, \delta^*}^0, a \restriction \E_{e^*, \delta^*}^0, b^* - a(\E_{e^*, \delta^*}^-), -c \restriction \E_{e^*, \delta^*}^0)$
        
        \If{$Y_{e^*, \delta^*, b^*}^0 \neq $ ``infeasible''}
        {

          $Y_{e^*, \delta^*, b^*} := \E_{e^*, \delta^*}^- \cup Y_{e^*, \delta^*, b^*}^0$

          $g_{e^*, \delta^*, b^*} :=$ linear piece from $(b^*, c(Y_{e^*, \delta^*, b^*}))$ to $(b^* + a(e^*), c(Y_{e^*, \delta^*, b^*}) + c(e^*))$
          
          $\overline{\E} := \overline{\E} \cup \{(e^*, \delta^*, b^*)\}$
        }
      }
    }
    
  }

  compute $g := \min\{g_{e^*, \delta^*, b^*} \mid (e^*, \delta^*, b^*) \in \overline{\E}\}$ (i.e., the pointwise minimum) \label{alg_RBCKP_leader_pw_min} 

  compute some $b \in \argmax\{g(b') \mid b' \in [\bm, \bp]\}$

  \Return{$b$}
\end{algorithm}

In Algorithm~\ref{alg_RBCKP_leader},
analogously to Algorithm~\ref{alg_RBCKP_adversary},
we enumerate all possible fractional items~$e^* \in \E$,
the values~$\delta^* \in \U_{e^*}$ assigned to them by the adversary,
and all possible total sizes~$b^*$ of the items
the follower selects before the fractional item~$e^*$.
The feasible values for~$b^*$ can be derived from
the total sizes of the sets~$\E_{e^*, \delta^*}^-$ and~$\E_{e^*, \delta^*}^0$,
like in Algorithm~\ref{alg_RBCKP_adversary},
but without any fixed capacity~$b$,
as the leader's decision is not fixed yet.
If the subroutine returns a feasible solution,
we construct a linear piece corresponding to the leader's values
of the follower's solutions consisting of
the binary prefix and any fraction~$y_{e^*} \in [0, 1]$
of item~$e^*$,
similarly to Algorithm~\ref{alg_RCBSP_discrete_uncorr_leader}.
Note that the width of the linear piece is now given by the size~$a(e^*)$,
in contrast to the setting of Algorithm~\ref{alg_RCBSP_discrete_uncorr_leader}
where all items have size~$1$.
Accordingly,
we cannot compute the pointwise minimum
for every $b^*$ separately
like in Algorithm~\ref{alg_RCBSP_discrete_uncorr_leader},
but we collect all linear pieces
and compute the pointwise minimum globally in Line~\ref{alg_RBCKP_leader_pw_min}
of Algorithm~\ref{alg_RBCKP_leader}.

Using the same arguments as in~\cite{buchheimhenke2022}
and Section~\ref{sec_RCBSP_discrete_uncorr},
one can see that
the function~$g$
computed by Algorithm~\ref{alg_RBCKP_leader}
is exactly the leader's objective function
for the \RBCKP{}.
The algorithm's running time
can be estimated similarly to Theorems~\ref{thm_RCBSP_discrete_uncorr_leader} and~\ref{thm_RBCKP_adversary}.
Note that
the term~$a(\E)$
enters the running time
due to the subroutine solving a binary \KP{}
as well as due to the enumeration of~$b^*$.
Together,
we get:

\begin{theorem} \label{thm_RBCKP_leader}
  Algorithm~\ref{alg_RBCKP_leader} solves the \RBCKP{}
  with a discrete uncorrelated uncertainty set~$\U = \prod_{e \in \E} \U_e$
  in time~$O(u A (n + \log(u A)))$,
  where $n = \lvert \E \rvert$, $u = \sum_{e \in \E} \lvert \U_e \rvert$, and $A = a(\E)$.
\end{theorem}

In summary,
we complemented the weak NP-hardness results from~\cite{buchheimhenke2022}
by pseudopolynomial-time algorithms
for both the adversary's problem and the leader's problem
of the \RBCKP{} with discrete uncorrelated uncertainty.

\printbibliography

\end{document}